\documentclass[%
a4paper,							% paper format
11pt,								% Schriftgröße (12pt, 11pt (Standard))
bibliography=totoc,						% Literaturverzeichnis im Inhalt
abstracton,					% Überschrift über der Zusammenfassung an	
]
{scrartcl}

\usepackage[a4paper,left=3.0cm,right=2.5cm, top=2.5cm, bottom=3.0cm]{geometry}
\usepackage[headsepline=.4pt,footsepline=.4pt,automark,autooneside=false,]{scrlayer-scrpage}
\clearscrheadfoot 
\automark[subsection]{section} 
\automark*[subsubsection]{subsection}
\ihead{\scriptsize\rightmark} 
\usepackage[ngerman, english]{babel}
\usepackage[T1]{fontenc}
\usepackage[utf8]{inputenc}
\usepackage{lmodern}
\usepackage[onehalfspacing]{setspace} %1,5 spacing

\usepackage{xspace}
\usepackage{calc}%\widthof{}

\usepackage{amsmath}
\usepackage{amssymb}
\usepackage{amsthm,thmtools}
\usepackage[mathscr]{eucal}

\usepackage{mathtools}
\usepackage{bbm}
\begingroup
    \makeatletter
    \@for\theoremstyle:=definition,remark,plain\do{%
        \expandafter\g@addto@macro\csname th@\theoremstyle\endcsname{%
            \addtolength\thm@preskip\parskip
            }%
        }
\endgroup

\usepackage{chngcntr}%reset equation counter after each section
\counterwithin*{equation}{section}

\theoremstyle{plain}
\newtheorem{theorem}{Theorem}[section]
\newtheorem{lemma}[theorem]{Lemma}
\theoremstyle{definition}

\theoremstyle{remark}
\newtheorem{remark}[theorem]{Remark}
\theoremstyle{plain}
\newtheorem{proposition}[theorem]{Proposition}
\theoremstyle{plain}

\theoremstyle{remark}
\newtheorem{example}[theorem]{Example}
\usepackage{tabularx}
\usepackage{booktabs}
\usepackage{multirow}
\usepackage{multicol}
\usepackage{pdflscape}

\usepackage[flushleft]{paralist}
\usepackage{color}
\usepackage{graphicx}
\usepackage{fancyvrb}
\usepackage{float}
\usepackage{placeins}
\usepackage{csquotes}
\usepackage{titleref}
\usepackage{hyperref}
\hypersetup{
    colorlinks,
		linkcolor={red},
    citecolor={blue},
    urlcolor={blue}
}
\mathtoolsset{showonlyrefs}% cant be used with cleverref!!!
\SaveVerb{euler}=euler=
\SaveVerb{exact}=exact=
\SaveVerb{m1}=m1=
\SaveVerb{m2}=m2=
\SaveVerb{m3}=m3=
\newcommand{\Mlog}{Y}%Magnus logarithm
\newcommand \s {{\sigma}}

\newcommand \R {\mathbb{R}}

\DeclarePairedDelimiter\floor{\lfloor}{\rfloor}

\newcommand\Eb{\mathbb{E}}

\newcommand\Lc{\mathscr{L}}
\newcommand\Qc{\mathscr{Q}}

\newcommand{\Matrix}{\Sigma}
\newcommand{\eps}{\varepsilon}
\newcommand{\e}{\varepsilon}

\newcommand\abf{{\bf a}}
\newcommand\bbf{{\bf b}}
\newcommand\cbf{{\bf c}}

\newcommand\gbf{{\bf g}}
\newcommand\Lbf{{\bf L}}
\newcommand\Gbf{{\bf G}}
\newcommand\sbf{{\mathbf{ \s}}}

\newcommand{\norm}[1]{\left\|{#1}\right\|}
\newcommand{\dd}{d}
\newcommand\dD{\Lc}
\newcommand\dDD{\Qc}

\renewcommand \t {\tau}

\renewcommand \O {\Omega}

\renewcommand \o {\omega}
\renewcommand \d {\delta}

\renewcommand \s {\sigma}

\newcommand{\N}{\mathbb{N}}

\renewcommand{\phi}{\varphi}
\renewcommand{\tilde}{\widetilde}

\renewcommand\l {\lambda}

\newcommand \F {\mathcal{F}}

\newcommand{\Md}{\mathcal{M}^{d\times d}}
\newcommand{\ad}{\text{ad}}

\newcommand{\comm}[2]{\left[#1,#2\right]}
\newcommand{\bern}{\beta}

\newcommand{\abs}[1]{\left|#1\right|}
\newcommand{\fnorm}[1]{\|{#1}\|_{F}}
\newcommand{\fnormp}[1]{\|{#1}\|_{F}^{2p}}
\newcommand{\snorm}[1]{\left\|{#1}\right\|}
\newcommand{\enorm}[1]{\left|{#1}\right|}
\newcommand{\mainfolder}{Figures}
\newcommand{\BfixCDSeuler}{
    \includegraphics[width=\textwidth]{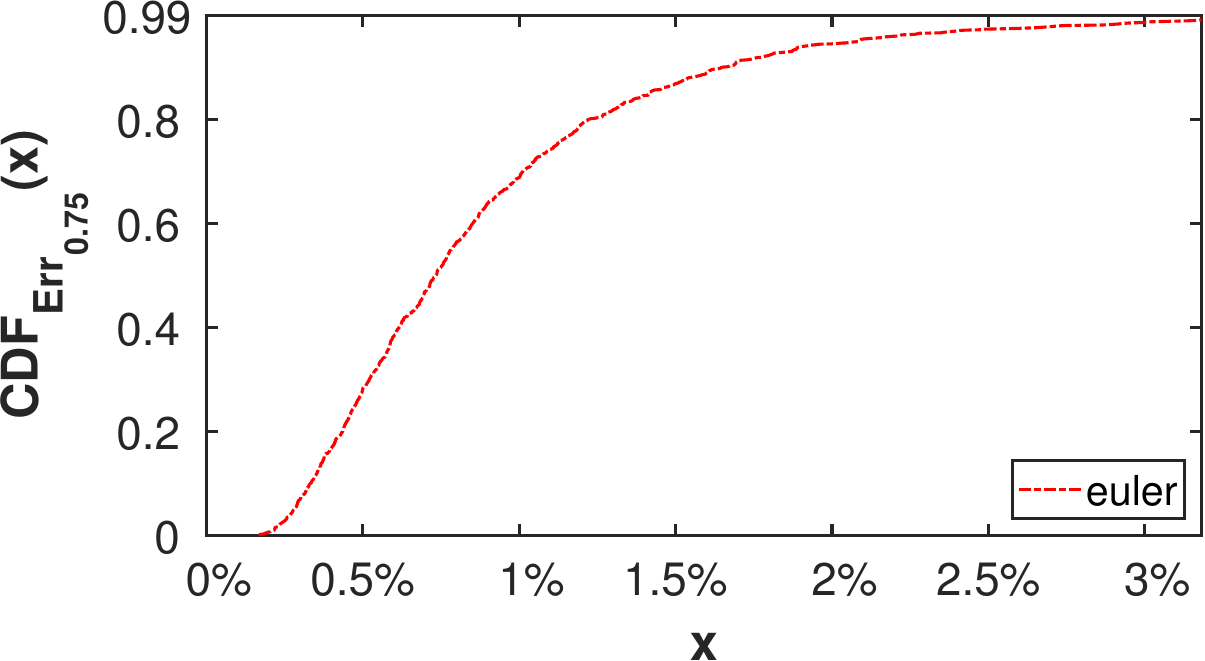}
}
\newcommand{\BfixCDSmone}{
    \includegraphics[width=\textwidth]{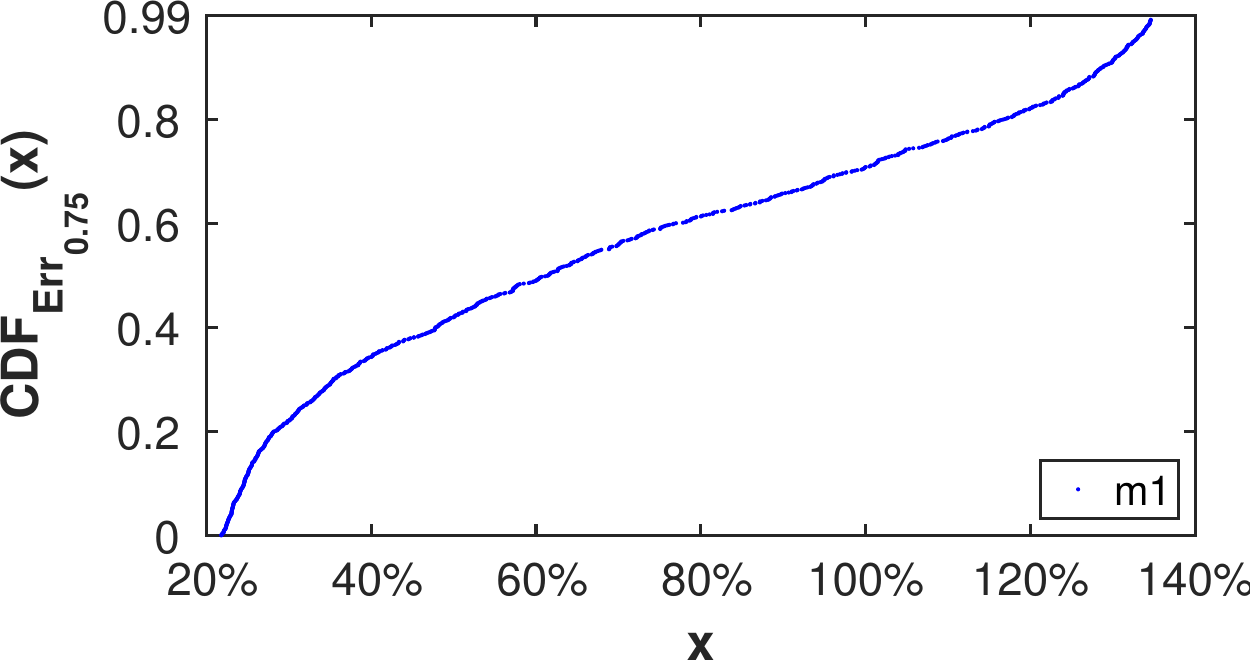}
}
\newcommand{\BfixCDSmtwo}{
    \includegraphics[width=\textwidth]{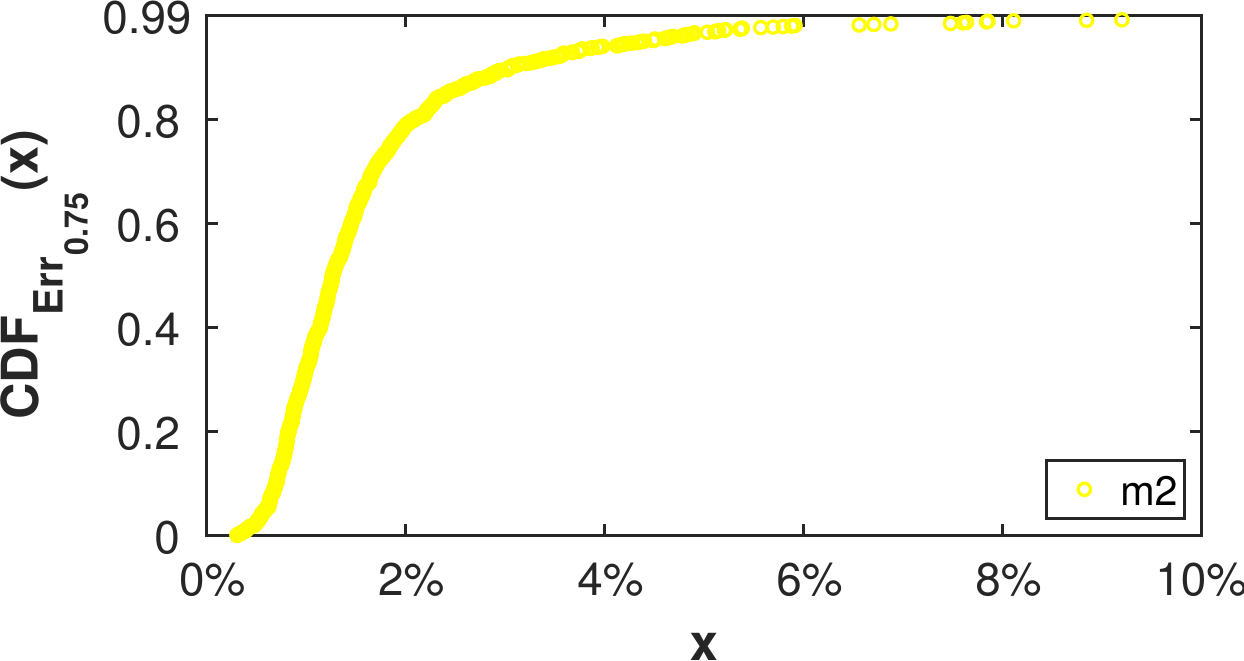}
}
\newcommand{\BfixCDSmthree}{
    \includegraphics[width=\textwidth]{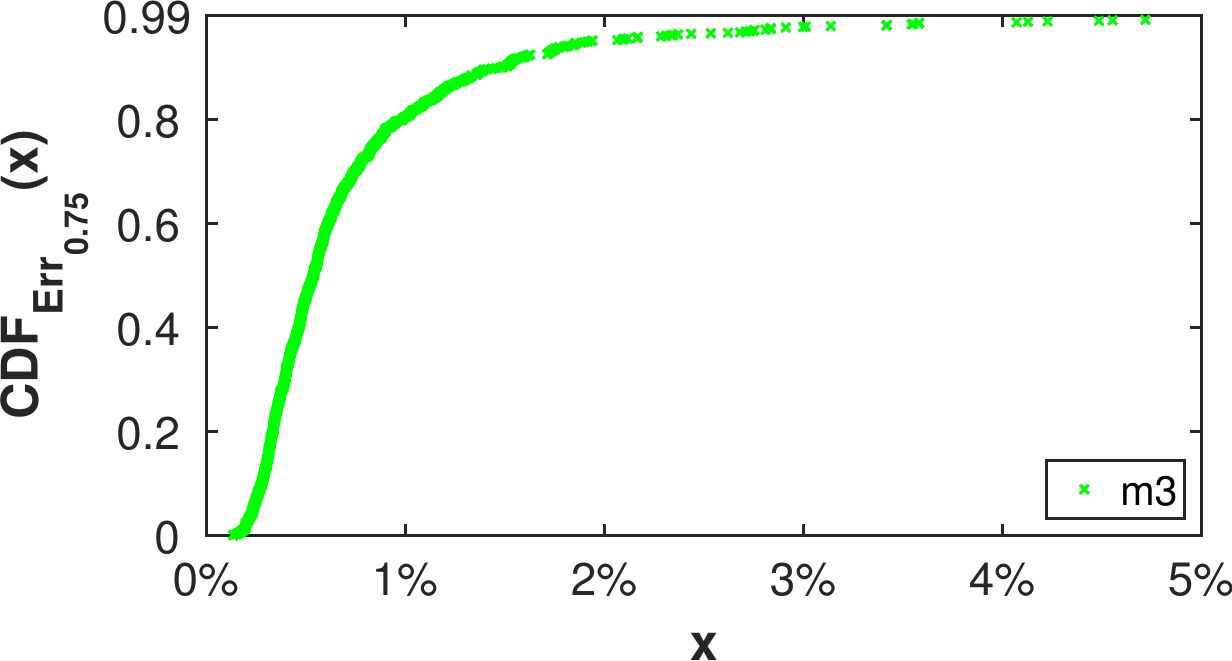}
}
\newcommand{\BvarTrajectory}{
        \includegraphics[width=\columnwidth]{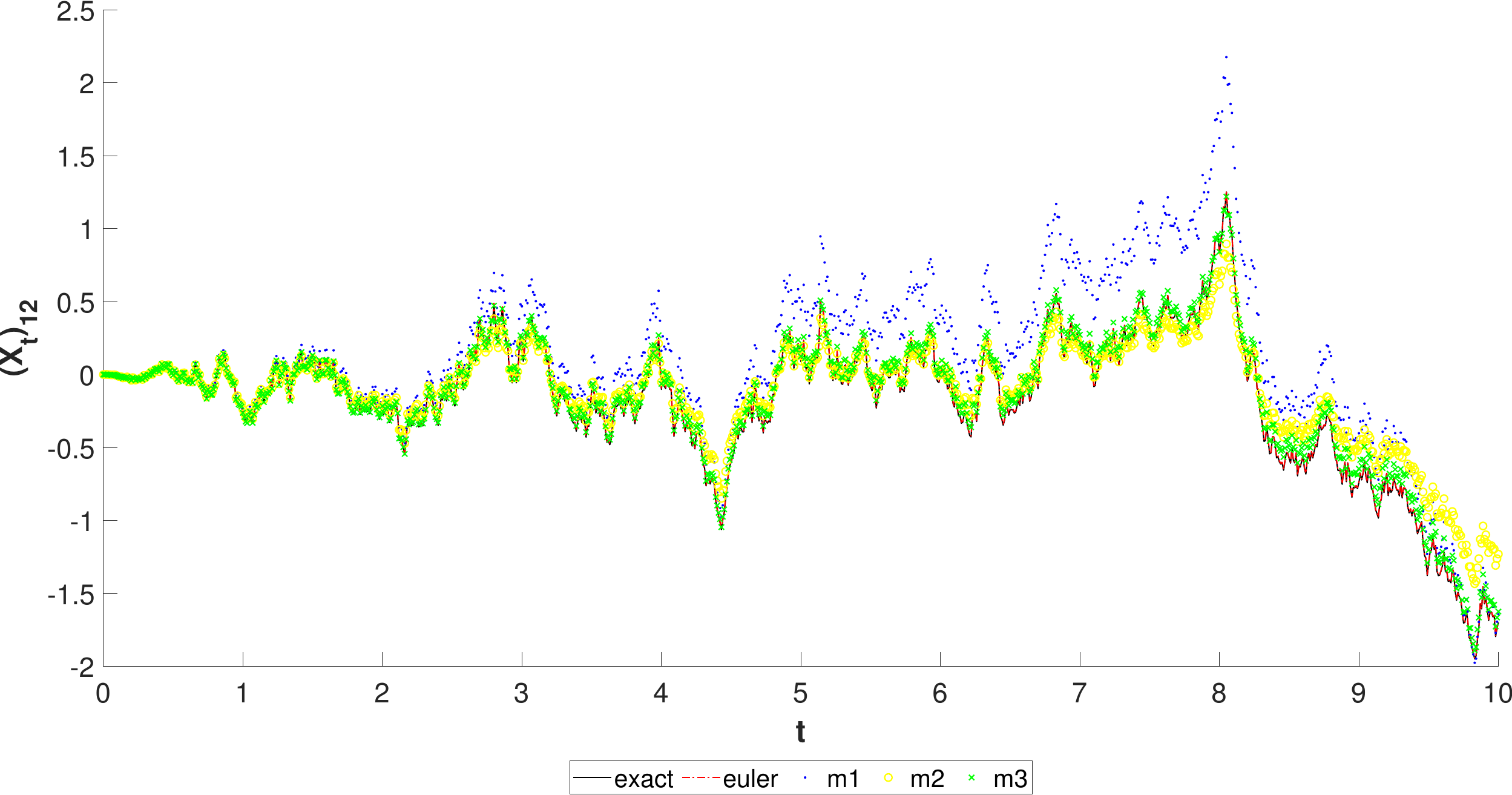}
}
\newcommand{\BvarCDSeuler}{
    \includegraphics[width=\textwidth]{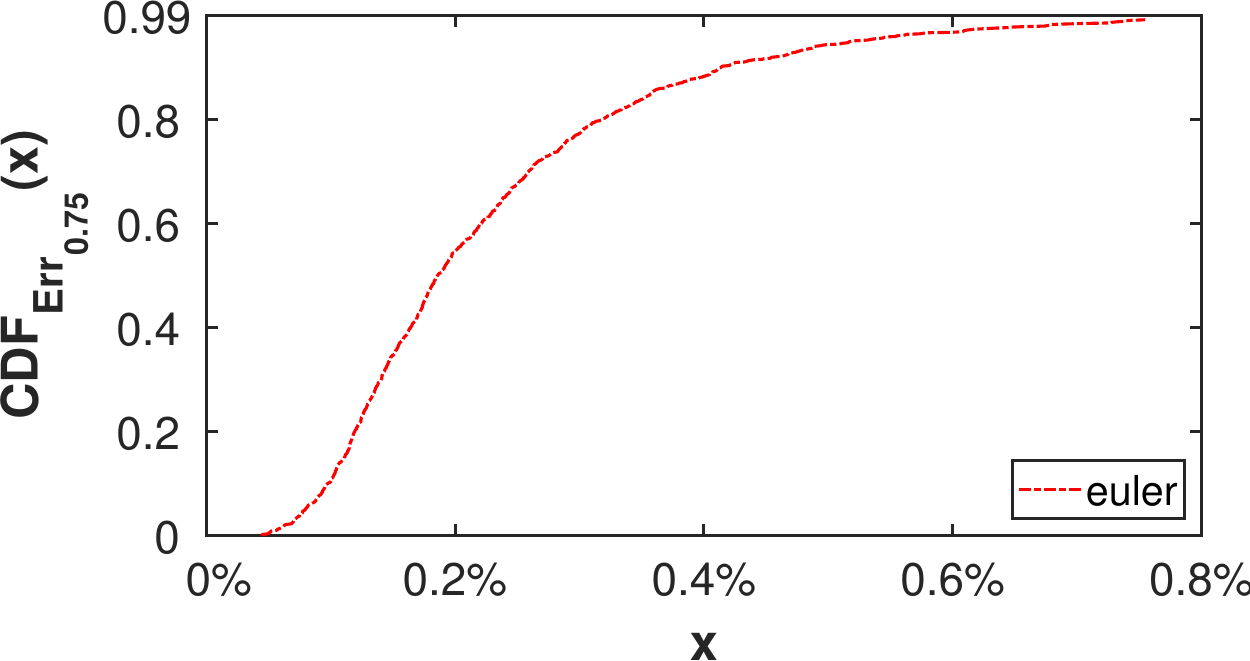}
}
\newcommand{\BvarCDSmone}{
    \includegraphics[width=\textwidth]{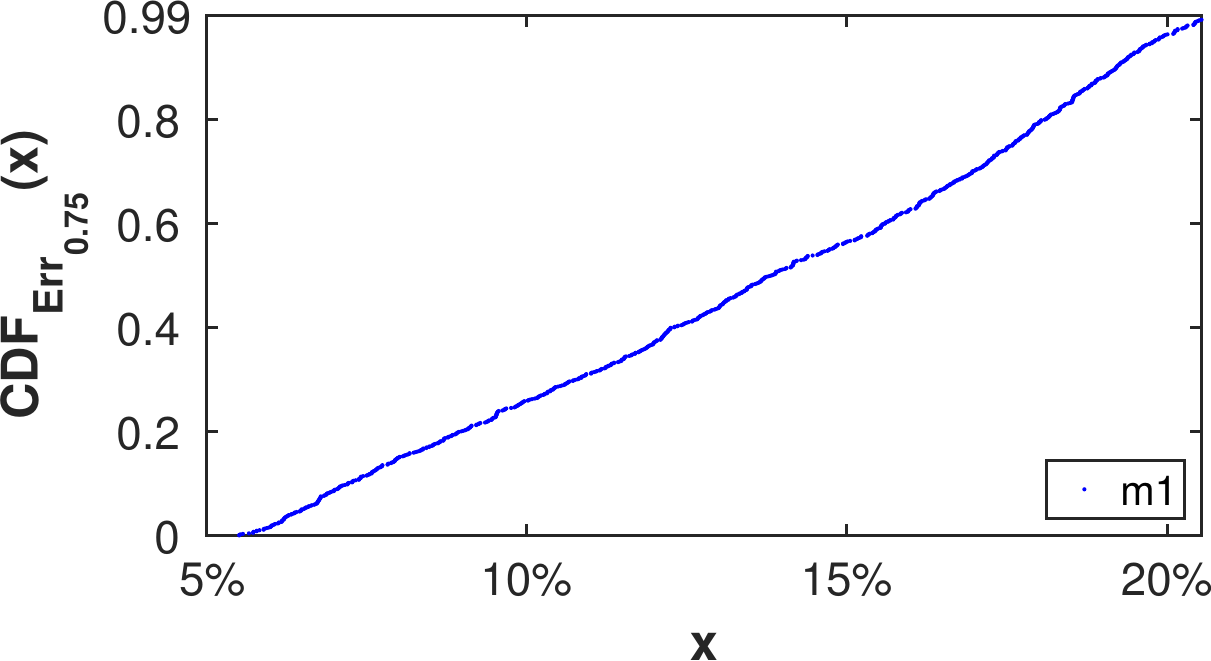}
}
\newcommand{\BvarCDSmtwo}{
    \includegraphics[width=\textwidth]{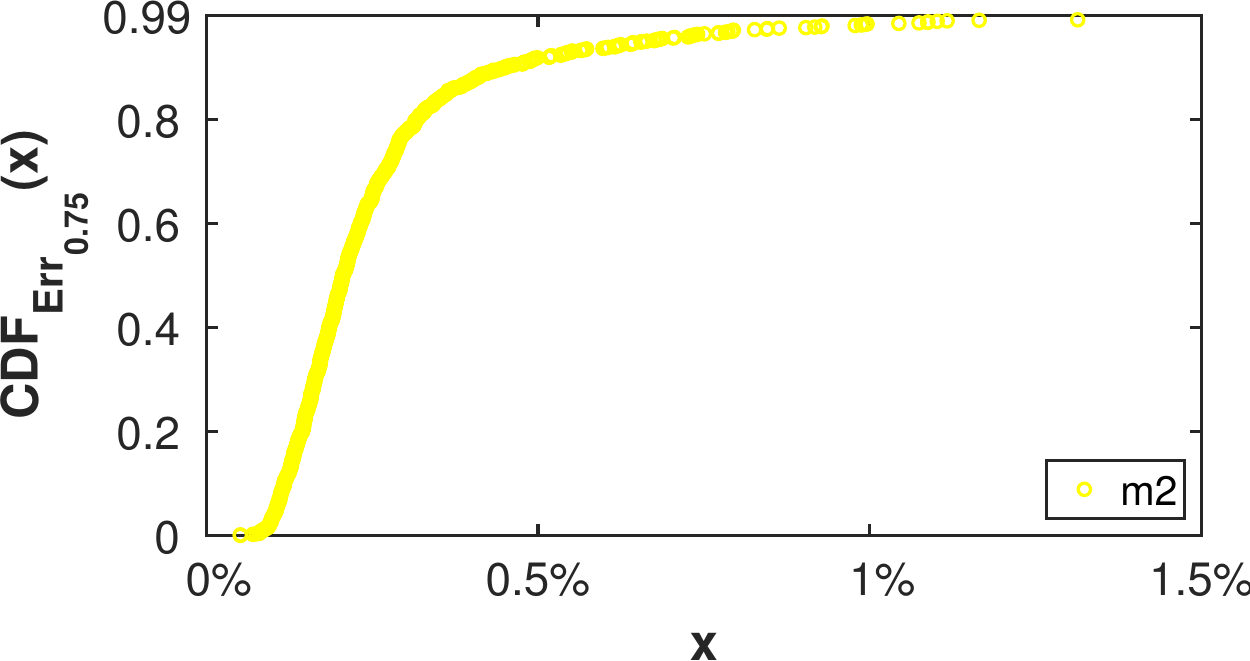}
}
\newcommand{\BvarCDSmthree}{
    \includegraphics[width=\textwidth]{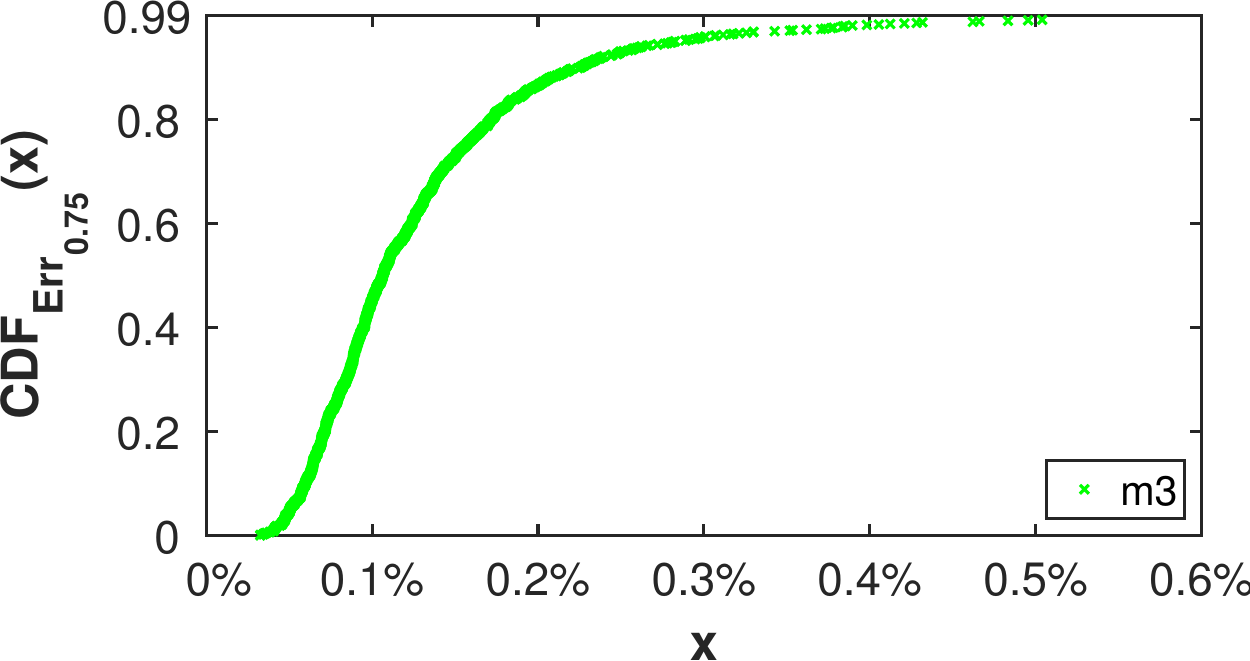}
}
\newcommand{\ABconstTrajectory}{
        \includegraphics[width=\columnwidth]{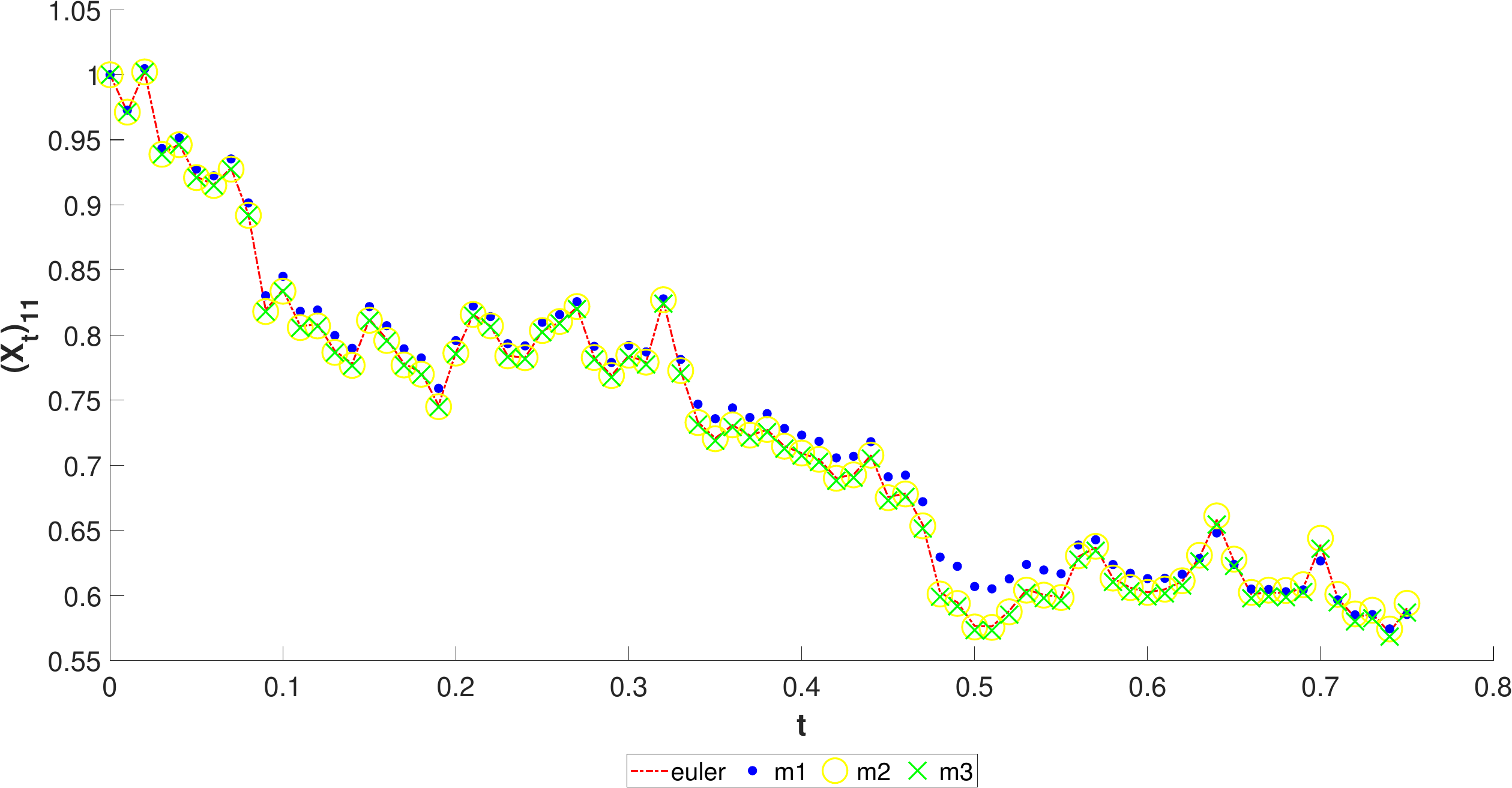}
}
\newcommand{\ABconstCDSmone}{
    \includegraphics[width=\textwidth]{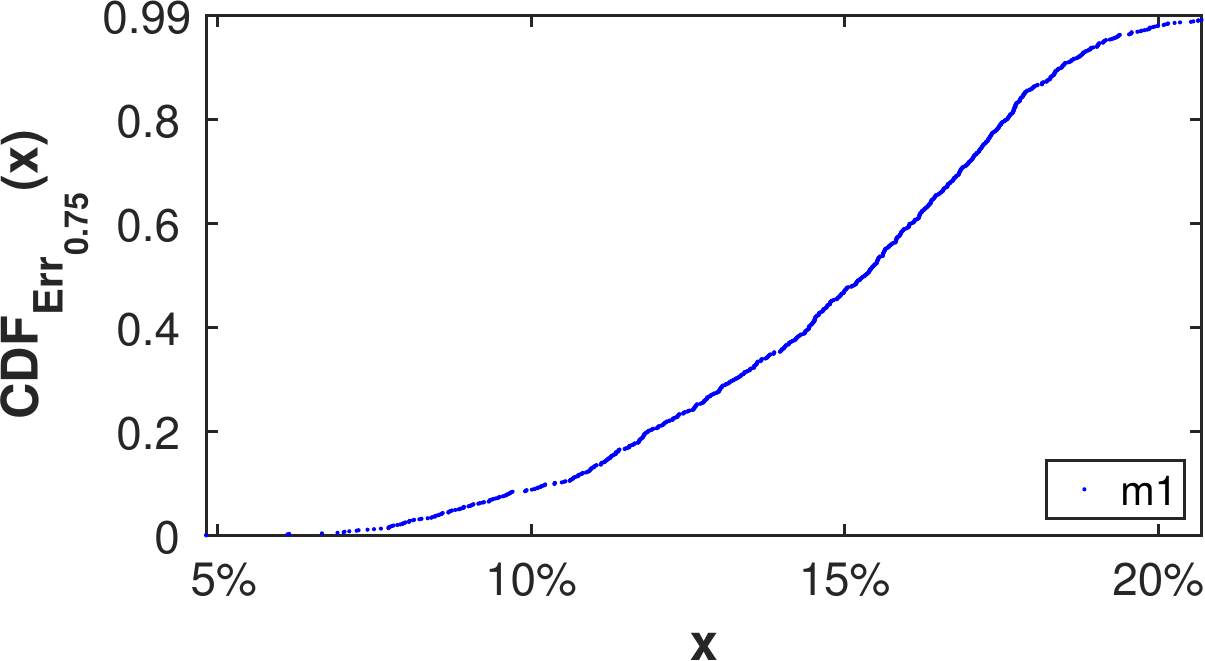}
}
\newcommand{\ABconstCDSmtwo}{
    \includegraphics[width=\textwidth]{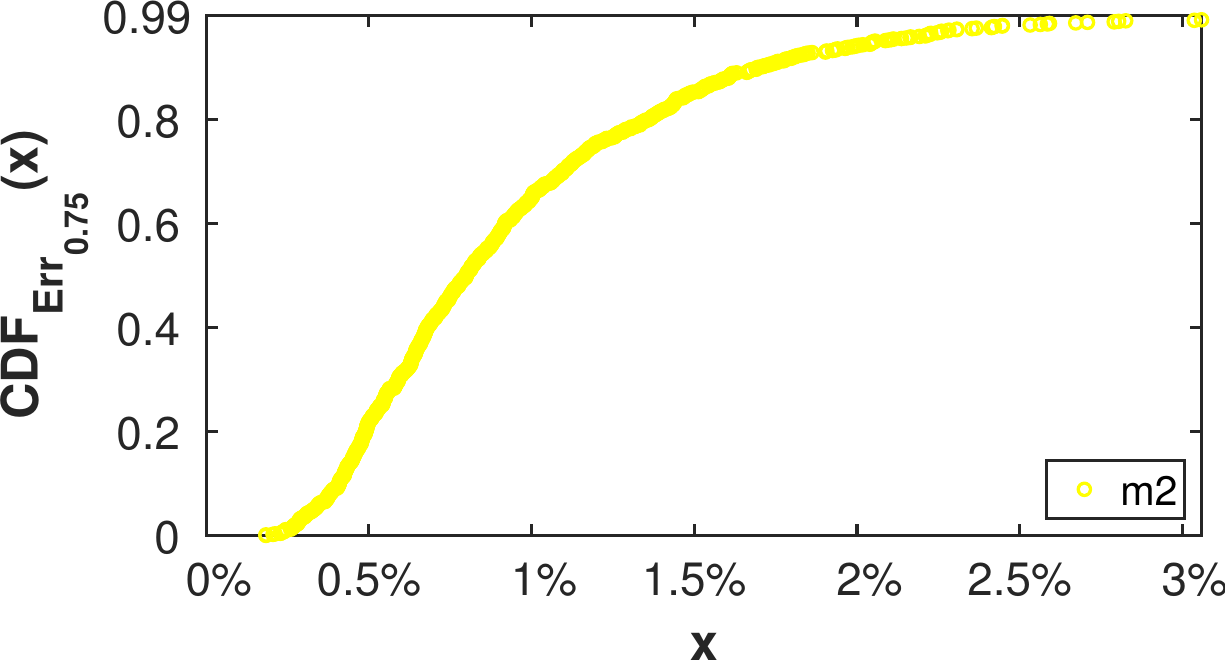}
}
\newcommand{\ABconstCDSmthree}{
    \includegraphics[width=\textwidth]{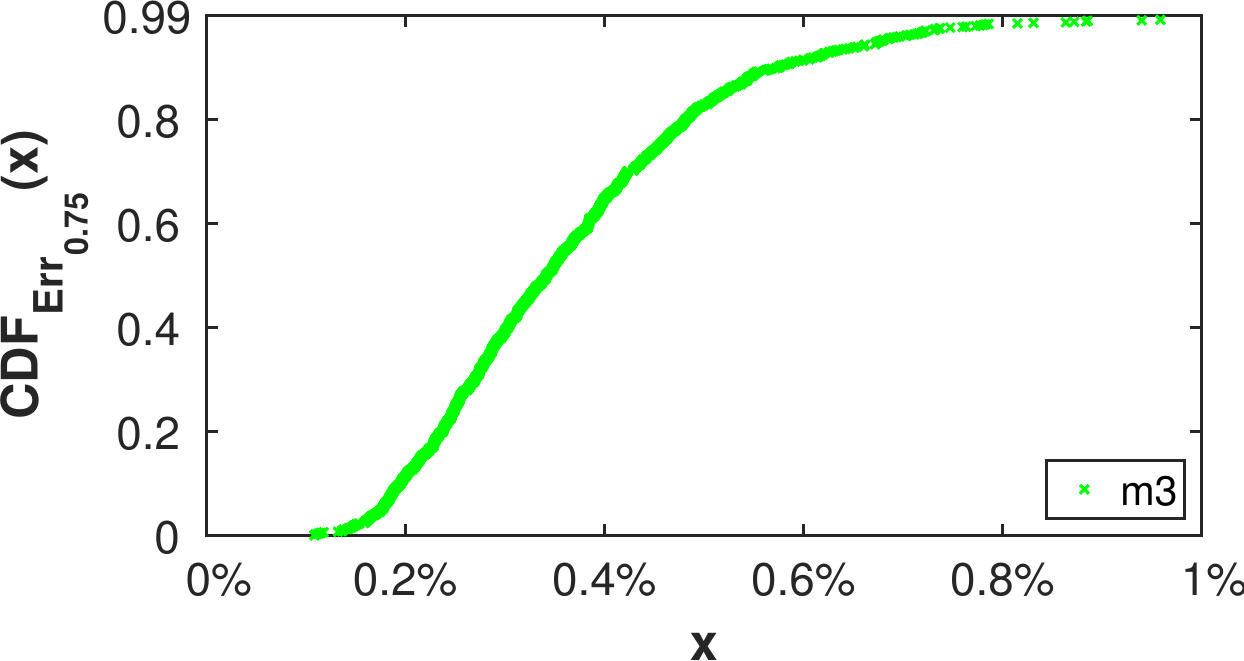}
}
\begin{document}
\newcommand{\headlogo}{
\vspace{-5\headheight}
\raisebox{2pt}{
\includegraphics[width=.1\textwidth]{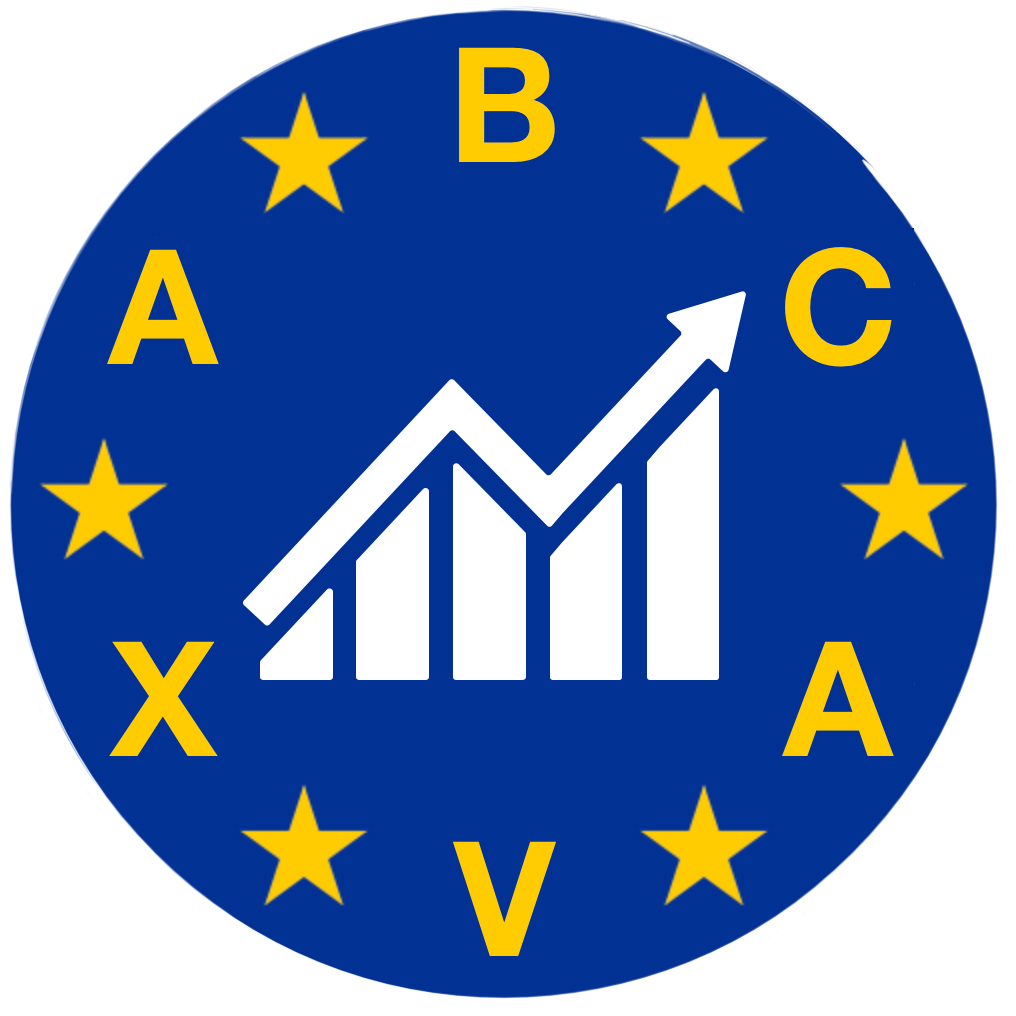}
}
\hfill
\mbox{
\includegraphics[trim=5.1cm 11cm 5.1cm 12cm,clip, width=.2\textwidth]{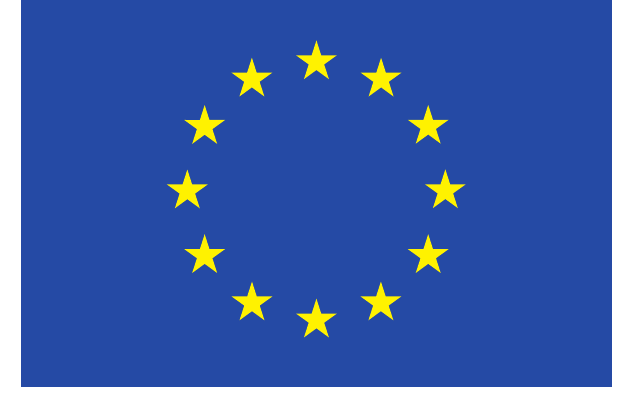}
} \hspace{-1.5em}}

%\title{Magnus expansion for matrix-valued It\^o SDEs, and application to parabolic SPDEs}
\title{{\headlogo}\\
{On the stochastic Magnus expansion and its application to SPDEs}}% and applications to SPDEs}

\author{Kevin Kamm\thanks{Dipartimento di Matematica, Universit\`a di Bologna, Bologna, Italy.
\textbf{e-mail}: kevin.kamm@unibo.it} \and Stefano Pagliarani\thanks{Dipartimento di Matematica,
Universit\`a di Bologna, Bologna, Italy. \textbf{e-mail}: stefano.pagliarani9@unibo.it} \and
Andrea Pascucci\thanks{Dipartimento di Matematica, Universit\`a di Bologna, Bologna, Italy.
\textbf{e-mail}: andrea.pascucci@unibo.it}}

\date{This version: \today}
\maketitle
\begin{abstract}
We derive the stochastic version of the Magnus expansion for linear systems of
stochastic differential equations (SDEs). The main novelty with respect to the related literature
is that we consider SDEs in the It\^o sense, with progressively measurable coefficients, for which
an explicit It\^o-Stratonovich conversion is not available. We prove convergence of the Magnus
expansion up to a stopping time $\t$ and provide a novel asymptotic estimate of the cumulative
distribution function of $\t$. As an application, we propose a new method for the numerical
solution of stochastic partial differential equations (SPDEs) based on spatial discretization and
application of the stochastic Magnus expansion. A notable feature of the method is that it is
fully parallelizable. We also present numerical tests in order to asses the accuracy of the
numerical schemes.
\end{abstract}

\noindent
\textbf{Acknowledgements}: This project has received funding from the European Union’s Horizon 2020 research and innovation
programme under the Marie Sklodowska-Curie grant agreement No 813261 and is part of the ABC-EU-XVA project.
\section{Introduction}
\label{sec:intro} The Magnus expansion (hereafter referred to as ME) is a classical tool to solve
non-autonomous linear differential equations. Generalizations of the ME to {\it Stratonovich SDEs}
are well-known and were proposed by several authors (see for instance \cite{BenArous},
\cite{MR2494199}, \cite{burrage1999high}, \cite{Wang} and the references given in Section
\ref{refer}). In this paper we derive, for the first time to the best of our knowledge, the ME for
{\it It\^o SDEs} under general assumptions which do not guarantee an explicit It\^o-Stratonovich
conversion, namely progressively measurable stochastic coefficients. Our main results are the
convergence of the stochastic ME up to a stopping time $\t$ and a novel asymptotic estimate of the
cumulative distribution function of $\t$. The latter improves some previous estimates obtained in
purely Markovian settings and is based on an application of Morrey's inequality. We also explore
possible applications to the numerical solution of stochastic partial differential equations
(SPDEs).

%We derive a Magnus-type expansion for
Let $d,q\in\N$ and consider the linear matrix-valued It\^o SDE
\begin{equation}\label{eq:SDE_linear_b}
\begin{cases}
 \dd X_t = B_t X_t dt + %\sum_{j=1}^q
 A^{(j)}_t X_t \dd W^j_t,\\ % \qquad t>0 \\
 X_0 = I_d,
\end{cases}
\end{equation}
with $A^{(1)},\dots,A^{(q)},B$ being real $(d\times d)$-matrix-valued bounded {stochastic}
processes, $I_d$ the identity $(d\times d)$-matrix and $W=(W^1,\dots,W^q)$ a $q$-dimensional
standard Brownian motion. In \eqref{eq:SDE_linear_b}, as well as anywhere throughout the paper,
{we use Einstein summation convention to imply summation of terms containing $W^{j}$, over the
index $j$ from $1$ to $q$.}
%\red{\bf [Lo vogliamo proprio mettere?] }{\blue The solution to
%\eqref{eq:SDE_linear_b} can be seen as the fundamental solution of the analogous vector-valued
%linear SDE, in that the process defined by $Z_t = \red{X_{t}} z$ with $z$ column vector of $\R^d$
%is the solution to
%\begin{equation}
%\begin{cases}
% \dd Z_t = B_t \red{Z_t} dt + %\sum_{j=1}^q
% A^{(j)}_t Z_t \dd W^j_t,\\ % \qquad t>0 \\
% Z_0 = z.
%\end{cases}
%\end{equation}
%}

In the deterministic case, i.e. $A^{(j)}\equiv 0$, $j=1,\dots,q$, \eqref{eq:SDE_linear_b} reduces
to the matrix-valued ODE
\begin{equation}\label{ode}
 \begin{cases}
  \frac{d}{dt}X_t =B_t X_t,\\ %   , \qquad t>0 \\
  X_0 = I_d,
 \end{cases}
\end{equation}
which admits an explicit solution, in terms of matrix exponential, in the time-homoge-neous case. Namely, if $B_t\equiv B$, the unique solution to \eqref{ode} reads as
\begin{equation}
X_t = e^{tB}, \qquad t\geq 0.
\end{equation}
However, in the non-autonomous case, the ODE \eqref{ode} does not admit an explicit solution. In
particular, if $B_t$ is not constant, the solution $X_t$ typically differs from $e^{\int_0^t B_s
\dd s}$. This is due to the fact that, in general, $B_t$ and $B_{s}$ do not commute for $t\neq s$.
As it turns out, a representation of the solution in terms of {a} matrix exponential is still
possible, at least for short times, i.e.
\begin{equation}\label{eq:exponential}
X_t = e^{Y_t}, %\qquad \text{}.
\end{equation}
for $t\geq 0$ suitably small and $Y_t$ real valued $(d\times d)$-matrix.
Moreover, $Y$ %the matrix-valued function at the exponent
admits a semi-explicit expansion as a series of iterated integrals involving nested Lie
commutators of the function $B$ at different times. Such representation is known as \emph{Magnus
expansion} (\cite{MR67873}) and its first terms read as
\begin{align}\label{eq:maguns_det}
Y_t = \int_0^t B_s \dd s &+ \frac{1}{2} \int_0^t  \dd s \int_0^s [B_s,B_u] \dd u \\&+ \frac{1}{6} \int_0^t  \dd s \int_0^s  \dd u \int_0^u \big(  \big[B_s,[B_u,B_r]\big]
+\big[B_r,[B_u,B_s]\big] \big) \dd r  + \cdots ,
\end{align}
{where $\comm{A}{B}\coloneqq AB-BA$ denotes the Lie commutator.}
{The ME has a wide range of physical applications and the related literature has grown
increasingly over the last decades (see, for instance, the excellent survey paper \cite{MR2494199}
and the references given therein).}

In the stochastic case, when $j=1$, $B_t\equiv 0$ and $A$ is constant, i.e. $A_t(\omega)\equiv A$, {the}
It\^o equation \eqref{eq:SDE_linear_b} reduces to
\begin{equation}
\begin{cases}
 \dd X_t = A X_t \dd W_t,\\ % \qquad t>0 \\
 X_0 = I_d,
\end{cases}
\end{equation}
whose explicit solution can be easily proved to be of the form \eqref{eq:exponential}, with
\begin{equation}
 Y_t = -\frac{1}{2} A^2 t + A W_t , \qquad t\geq 0.
\end{equation}
%when either $j>1$, $B\neq 0$, or $A$ is not constant,
In general, when the matrices $A^{(j)}_t, A^{(j)}_{s},B_t,B_{s}$ with $t\neq s$ do not commute, an
explicit solution to \eqref{eq:SDE_linear_b} is not known. For instance, {in the non-commutative case,} neither the equation
\begin{equation}\label{eq:example1}
\begin{cases}
 \dd X_t =  B \dd t + A X_t \dd W_t,\\ % \qquad t>0 \\
 X_0 = I_d,
\end{cases}
\end{equation}
nor the equation
\begin{equation}\label{eq:example2}
\begin{cases}
 \dd X_t = A_t X_t \dd W_t,\\ % \qquad t>0 \\
 X_0 = I_d,
\end{cases}
\end{equation}
admit an explicit solution{, save some particular cases (see for instance the example in Section \ref{subsec:Auppertriang}). %Therefore, the problem of finding approximate solutions to \eqref{}, including stochastic versions of the Magnus expansion \eqref{eq:exponential}-\eqref{eq:maguns_det}.
Among the approximation tools that were developed in the literature to solve stochastic
differential equations, including \eqref{eq:SDE_linear_b}, some Magnus-type expansions that extend
\eqref{eq:exponential}-\eqref{eq:maguns_det} were derived in different contexts. We now go on to
describe our contribution to this stream of literature, and then to firm our results within the
existing ones. In particular, a detailed comparison with existing stochastic MEs previously
derived by other authors will be provided below, in the last subsection.}

\subsection{Description of the main results.} In this paper we derive a {Magnus}-type
representation formula for the solution to {the It\^o SDE} \eqref{eq:SDE_linear_b}, which is
\eqref{eq:exponential} together with
\begin{equation}\label{eq:magnusY}
Y_t = Y^{(1)}_t + Y^{(2)}_t + Y^{(3)}_t + \cdots
%\sum_{n=1}^{\infty} Y^{(n)}_t
\qquad t\in [0,\tau],
\end{equation}
for $\t$ suitably small, strictly positive stopping time. {In analogy with the deterministic ME,
the general term $Y^{(n)}$ can be expressed recursively, and contains iterated stochastic
integrals of nested Lie commutators of the processes $B,A^{(j)}$ at different times.}

In the case $j=1$, the first two terms of {the} expansion read as
\begin{align}
Y^{(1)}_t &=  \int_0^t
B_s
\dd s  +  \int_0^t  A_s  \dd W_s ,  \\
Y^{(2)}_t &= \frac{1}{2}
                \int_{0}^{t}{ \bigg(  \Big[  B_s, \int_{0}^s B_u \dd u + \int_{0}^s A_u \dd W_u \Big] -  A^2_s
                \bigg)  \dd s}
\\&\quad+
\frac{1}{2}
                \int_{0}^{t}{ \Big[  A_s, \int_{0}^s B_u \dd u + \int_{0}^s A_u \dd W_u \Big]  \dd W_s} .
\end{align}
For example, in the case of SDE \eqref{eq:example1} the latter can be reduced to
\begin{align}
Y^{(1)}_t &=
Bt+AW_t  ,  \\
Y^{(2)}_t &=  [{A},{B}]\left(
                    \frac{1}{2}tW_t-
                    \int_{0}^{t}{W_s ds} \right)
                    -\frac{1}{2}A^2 t                  .
\end{align}
Notice that the last expressions do not contain stochastic integrals. In fact, in the general
autonomous case, {and if $j=1$}, {all the iterated} stochastic integrals in $Y^{(n)}$ can be solved for any $n$ {(see Corollary 5.2.4 in \cite{KloedenPlaten})}. Therefore, in this case the
expansion becomes numerically computable by only approximating Lebesgue integrals, as opposed to
 {stochastic Runge-Kutta} schemes, which typically
require the numerical approximation of stochastic integrals. As we shall see in the numerical
tests in Section \ref{sec:num_tests}, this feature allows {us} to choose a sparser time-grid in order
to save computation time. This feature is also preserved in some non-autonomous cases as
illustrated in Section \ref{sec:num_tests}.

A notable feature of the expansion is the possibility of parallelizing the computation of its
terms. In contrast to standard iterative methods, which require the solution at a given time-step
in order to go through the next step in the iteration, the discretization of the integrals in the
terms $Y^{(n)}$ can be done simultaneously for all the time steps. Conclusively, this entails the
possibility of parallelizing {\it over all times} in the time-grid and makes the numerical
implementation of the stochastic ME perfectly GPU applicable.

As it often happens when deriving convergent (either asymptotically or absolutely) expansions, a
formal derivation precedes the rigorous result: {that is what we do for }
\eqref{eq:exponential}-\eqref{eq:magnusY} in Section \ref{sec:formal_expansion}. Just like the
derivation of the deterministic ME relies on the possibility of writing the logarithm $Y$ as the
solution to an ODE, in the stochastic case the first step consists in representing $Y$ as the
solution to an SDE. Such representation of $Y$ will be more involved compared to the deterministic
case {because of} the presence of the second order derivatives of the exponential map {coming
from} the application of It\^o{'s} formula. {This is a distinctive feature of our derivation with
respect to other analogous results in the Stratonovich setting where the standard chain-rule
applies.} With the SDE representation for $Y$ at hand, the expansion \eqref{eq:magnusY} stems,
like in the deterministic case, from applying a Dyson-type perturbation procedure to the SDE
solved by $Y$.

In the deterministic case, the convergence of the ME \eqref{eq:maguns_det} to the exact logarithm
of the solution to \eqref{ode} was studied by several authors, who proved progressively sharper
lower bounds on the maximum $\bar{t}$ such that the convergence {to the exact solution} is assured
for any $t\in [0,\bar{t}]$. At the best of our knowledge, the sharpest estimate was given in
\cite{MR2413145}, namely
\begin{equation}\label{eq:converg_cond_determ}
\bar{t} \geq \sup \bigg\{ t\geq 0 \mid \int_{0}^t  \snorm{ B_s }  \dd s < \pi  \bigg\},
\end{equation}
{where $\snorm{ B_s}$ denotes the {spectral norm.}} Note that the existence of a real logarithm
 of $X_t$ is an issue that underlies the study of the convergence of the ME. We state here our main result, proved in Section \ref{sec:convergence}, which deals
with these matters in the stochastic case{, when the coefficients $B,A^{(j)}$ in
\eqref{eq:SDE_linear_b} are progressively measurable processes. We defer a comparison with
previous convergence results for stochastic Magnus-type expansions to the next subsection}.
 We denote by $\Md$ the space of the $(d\times d)$-matrices
with real entries. {Also, for an $\Md$-valued stochastic process $M=(M_t)_{t\in[0,T]}$, we set
$\norm{M}_{T}\coloneqq\norm{\fnorm{M}}_{L^{\infty}([0,T]\times \O)}$, where $\fnorm{\cdot}$ denotes the Frobenius
(Euclidean entry-wise) norm.}
\begin{theorem}\label{th:convergence}
Let $A^{(1)},\dots,A^{(q)}$ and $B$ be bounded, progressively measurable, $\Md$-valued processes
defined on a filtered probability space $(\Omega,\mathcal{F},P,(\mathcal{F}_t)_{t\ge0})$ equipped
with a standard $q$-dimensional Brownian motion $W=(W^1,\cdots, W^q)$. For $T>0$ let also
$X=(X_t)_{t\in[0,T]}$ be the unique strong solution to \eqref{eq:SDE_linear_b} (see Lemma \ref{lem:preliminary}). There exists a strictly positive
stopping time $\tau\le T$ such that:
\begin{itemize}
\item[(i)] $X_t$ has a real logarithm $Y_t\in \Md$ up to time $\tau$, i.e.
\begin{equation}\label{eq:logarith_real_2}
 X_t = e^{Y_t},\qquad 0\le t<\tau;
\end{equation}
\item[(ii)] the following representation holds $P$-almost surely:
\begin{equation}\label{eq:convergence}
Y_t = \sum_{n=0}^{\infty} Y^{(n)}_t,\qquad 0\le t<\tau,
\end{equation}
where $Y^{(n)}$ is the $n$-th term in the stochastic ME as defined in \eqref{eq:stoch_magnus_exp} and \eqref{eq:represent_Y_general}--\eqref{eq:mu_general};
\item[(iii)] there exists a positive constant $C$, {only dependent on $\|A^{(1)}\|_{T},\dots,\|A^{(q)}\|_{T}$, $\|B\|_{T}$, $T$ and $d$}, such that
\begin{equation}\label{eq:estimate_tau_conv_b}
 P (\tau \leq t) \leq C t,\qquad t\in[0,T].
\end{equation}
\end{itemize}
\end{theorem}

The proof of (i) relies on the continuity of $X$ together with a standard representation for the
matrix logarithm. The key point in the proof of (ii) consists in
showing that $X^{\eps,\delta}_t$ {and its logarithm $Y^{\eps,\delta}_t$} are holomorphic as
functions of $(\eps,\delta)$, where $X^{\eps,\delta}_t$ represents the solution of
\eqref{eq:SDE_linear_b} when $A^{(j)}$ and $B$ are replaced by $\eps A^{(j)}$ and $\delta B$,
respectively. Once this is established, the representation \eqref{eq:convergence} follows from
observing that, by construction, the series in \eqref{eq:convergence} is exactly the formal power
series of $Y^{\eps,\delta}_t$ at $(\eps,\delta)=(1,1)$. To prove the {holomorphicity} of
$X^{\eps,\delta}_t$ we follow the same approach typically adopted to prove regularity properties
of stochastic flows.
Namely, in Lemma \ref{lem:preliminary} we {state} some %non-standard
maximal $L^p$ and H\"older estimates
(with respect {to} the parameters) for solutions to SDEs with random coefficients and combine them
with {the} Kolmogorov continuity theorem. Finally, the proof of (iii) owes one more time to the $L^p$
estimates in Lemma \ref{lem:preliminary} and to a Sobolev embedding theorem to obtain pointwise
estimates w.r.t. the parameters $(\eps,\delta)$ above.

{Theorem \ref{th:convergence} has been used in the recent paper \cite{Yang2021} (cf.
Lemma 1) %for a special case used to provide convergence of the Magnus expansion. In \cite{Yang2021}
where a semi-linear non-commuative It\^{o}-SDEs is studied and Euler, Milstein and derivative-free
numerical schemes are developed, with a convergence analysis for those schemes.}

In the last part of the paper we perform numerical tests with the Magnus expansion. In particular, Section \ref{sec:par_spdes} is devoted to the application of the
stochastic ME to the numerical solution of parabolic stochastic partial differential equations
(SPDEs). The idea is to discretize the SPDE {\it only in space} and then approximate the resulting
linear matrix-valued SDE by truncating the series in
\eqref{eq:logarith_real_2}-\eqref{eq:convergence}. The goal here is to propose the application of
stochastic MEs as novel approximation tools for SPDEs; we study the error of this approximating
procedure only numerically, in a case where an explicit benchmark is available, and we defer the
theoretical error analysis to further studies.

\subsection{Review of the literature and comparison.}\label{refer} Stochastic generalizations of the MEs were
{proposed by several authors.
%\sout{not unexplored in the existingliterature.}
} To the best of our knowledge, we recognize mainly two streams of research.

The beginning of the first one can be traced back to the work \cite{BenArous}, where the author derived
\emph{exponential} stochastic Taylor expansions (see also \cite{Azencott}, \cite{KloedenPlaten}
for general stochastic Taylor series) of the solution of a system
of Stratonovich SDEs with values on a manifold $\mathcal{M}$, i.e.
\begin{equation}\label{eq:sde_benarous}
\begin{cases}
\dd X_t = B(X_t) \dd t + A^{(j)}(X_t) \circ \dd W^{j}_t,  \\
 X_0 = x_0,
\end{cases}
\end{equation}
with $B,A^{(j)}%\in C^{\infty}(M).
$ being smooth, {deterministic and autonomous} vector fields on {$\mathcal{M}$}.
%\begin{equation}
%\dd X_t = B(X_t) \dd t + A^{(j)}(X_t) \circ \dd W^{j}_t, \qquad \text{with }B,A^{(j)}\in C^{\infty}(M).
%\end{equation}
The stochastic flow of \eqref{eq:sde_benarous} is represented in terms of the exponential map of a
stochastic vector field $Y$, i.e.
\begin{equation}
X_t (x_0) = \text{exp}\, Y_t (x_0) ,
\end{equation}
the vector field $Y$ being expressed by an infinite series of iterated stochastic integrals
multiplying nested commutators of the vector fields $B,A^{(j)}$. This representation is proved up
to a strictly positive stopping time and {extends} some previous results in \cite{Doss},
\cite{Sussmann} for the commutative case and in \cite{Yamato}, \cite{Kunita}, \cite{MR658689} for
the nilpotent case. {Refinements} of \cite{BenArous} were proved in \cite{Castell} making the
expansion of $Y$ more explicit. Later, numerical methods based on these representations were
proposed in \cite{CastellGaines} and \cite{CastellGaines96}. Such techniques, known as
Castell-Gaines methods, require the approximation of the solution to a time-dependent ODE besides
the approximation of iterated stochastic integrals. Truncating the expansion of $Y$ at a specified
order, these schemes turn out to be asymptotically efficient in the sense of Newton \cite{Newton}.

%   When $\mathcal{M}=\mathcal{G}$ is a Lie group, and $B$, $A^{j}$ are smooth invariant vector fields on $\mathcal{G}$, the representation for $X$ given in \cite{BenArous} can be understood as a stochastic extension of the deterministic Magnus expansion.
% We shall observe that the author of \cite{BenArous} does not mention this connection. Possibly, this is due to the fact that the original Magnus expansion and the expansions in \cite{BenArous} are motivated by different problems. Indeed, the former is an attempt to represent the solution of a non-autonomous linear deterministic differential equation, and the latter is an attempt of representing the flow of a stochastic differential equation, which depends on several vector fields, by means of the flow of a deterministic differential equation, which depends on one single stochastic vector field.
If $\mathcal{M}=\Md$ and the vector fields are linear, then \eqref{eq:sde_benarous} reduces to the
Stratono-vich version of \eqref{eq:SDE_linear_b} with $B,A^{(j)}$ constant matrices, and the
representation of $X$ given in \cite{BenArous} can be seen as a stochastic ME, in that the
exponential map of $Y$ reduces to the multiplication by a matrix exponential. In fact, in this
case the expansion in \cite{BenArous} becomes explicit in terms of iterated stochastic integrals,
and can be shown to coincide with the expansion in this paper by applying It\^o-Stratonovich
conversion formula. In the {very interesting} paper \cite{LordMalhamSimonWiese}, the authors study
several computational aspects of numerical schemes stemming from the truncated ME, in which the
iterated stochastic integrals are approximated by their conditional expectation. Besides showing
that asymptotic efficiency holds for an arbitrary number of Brownian components, they compare the
theoretical accuracy with the one of analogous schemes based on Dyson (or Neumann) series, which
are obtained by applying stochastic Taylor expansion directly on the equation. They find that,
although the theoretical accuracy of Magnus schemes is not superior, Magnus-based approximations
seem more accurate than their Dyson counterparts in practice. They also discuss the computational
cost deriving from approximating the iterated stochastic integrals and the matrix exponentiation,
in relation to different features of the problem such as the dimension and the number of Brownian
motions, as well as to the order of the numerical scheme.
%  can be proved to be equivalent to the ME obtained in this paper, by converting Stratonovich integrals into It\^o ones.

{The second stream of literature is explicitly aimed at extending the original Magnus results to
stochastic settings and can be traced back to \cite{burrage1999high} where the ME is derived via
formal arguments for a linear system of Stratonovich SDEs with deterministic coefficients.}
Clearly, in the autonomous case such expansion coincides with the one obtained by Ben Arous (\cite{BenArous}),
whereas in the non-autonomous case, $B\equiv 0$ and $j=1$, it is formally equivalent to the
deterministic ME \eqref{eq:maguns_det} with all the Lebesgue integrals replaced by Stratonovich
ones. The authors {of \cite{burrage1999high}} do not address the convergence of the ME, but rather
study computational aspects of the resulting approximation, in particular in comparison with
Runge-Kutta stochastic schemes. The authors of \cite{GoranSolo} consider the Ito SDE
\eqref{eq:SDE_linear_b} with constant coefficients, and propose to resolve via Euler method the
SDE \eqref{eq:SDE_linear_bbb} for the logarithm of the solution. In \cite{Wang} the ME for the
Stratonovich version of \eqref{eq:SDE_linear_b} with deterministic coefficients is applied to
solve non-linear SDEs; however, the error analysis of the truncated expansion seems flawed, since
the fact that the Magnus series converges only up to a positive stopping time is overlooked. In \cite{Muniz2021}, a general procedure for designing higher
strong order methods for It\^o SDEs on matrix Lie groups is outlined.

We now go on to discuss the contribution of this paper with respect to {the existing
literature.} In the first place, Theorem \ref{th:convergence} on the convergence of the ME %\red{\sout{up to a strictly positive stopping time}}
requires very weak conditions on the coefficients, which are stochastic processes satisfying the
sole assumption of progressive measurability. This is a novel aspect compared to the results in
\cite{BenArous}, \cite{Castell}, which surely cover a wider class of SDEs, but under the
assumption of {time-independent} deterministic coefficients. We point out that this feature is
also relevant in light of the fact that our result is stated for It\^o SDEs as opposed to
Stratonovich ones. Indeed, while this difference might appear as minor in the Markovian case,
where a simple conversion formula exists (cf. \cite{correales2018vs} and \cite{MR2180429}), it
becomes substantial in the case of progressively measurable coefficients. We also point out that,
even in the Markovian non-autonomous case, convergence issues were not discussed in
\cite{burrage1999high} and \cite{LordMalhamSimonWiese}.

Another novel aspect of our result concerns the estimate \eqref{eq:estimate_tau_conv_b} for the
{cumulative distribution function} of the stopping time $\tau$ up to which the Magnus series
converges to the real logarithm of the solution: {this kind of estimate} was unknown even in
the autonomous case. Theorem 11 in \cite{BenArous} (see also \cite{Castell}) provides an
asymptotic estimate for the truncation error of the logarithm, which in the linear case studied in
this paper would read as
 \begin{equation}
%Err^{(N)} :=
Y_t =  \sum_{n=1}^N Y^{(n)}_t + t^{\frac{N+1}{2}} R_t, \qquad t<T,
\end{equation}
with $R$ bounded in probability. Although this type of result holds for the general SDE
\eqref{eq:sde_benarous}, it is weaker {than} Theorem \ref{th:convergence} in the linear case. In
fact, it can be obtained by \eqref{eq:estimate_tau_conv_b} together with the standard estimate
$\norm{ \sup_{0\leq s\leq t} \fnorm{Y^{(n)}_s}  }_{L^2(\Omega)}\leq C t^{\frac{N+1}{2}}$, but not the other
way around.

{A rigorous error analysis of the ME is left for future research, as well as
applications to non-linear SDEs (see \cite{Wang} for a recent attempt in this direction).} %For the latter, ideas can already be found in \cite{Wang}.}

%We also note, that this work has recently already influenced the works of \cite{Friz2021} and \cite{Muniz2021}.

%\sout{Another goal, which is currently being developed, is an application to financial mathematics with the aim of providing a novel SPDE method for so-called Rating-trigger models. }}
%Finally, we also propose a new method for the numerical solution of SPDEs based on spatial
%discretization and application of the stochastic ME.

{\vspace{10pt} %\blu
The rest of the paper is structured as follows. In Section \ref{sec:Magnus} we
derive %and test
the ME and prove Theorem \ref{th:convergence}. In particular, Section \ref{sec:prel} contains
the key Lemma \ref{eq:lemm_diff} with a representations for the first and second order differentials through which the terms $Y^{(n)}$ in \eqref{eq:logarith_real_2}-\eqref{eq:convergence}
will be defined, and some preliminary results that will be used to derive the expansion. Section
\ref{sec:formal_expansion}  contains a formal derivation of
\eqref{eq:logarith_real_2}-\eqref{eq:convergence}. Section \ref{sec:convergence} is entirely
devoted to the proof of Theorem \ref{th:convergence}.

In Section \ref{sec:num_tests} we first introduce a numerical test for an SDE with
constant, non-commuting coefficients. The formulas for the first three orders of the
ME in this test will also be used in
in Subsection \ref{sec:par_spdes}, where we present
the application of the ME to the numerical solution of parabolic SPDEs. In particular, {in Subsection
\ref{subsec:SPDEs} we recall some general facts about stochastic Cauchy problems,
 in Subsection \ref{sec:finitediff_magnus}} we introduce the finite-difference--Magnus
approximation scheme and
we check the effectiveness of the proposed approach through numerical tests.
Finally, we provide an additional numerical test to assess the accuracy of the ME in the case of time-dependent coefficients.}

%\vspace{10pt} %\blu
%{The rest of the paper is structured as follows. In Section \ref{sec:Magnus} we
%derive %and test
%the ME and prove Theorem \ref{th:convergence}. In particular, Section \ref{sec:prel} contains the definitions of the
%operators through which the terms $Y^{(n)}$ in \eqref{eq:logarith_real_2}-\eqref{eq:convergence}
%will be defined and some preliminary results that will be used to derive the expansion. Section
%\ref{sec:formal_expansion}  contains a formal derivation of
%\eqref{eq:logarith_real_2}-\eqref{eq:convergence}. Section \ref{sec:convergence} is entirely
%devoted to the proof of Theorem \ref{th:convergence}. In Section \ref{sec:par_spdes} we present
%the application of the ME to the numerical solution of parabolic SPDEs. In particular, {in Section
%\ref{subsec:SPDEs} we recall some general facts about stochastic Cauchy problems,
 %in Section \ref{sec:finitediff_magnus}} we introduce the finite-difference--Magnus
%approximation scheme and %, while in Section \ref{sec:spde_numerics}
%we check the effectiveness of the proposed approach through numerical tests. {\color{red}\sout{Appendix} Section} \ref{sec:num_tests} contains some additional numerical tests to assess the accuracy of the ME, and
%{\color{red}\sout{Appendix} Lemma}
%\ref{eq:lemm_diff} contains a representations for the first and second order differentials of the
%matrix exponential map.}

\section{It\^{o}-Stochastic ME}\label{sec:Magnus}
{In this section we define the terms in the expansion \eqref{eq:convergence} and
%present some
%numerical tests to confirm the accuracy of the approximate solutions to \eqref{eq:SDE_linear_b}
%stemming from the truncation of %such an
%{the series.}
prove Theorem \ref{th:convergence}.
}

\subsection{Preliminaries}\label{sec:prel}
Let $\Md$ be the vector space of $(d\times d)$ real-valued matrices. {For the readers' convenience we recall the following notations.} Throughout the paper we denote by $[\cdot,\cdot]$ the standard Lie brackets operation, i.e.
\begin{equation}
[M,N] \coloneqq MN - NM ,\qquad M,N\in\Md,
\end{equation}
and by $\snorm{\cdot}$ the spectral norm on $\Md$. Also, we denote by $\bern_k$, $k\in\N_0$, the Bernoulli numbers defined as the derivatives of the function $x\mapsto x/(e^{x}-1)$ computed at $x=0$. For sake of convenience we report the first three Bernoulli numbers: $\bern_0=1$, $\bern_1=-\frac{1}{2}$, $\bern_2=\frac{1}{6}$. Note also that $\bern_{2m+1}=0$ for any $m\in\N$.

We now define the operators that we will use in the sequel. For a fixed $\Matrix\in\Md$, we let: %define the following operators:
\begin{itemize}
\item $ \ad^j_{\Matrix} : \Md \to \Md $, for $j\in\N_0$, be the linear operators defined as
\begin{align}
\ad^0_{\Matrix}(M)  &\coloneqq M,\\
    \ad^j_{\Matrix}(M) &\coloneqq \comm{{\Matrix}}{\ad^{j-1}_{\Matrix}(M)}, \qquad j\in\N .%\qquad  \qquad M\in \Md
\end{align}
To ease notation we also set $ \ad_{\Matrix} \coloneqq \ad^1_{\Matrix}$;
\item $e^{\ad_{\Matrix}}: \Md \to \Md $ be the linear operator defined as
\begin{equation}\label{eq:ad_operator}
e^{\ad_{\Matrix}}(M)\coloneqq
        \sum_{n=0}^{\infty}{
            \frac{1}{n!} \ad_{\Matrix}^n(M)
        }=
        e^{{\Matrix}} M e^{-{\Matrix}},
\end{equation}
where $e^{\Matrix}\coloneqq \sum_{j=0}^{\infty}{\frac{{\Matrix}^j}{j!}}$ is the standard matrix exponential;
\item $\dD_{\Matrix}: \Md \to \Md $ be the linear operator defined as
\begin{equation}
\label{eq:operator_deB}
   \dD_{\Matrix} (M) \coloneqq
            \int_{0}^{1}{
                 e^{\ad_{\tau\! {\Matrix}}}(M) \dd\tau
            }=
        \sum_{n=0}^{\infty} \frac{1}{(n+1)!} \ad^n_{\Matrix}(M);
\end{equation}
\item $\dDD_{\Matrix}: \Md \times \Md \to \Md $ be the bi-linear operator defined as
\begin{align}
\label{eq:operator_ddeB}
        &
    \dDD_{\Matrix}(M,N)\coloneqq
          \dD_{\Matrix}(M)\, \dD_{\Matrix}(N)  +
        \int_{0}^{1}{
            \tau \comm{\dD_{\tau \Matrix}  (N) }{e^{\ad_{\tau {\Matrix}}}(M)}\dd \tau
        }\\ \label{eq:operator_ddeB_bis}
    &=
    \sum_{n=0}^{\infty}{
        \sum_{m=0}^{\infty}{
            \frac{ \ad^n_{\Matrix}(M)}{(n+1)!} \frac{ \ad^m_{\Matrix} (N)}{(m+1)!}
        }
    }  +
    \sum_{n=0}^{\infty}{
        \sum_{m=0}^{\infty}{
            \frac{
                \comm{\ad^n_{{\Matrix}}(N)}{\ad^m_{{\Matrix}}(M)}
            }{
                (n+m+2)(n+1)!m!
            }
        }
    }.
\end{align}

\end{itemize}

In the next lemma we provide explicit expressions for the first and second order differentials of
the exponential map $\Md\ni M\mapsto e^M$. We recall that this map is smooth and in particular, it
is continuously twice differentiable.
\begin{lemma}\label{eq:lemm_diff}
%The standard exponential map $B\mapsto e^B$ is of class $C^2$ on $\Md$. In particular,
For any $\Matrix\in\Md$, the first and the second order differentials at $\Matrix$ of the exponential map $\Md\ni M\mapsto e^M$ are given by
\begin{align}\label{eq:first_diff}
   M&\mapsto  \dD_\Matrix(M)\, e^\Matrix = e^\Matrix \, \dD_{-\Matrix}(M) ,  && M\in\Md,\\
    (M,N)&\mapsto \dDD_\Matrix(M,N)\, e^\Matrix, && M,N\in\Md,
\end{align}
where $\dD_\Matrix$ and $\dDD_\Matrix$ are the linear and (symmetric) bi-linear operators as defined in \eqref{eq:operator_deB}-\eqref{eq:operator_ddeB}.
\end{lemma}
{We point out that this result, though very basic, is novel and of independent interest (for instance it was recently employed in \cite{Friz2021}).}
 %and \cite{Muniz2021}.
\begin{proof}
The first part of the statement, concerning the first order differential, is a classical result; its proof can be found in \cite[Lemma 2]{MR2494199}
among other references.

We prove the second part.
%Denote by $e_{ij}$ the generic element of the canonical basis of $\Md$. By the first part, the first-order partial derivative of $e^B$ w.r.t. $e_{ij}$ is
%\begin{equation}
%\partial_{e_{ij}} e^B = \dD_B (e_{ij}) \cdot e^B, \qquad B\in\Md.
%\end{equation}
Fix $M\in \Md$ and denote by $\partial_{M} e^\Matrix$ the first order directional derivative of $e^\Matrix$ w.r.t. $M$, i.e.
\begin{equation}
\partial_{M} e^\Matrix \coloneqq \frac{\dd}{\dd t} e^{\Matrix+t M}\Big|_{t=0}
.
\end{equation}
By the first part, we have
\begin{equation}
\partial_{M} e^\Matrix = \dD_\Matrix (M)\, e^\Matrix, \qquad \Matrix\in\Md.
\end{equation}
We now show that, for any $M,N\in\Md$, %there exists
the second order directional derivative %$\partial_{M,N} e^\Matrix$%, i.e.
\begin{equation}
\partial_{N,M} e^\Matrix: = \frac{\dd}{\dd t} \partial_M e^{\Matrix+t N}\Big|_{t=0}
\end{equation}
is given by %and
\begin{equation}\label{eq:second_derivatives}
\partial_{N,M} e^\Matrix = \dDD_\Matrix (N,M) \, e^\Matrix, \qquad \Matrix\in\Md.
\end{equation}
We have
\begin{align}
\frac{\dd}{\dd t} \partial_{M} e^{\Matrix+t N} &= \frac{\dd}{\dd t} \Big( \dD_{\Matrix+t N} (M)\, e^{\Matrix+t N} \Big)\\ \label{eq:ste10}
 &=  \dD_{\Matrix+t N} (M) \, \dD_{\Matrix+t N} (N)\, e^{\Matrix+t N} +  \Big( \frac{d}{dt} \dD_{\Matrix+t N} (M)\Big)\, e^{\Matrix+t N} .
\end{align}
%Consider now a smooth test function $\Omega:[0,T]\to \Md$ and compute
%\begin{equation}
%\frac{\dd}{\dd t} \partial_{e_{ij}} e^{Y_t} = \frac{\dd}{\dd t} \dD_{Y_t} (e_{ij}) e^{Y_t} = \dD_{Y_t} (e_{ij}) \dD_{Y_t} (\Omega')
%\end{equation}
%and note that by the product and chain rule for the derivative we have, as well as part one of the
%proof we have
%\begin{equation}\label{eq:ste_01}
%    \frac{\dd^2}{\dd t^2} e^{Y_t} =
%    \frac{\dd}{\dd t} \left(
%        \dD_{Y_t}\left(\Omega'_t\right) e^{Y_t}
%    \right) =
%    \left(
%        \dD_{Y_t}\left(\Omega'_t\right)
%    \right)^2 \, e^{Y_t} +
%    \left(
%        \frac{\dd}{\dd t} \left(\dD_{Y_t}\left(\Omega'_t\right)\right)
%    \right) \cdot e^{Y_t}.
%\end{equation}
%Now, let us consider the derivative in the last summand.
%First of all,
We use the definition \eqref{eq:operator_deB} and exchange the differentiation and
integration signs to obtain
\begin{align}
    \frac{\dd}{\dd t} \dD_{\Matrix+t N}(M)   & =
    \int_{0}^{1}{
        \frac{\dd}{\dd t}
            e^{\ad_{\tau (\Matrix+tN)}}(M)
      \dd\tau
    }
 \intertext{(by \eqref{eq:ad_operator})}
% By \eqref{eq:ad_operator} we have
%$e^{\ad_{\tau (B+tN)}}(M)=
%    e^{\tau (B+tN)}\, M \, e^{-\tau (B+tN)}$, which leads to
%\begin{align*}
%    \int_{0}^{1}{
%        \frac{\dd}{\dd t} \left(
%            e^{\ad_{\tau (B+tN)}}(M)
%        \right)\dd\tau
%    }
    &=
    \int_{0}^{1}{
        \frac{\dd}{\dd t} \Big( e^{ \tau (\Matrix+tN)} M e^{ - \tau (\Matrix+tN)} \Big) \dd\tau  }\\
        &=       \int_{0}^{1}{
            \left(\frac{\dd}{\dd t} e^{ \tau (\Matrix+tN)} \right)\,
            M \,
            e^{ - \tau (\Matrix+tN)}
            \dd\tau
        }%\\
        +
%        \int_{0}^{1}{
%            e^{ \tau (B+tN)}\,
%            \Omega''_t \,
%            e^{ - \tau (B+tN)}  \dd\tau
%        }+
        \int_{0}^{1}{
            e^{ \tau (\Matrix+tN)} \,
            M \,
             \frac{\dd}{\dd t} e^{ - \tau (\Matrix+tN)}
            \dd\tau
        }
        \intertext{(by employing the two expressions in \eqref{eq:first_diff} for the first-order differential)}
    &=
        \int_{0}^{1}{
            \tau
                \dD_{ \tau (\Matrix+tN)}(N) \,
                e^{ \tau (\Matrix+tN)}\,
            M \,
            e^{ - \tau (\Matrix+tN)}
            \dd\tau
        }  \\
        &\quad      -
%        \int_{0}^{1}{
%            e^{ \tau (B+tN)}\,
%            \Omega''_t \,
%            e^{ - \tau (B+tN)}  \dd\tau
%        }+
         \int_{0}^{1}{
            \tau
            e^{ \tau (\Matrix+tN)} \,
            M \,
                e^{ -\tau (\Matrix+tN)} \,
                \dD_{ \tau (\Matrix+tN)}(N)
            \dd\tau
        }\\
        &=   \int_{0}^{1}{
            \tau
            \comm{
                \dD_{\tau (\Matrix+tN)} (N)
            }{
                e^{\ad_{\tau (\Matrix+tN)}}(M)
            }
            \dd\tau
        }.
        \end{align}
This, together with \eqref{eq:ste10}, proves \eqref{eq:second_derivatives}. %Now, the continuity of $\Matrix\mapsto \partial_{M,N} e^\Matrix$ is a straightforward application of Lebesgue dominated convergence theorem.

To conclude, we prove equality \eqref{eq:operator_ddeB_bis}. It is enough to observe that
\begin{align*}
    \int_{0}^{1}{
        \tau \,
            \dD_{\tau \Matrix} (N)
    \,
            e^{\ad_{\tau \Matrix}}  (M)
        \dd \tau
    }&=
    \int_{0}^{1} {
        \tau
        \sum_{n=0}^{\infty}{
            \sum_{m=0}^{\infty}{
                \frac{ \ad^n_{\tau \Matrix}(N)}{(n+1)!}
                \frac{ \ad^m_{\tau \Matrix} (M)}{m!}
            }
        }
        \dd \tau
    }\\&=
    \int_{0}^{1}{
        \sum_{n=0}^{\infty}{
            \sum_{m=0}^{\infty}{
                \frac{\ad^n_{\Matrix} (N)}{(n+1)!}
                \frac{  \ad^m_{\Matrix} (M)}{m!}
                \tau^{n+m+1}
            }
        }\dd \tau
    }\\&=
    \sum_{n=0}^{\infty}{
        \sum_{m=0}^{\infty}{
            \frac{
                \ad^n_{\Matrix} (N)
                \ad^m_{\Matrix} (M)
            }{(n+m+2)(n+1)!m!}
        }
    }.
\end{align*}
%Finally, equality \eqref{eq:operator_ddeB_ter} stems from a straightforward reordering.
\end{proof}

\begin{proposition}[\bf It\^o formula]\label{prop:Ito 1}
   % Denoting by $e^{Y_t}$ the standard matrix exponential,
  Let $Y$ be an $\Md$-valued It\^o process of the form
\begin{equation}\label{eq:general_ito}
 \dd Y_t = \mu_t \dd t + %\sum_{j=1}^{q}
 {\sigma^{j}_t \dd W^{j}_t}. %, \qquad t>0.
\end{equation}
Then we have
            \begin{align*}
        \dd e^{Y_t} =%&=
               \bigg(
            \dD_{Y_t} \left(\mu_t\right)  +
            \frac{1}{2}
                \sum_{j=1}^{q}
                {
                    \dDD_{Y_t}\big( \s^j_t , \s^j_t \big)
                }
        \bigg) e^{Y_t}  \dd t +
            %\sum_{j=1}^{q}
            {\dD_{Y_t} \big( \s^j_t \big)\, e^{Y_t} \dd W^j_t
            }.%, \qquad t>0.
    \end{align*}
\end{proposition}
\begin{proof}
%The exponential map $\Md\ni M\mapsto e^M$ is of class $C^2$. Therefore,
The statement follows from %is a straightforward application of
the multi-dimensional It\^o formula (see, for instance, \cite{MR2791231}) combined with Lemma
\ref{eq:lemm_diff} and applied to the exponential process $e^{Y_t}$.
\end{proof}
We also have the following inversion formula for the operator $\dD_{\Matrix}$.
\begin{lemma}[\bf Baker, 1905]\label{lem:Baker}
Let ${\Matrix}\in\Md$. The operator $\dD_{\Matrix}$ is invertible if and only if the eigenvalues of the linear operator $\ad_{\Matrix}$ are different from $2m\pi$, $m\in\mathbb{Z}\setminus\{0\}$.
%such that the eigenvalues of the linear operator $\ad_{\Matrix}$ are different from $2m\pi$, $m\in\mathbb{Z}$. Then, the operator $\dD_{\Matrix}$ is invertible.
Furthermore, if $\norm{{\Matrix}}<\pi$, then
\begin{equation}\label{eq:inversion_baker}
\dD_{\Matrix}^{-1}(M) = \sum_{k=0}^{\infty} \frac{\bern_k}{k!}\ad^k_{\Matrix}(M), \qquad M\in \Md.
\end{equation}
\end{lemma}
For a proof to Lemma \ref{lem:Baker} we refer the reader to \cite{MR2494199}.

\subsection{Formal derivation}\label{sec:formal_expansion}
In this section we perform formal computations to derive the terms $Y^{(n)}$ appearing in the ME
\eqref{eq:convergence}. Although such computations are heuristic at this stage, they are meant to
provide the reader with an intuitive understanding of the principles that underlie the expansion
procedure.
%they allow to derive the terms of the expansion in a simple fashion.
 Their validity will be proved a fortiori, in Section \ref{sec:convergence}, in order to prove Theorem \ref{th:convergence}.

Let {$(\Omega,\mathcal{F},P,(\mathcal{F}_t)_{t\geq 0})$} {be} a filtered probability space. Assume
that, for any $\eps,\delta \in \mathbb{R}$, the process
$X^{\eps,\delta}=\big(X^{\eps,\delta}_t\big)_{t\geq 0}$ solves the It\^o SDE
{\begin{equation}\label{eq:SDE_linear_eps}
\begin{cases}
 \dd X^{\eps,\delta}_t =\delta B_t X^{\eps,\delta}_t dt+\eps A^{(j)}_t X^{\eps,\delta}_t \dd W^j_t,\\
 X^{\eps,\delta}_0 = I_d,
\end{cases}
\end{equation}
}and that it admits the exponential representation
\begin{equation}\label{eq:exponential_esp_del}
X^{\eps,\delta}_t = e^{Y^{\eps,\delta}_t} %\qquad \text{}.
\end{equation}
with {$Y^{\eps,\delta}%=\big(Y^{\eps,\delta}_t\big)_{t\geq 0}
$} being an %\blu{\sout{real matrix-valued}
{$\Md$-valued} It\^o process. %{\blue(If we use $\Md_{\R}$ here we have to change it all
%over the paper. My original idea was to keep $\Md$ in the paper to denote real-valued matrixes,
%except for Sec. 3, which is the only section were we use complex-valued matrices).}
Clearly, if $(\eps,\delta)=(1,1)$, then \eqref{eq:SDE_linear_eps}-\eqref{eq:exponential_esp_del}
reduce to \eqref{eq:SDE_linear_b}-\eqref{eq:exponential}.
%\eqref{eq:SDE_linear_b}, and that it admits an exponential representation as in \eqref{eq:exponential}, with $Y=(Y_t)_{t\geq 0}$ being a real matrix-valued It\^o process
Assume now that $Y^{\eps,\delta}$ is %\blu{\sout{ a $\Md_{\R}$-valued process}}
of the form \eqref{eq:general_ito}. Then, Proposition \ref{prop:Ito 1} yields
\begin{align}\label{eq:rep_A}
\eps A^{(j)}_t & = \dD_{Y^{\eps,\delta}_t} \big( \s^j_t \big), \qquad j=1,\dots,q,  \\ \label{eq:rep_B}
\delta B_t  &=  \dD_{Y^{\eps,\delta}_t} \left(\mu_t\right)  +
            \frac{1}{2}
                \sum_{j=1}^{q}{
                    \dDD_{Y^{\eps,\delta}_t}\big( \s^j_t , \s^j_t \big)
                }  .
\end{align}
Inverting now \eqref{eq:rep_A}-\eqref{eq:rep_B}, in accord with \eqref{eq:inversion_baker}, one obtains
\begin{align}
\s^j_t &=  \dD_{Y^{\eps,\delta}_t}^{-1} \big( \eps A^{(j)}_t \big) = \eps \sum_{k=0}^{\infty} \frac{\bern_k}{k!}\ad^k_{Y^{\eps,\delta}_t}\big(A^{(j)}_t\big) , \qquad \qquad j=1,\dots,q, \\
\mu_t &=  \dD_{Y^{\eps,\delta}_t}^{-1} \bigg( \delta B_t  -  \frac{1}{2}
                \sum_{j=1}^{q}{
                    \dDD_{Y^{\eps,\delta}_t}\big( \s^j_t , \s^j_t \big)
                }  \bigg) \\
                & =
                \sum_{k=0}^{\infty} \frac{\bern_k}{k!}\ad^k_{Y^{\eps,\delta}_t}\biggl(
                 \delta B_t  -  \frac{1}{2}
                \sum_{j=1}^{q}
                \sum_{n=0}^{\infty}
        \sum_{m=0}^{\infty}
                \begin{aligned}[t]\arraycolsep=0pt
\biggl(&
            \frac{ \ad^n_{Y^{\eps,\delta}_t}\big(\s^j_t\big)}{(n+1)!} \frac{ \ad^m_{Y^{\eps,\delta}_t} \big(\s^j_t\big)}{(m+1)!}
       \\&+            \frac{
                \big[{\ad^n_{{Y^{\eps,\delta}_t}}\big(\s^j_t\big)},{\ad^m_{{Y^{\eps,\delta}_t}}\big(\s^j_t\big)}\big]
            }{
                (n+m+2)(n+1)!m!
            }
    \biggr) \biggr).
                \end{aligned}
\end{align}
Equivalently, $Y^{\eps,\delta}$ solves the It\^o SDE %\eqref{eq:Ito} with
{\begin{equation}\label{eq:SDE_linear_bbb}
\begin{cases}
 \dd Y^{\eps,\delta}_t = \mu^{\eps,\delta}\big(t,Y^{\eps,\delta}_t\big) \dd t +  %\sum_{j=1}^q
 \sigma^{\eps}_j\big(t,Y^{\eps,\delta}_t\big) \dd W^j_t,%\qquad t>0,
 \\ Y^{\eps,\delta}_0 = 0,
\end{cases}
\end{equation}
}with
\begin{align} \label{ae1}
\sigma^{\eps}_j(t,\cdot)   &=  \eps  \sum_{n=0}^{\infty} \frac{\bern_n}{n!}
\ad^n_{\,\cdot}\big(A^{(j)}_t\big)       ,\qquad {j=1,\dots,d,}    \\  \label{ae2}
\mu^{\eps,\delta}(t,\cdot) & =        \sum_{n=0}^{\infty} \frac{\bern_n}{n!} \ad^n_{\,\cdot}
\bigg(\delta B_t - \frac{1}{2} \sum_{j=1}^q
\dDD_{\cdot}\big(\s^{\eps}_j(t,\cdot),\s^{\eps}_j(t,\cdot)\big)\bigg).
\end{align}

We now assume that $Y^{\eps,\delta}$ admits the representation
\begin{equation}\label{eq:expansion_Yepsdel}
Y^{\eps,\delta}_t = \sum_{n=0}^{\infty} \sum_{r=0}^{n} Y^{(r,n-r)}_t \eps^{r}\delta^{n-r},
\end{equation}
for a certain family {$(Y^{(r,n-r)})_{n,r\in\N_0}$} of stochastic processes. In particular,
setting $(\eps,\delta)=(1,1)$, \eqref{eq:expansion_Yepsdel} would yield
\begin{equation}\label{eq:stoch_magnus_exp}
Y_t = \sum_{n=0}^{\infty} Y^{(n)}_t \quad \text{with}\quad Y^{(n)}_t  \coloneqq \sum_{r=0}^{n}
Y^{(r,n-r)}_t.
\end{equation}
\begin{remark}
Note that it is possible to re-order the double series\\ $\sum\limits_{n=0}^{\infty}
\sum\limits_{r=0}^{n} Y^{(r,n-r)}_t$ according to any arbitrary choice, for the latter will be
proved to be absolutely convergent. The above choice for $Y^{(n)}$ contains all the terms of equal
order by weighing $\eps$ and $\delta$ in the same way. A different choice, which respects the
probabilistic relation $\sqrt{\Delta t} \approx  \Delta W_t$, corresponds to weighing $\delta$ as
$\eps^2$. This would lead to setting
\begin{equation}
Y^{(n)}_t \coloneqq \sum_{r=0}^{\floor*{\frac{n}{2}}} Y^{(n-2 r, r)}_t
\end{equation}
in \eqref{eq:stoch_magnus_exp}.
\end{remark}

\begin{remark}\label{rem:initial_condition}
Observe that, if the function $(\eps,\delta)\mapsto Y^{\eps,\delta}_0$ is assumed to be continuous $P$-almost surely, then the initial condition in \eqref{eq:SDE_linear_bbb} implies
\begin{equation}
 Y^{(i,j)}_0 = 0 \quad P\text{-a.s}., \qquad i,j\in\N_0,
\end{equation}
 and thus
 \begin{equation}
Y^{(n)}_0 = 0 \quad P\text{-a.s}., \qquad n\in\N_0.
\end{equation}
\end{remark}
We now plug \eqref{eq:expansion_Yepsdel} into \eqref{eq:SDE_linear_bbb} and collect {all} terms of equal
order in $\eps$ and $\delta$. Up to order $2$ we obtain
%together with Remark \ref{rem:initial_condition}, yields
\begin{align}\label{eq:Y0}
\eps^0\delta^0:&\quad  Y^{(0,0)}_t = 0,\\ \label{eq:Y10}
%\end{equation}
%\begin{equation}
 \eps^1\delta^0:&\quad
 Y^{(1,0)}_t =   % \sum_{j=1}^q
 \int_0^t  A^{(j)}_s  \dd W^j_s , %\qquad t\geq 0,
%\begin{cases}
%\dd Y^{(1,0)}_t =   \dd t +  \sum_{j=1}^q   \dd W^j_t,\qquad t>0 \\
%Y^{(1,0)}_0 = 0
%\end{cases},
\\ \label{eq:Y01}
\eps^0\delta^1:&\quad
Y^{(0,1)}_t =  \int_0^t   B_s   \dd s, %\qquad t\geq 0,
%\begin{cases}
%\dd Y^{(0,1)}_t =   \dd t +  \sum_{j=1}^q   \dd W^j_t,\qquad t>0 \\
%Y^{(0,1)}_0 = 0
%\end{cases},
\\
\eps^2\delta^0:&\quad
 Y^{(2,0)}_t = - \frac{1}{2} %\sum_{j=1}^q
 \int_0^t   \big(A^{(j)}_s\big)^2  \dd s  + \frac{1}{2}  %\sum_{j_1=1}^q \sum_{j_2=1}^q
 \int_0^t
 \Big[ A^{(j)}_s , \int_{0}^{s} A^{(i)}_r \dd {W^{i}_r}  \Big]  \dd W^{j}_s ,\qquad %\red{\textbf{Check!}} %\qquad t\geq 0,
%\begin{cases}
%\dd Y^{(2,0)}_t =   \dd t +  \sum_{j=1}^q   \dd W^j_t,\qquad t>0 \\
%Y^{(2,0)}_0 = 0
%\end{cases},
\\
 \eps^1\delta^1:&\quad Y^{(1,1)}_t =  \frac{1}{2} %\sum_{j=1}^q
 \int_0^t   \Big[ B_s , \int_{0}^{s}
 A^{(j)}_r \dd W_r^j  \Big]  \dd s  + \frac{1}{2}   %\sum_{j=1}^q
 \int_0^t \Big[ A^{(j)}_s ,
 \int_{0}^{s} B_r \dd %W_r
 r  \Big]  \dd {W^{j}_s} ,%\qquad \red{\textbf{Check: there was $W^{j_1}_s$}}
%\begin{cases}
%\dd Y^{(1,1)}_t =   \dd t +  \sum_{j=1}^q   \dd W^j_t,\qquad t>0 \\
%Y^{(1,1)}_0 = 0
%\end{cases},
\\
\eps^0\delta^2:&\quad
Y^{(0,2)}_t =  \frac{1}{2}    \int_0^t
\Big[ B_s , \int_{0}^{s} B_r \dd r  \Big]  \dd s ,
%\begin{cases}
%\dd Y^{(0,2)}_t =   \dd t +  \sum_{j=1}^q   \dd W^j_t,\qquad t>0 \\
%Y^{(0,2)}_0 = 0
%\end{cases}.
\end{align}
for any $t\geq 0$, where we used, one more time, Einstein summation convention to imply summation over the
indexes {$i,j$} %{\bf [We might use $i,j$ instead of $j,j_{1},j_{2}$]}}
and Remark
\ref{rem:initial_condition} to set all the initial conditions equal to zero. Proceeding by
induction, one can obtain a recursive representation for the general term $Y^{(r,n-r)}$ in
\eqref{eq:expansion_Yepsdel}, namely: {\begin{equation}\label{eq:represent_Y_general}
 Y^{(r,n-r)}_t =  \int_0^t \mu^{r,n-r}_s \dd s + %\sum_{j=1}^q
 \int_0^t \sigma^{r,n-r,j}_s \dd W^j_s, \qquad n\in \N_0, \ r=0,\dots,n,
\end{equation}
}%with
%\begin{align}
%I^{0,r,n-r}_t := \int_0^t \mu^{r,n-r}_s \dd s,  &&  I^{1,r,n-r}_t = \int_0^t \sigma^{r,n-r}_s \dd W_s,
%\end{align}
where the terms $\sigma^{r,n-r,j},\mu^{r,n-r}$ are defined recursively as
\begin{align}\label{eq:sigma_general}
%I^{1,r,n-r}_t = \int_0^t \sigma^{r,n-r}_s \dd W_s,&&
\sigma^{r,n-r,j}_s &\coloneqq \sum_{i=0}^{n-1}\frac{\beta_i}{i!}   S^{r-1,n-r,i}_s\big(A^{(j)}\big),\\ \label{eq:mu_general}
\mu^{r,n-r}_s & \coloneqq \sum_{i=0}^{n-1}\frac{\beta_i}{i!}    S^{r,n-r-1,i}_s(B)
 - \frac{1}{2}  \sum_{j=1}^q
  \sum_{i=0}^{{ n-2}}\frac{\beta_i}{i!}  \sum_{q_1=2}^{{ r }} \sum_{q_2=0}^{{ n-r}}  S^{r-q_1,n-r-q_2,i}   \big( Q^{q_1,q_2,j} \big),
\end{align}
with
\begin{align}
 Q^{q_1,q_2,j}_s \coloneqq    \sum_{i_1=2}^{q_1}&\sum_{i_2=0}^{q_2} \sum_{h_1=1}^{i_1-1} \sum_{h_2=0}^{i_2}  \sum_{p_1=0}^{q_1-i_1}
        \sum_{{p_2}=0}^{q_2-i_2}\ \sum_{m_1=0}^{p_1+p_2}
        \ \sum_{{m_2}=0}^{q_1-i_1-p_1+q_2-i_2-p_2} \\
        & \Bigg({
            \frac{S_s^{p_1,p_2,m_1}\big(\s^{h_1,h_2,j}_s\big)}{({m_1}+1)!} \frac{ S_s^{q_1-i_1-p_1,q_2-i_2-p_2,m_2} \big(\s^{i_1-h_1,i_2-h_2,j}_s\big)}{({m_2}+1)!}
           }   \\
        & \qquad\qquad +    {\frac{
                \comm{S_s^{p_1,p_2,m_1}\big(\s^{i_1-h_1,i_2-h_2,j}_s\big)}{S_s^{q_1-i_1-p_1,q_2-i_2-p_2,m_2}\big(\s^{h_1,h_2,j}_s\big)}
            }{
                ({m_1}+{m_2}+2)({m_1}+1)!{m_2}!
            }
        }
    \Bigg),
\end{align}
and with the operators $S$ being defined as
%\begin{align}
%S^{0,0,0}_s(A) = A,&&  S^{r-1,n-r,0}_s(A) = 0\ \text{for } r\neq 1 \ \text{or } n \neq r ,
%\end{align}
\begin{align}
S^{r-1,n-r,0}_s(A) &\coloneqq
\begin{cases}
    A  & \text{if } r=n=1,\\
    0  & \text{otherwise},
\end{cases}\\
%\end{equation}
%and
%\begin{align}
S^{r-1,n-r,i}_s(A)  &\coloneqq %\sum_{\substack{ m_1,\cdots, m_i\in \{0,1\}\\ m_1+\cdots+m_i=l}}
\sum_{\substack{(j_1,k_1),\dots,(j_i,k_i) \in\N_0^2 %\\ (k^0_1,j^0_1),\cdots,(k^0_{i-l},j^0_{i-l})\in\N_0^2
\\  j_1 + \cdots %+ k^1_l + k^0_1 + \cdots
+  j_i  = r-1
\\  k_1+ \cdots %+ j^1_l + j^0_1 + \cdots
 +k_{i} = n-r
}}     \big[Y^{(j_1,k_1)}_s ,  \big[ \dots , \big[  Y^{(j_i,k_i)}_s, A_s   \big] \dots \big]  \big]  \\
&= \sum_{\substack{(j_1,k_1),\dots,(j_i,k_i) \in\N_0^2 %\\ (k^0_1,j^0_1),\cdots,(k^0_{i-l},j^0_{i-l})\in\N_0^2
\\  j_1 + \cdots %+ k^1_l + k^0_1 + \cdots
+  j_i  = r-1
\\  k_1+ \cdots %+ j^1_l + j^0_1 + \cdots
k_{i} = n-r
}}   \ad_{Y^{(j_1,k_1)}_s}%\circ \ad_{I^{m_2,k_2}_s}
\circ \cdots \circ \ad_{Y^{(j_i,k_i)}_s}(A_s) , \qquad i\in\N.
\end{align}
\begin{remark}
All the processes $Y^{(r,n-r)}$, with $n\in\N$ and $r=0,\dots,n$, are well defined according to
the recursion \eqref{eq:represent_Y_general}-\eqref{eq:sigma_general}-\eqref{eq:mu_general}, as
long as $B$ and $A^{(1)},\dots, A^{(q)} $ are bounded and progressively measurable stochastic
processes.
\end{remark}
\begin{example}
As we already pointed out in the introduction, in the case $j=1$ and $B\equiv 0$, the SDE \eqref{eq:SDE_linear_b} admits an explicit solution given by
\begin{equation}
 Y_t = -\frac{1}{2} A^2 t + A W_t , \qquad t\geq 0,
\end{equation}
and the terms in the ME \eqref{eq:stoch_magnus_exp} read as
\begin{equation}
Y^{(1)}_t = A W_t,\qquad Y^{(2)}_t = -\frac{1}{2} A^2 t,\qquad Y^{(n)}_t=0, \ n\geq 3.
\end{equation}
In particular, the ME coincides with the exact solution with the first two terms.
\end{example}

\subsection{Convergence analysis}\label{sec:convergence} In this section we prove Theorem
\ref{th:convergence}. {To avoid ambiguity, only in this section, w}e denote by $\Md_{\R}$ and
$\Md_{\mathbb{C}}$ the spaces of $(d\times d)$-matrices with real and complex entries,
respectively; on these spaces we shall make use of the Frobenius %(Euclidean entrywise)
norm denoted by $\fnorm{\cdot}$. We say that a matrix-valued function is holomorphic if all its entries
are holomorphic functions. {We recall that
%For fixed $\e,\d\in\mathbb{C}$, we consider the SDE
%\begin{equation}\label{eq:SDE_linear_eps2}
% \dd X^{\eps,\delta}_t =\delta B_t X_t^{\e,\d} dt+ \eps A^{(j)}_t X_t^{\e,\d} \dd W^j_t,\qquad
% X^{\eps,\delta}_0 = I_d,
%\end{equation}
$W=(W^1,\dots, W^q)$ is a $q$-dimensional standard Brownian motion and $A^{(1)},\dots, A^{(q)},B$
are bounded $\Md_{\R}$-valued progressively measurable stochastic processes defined on a filtered
probability space $(\Omega,\mathcal{F},P,(\mathcal{F}_t)_{t\geq 0})$. Also recall that, for any
$\Md_{\mathbb{R}}$-valued process $M=(M_t)_{t\in[0,T]}$, we set
$\norm{M}_{T}\coloneqq\norm{\fnorm{M}}_{L^{\infty}([0,T]\times \O)}$.

%As in Section \ref{sec:formal_expansion}, we associate to \eqref{eq:SDE_linear_eps} the auxiliary
%SDE
%\begin{equation}\label{eq:SDE_linear_bbb2}
% \dd Y^{\eps,\delta}_t = \mu^{\eps,\delta}\big(t,Y^{\eps,\delta}_t\big) \dd t +
% \sigma_{j}^{\eps}\big(t,Y^{\eps,\delta}_t\big) \dd W^j_t,\qquad Y^{\eps,\delta}_0 = 0,
%\end{equation}
%where the coefficients $\mu^{\eps,\delta}$ and $\sigma_j^{\eps}$ are defined as in
%\eqref{ae1}-\eqref{ae2}.

%For any $\eps,\delta \in\mathbb{C}$ consider the complex matrix-valued SDE
%\begin{equation}\label{eq:SDE_linear_complex}
%\dd X^{\eps,\delta}_t =
%\sum_{j=1}^q \eps A^{(j)}_t X^{\eps,\delta}_t \dd W^j_t+
%\delta B_t X^{\eps,\delta}_t dt
%, \qquad X^{\eps,\delta}_0=I_d.
%\end{equation}
%Clearly, for $\eps=\delta=1$ we obtain the original SDE \eqref{eq:SDE_linear}. In the sequel, we say that a matrix-valued function is holomorphic if all its components are holomorphic functions.  Furthermore, for any $\eps,\delta\in\mathbb{C}$ we fix the following notation:
%\begin{align}
%\eps = u_{\eps} + i v_{\eps},\qquad \delta = u_{\delta} + i v_{\delta}, \qquad \text{with }  u_{\eps},v_{\eps},u_{\delta},v_{\delta}\in \R .   %\\
%\end{align}
We start with two preliminary lemmas.
\begin{lemma}\label{lem:verification}
%Assume that $Y=(Y^{\eps,\delta}_t)_{t\geq 0,\, |\eps|,|\delta|<h}$ is a $\Md_{\R}$-valued
Assume that $Y=(Y^{\eps,\delta}_t)_{\eps,\delta\in\R,\, t\in\R_{\ge0}}$ is a $\Md_{\R}$-valued
stochastic process that can be represented as a convergent series of the form
\eqref{eq:expansion_Yepsdel}. If $Y$ solves the SDE \eqref{eq:SDE_linear_bbb} up to a positive
stopping time $\tau$, then $Y^{(r,n-r)}$ in \eqref{eq:expansion_Yepsdel} are It\^o processes and
satisfy \eqref{eq:represent_Y_general}-\eqref{eq:sigma_general}-\eqref{eq:mu_general} for any
$t<\tau$.
%such that:
%\begin{enumerate}
%\item[(a)] for any $(\eps,\delta)\in D$, $Y^{\eps,\delta}$ ;
%%\item[(b)] $(t,\eps,\delta)\mapsto Y^{\eps,}$ is continuous on $[0,\infty[\times U$.%, $P$-a.s.;
%\item[(b)] $Y^{\eps,\delta}_t$ can be represented as \eqref{eq:expansion_Yepsdel} for any $(t,\eps,\delta)\in [0,\tau[\times U$, %$P$-a.s.,
%for some random coefficients $Y^{(r,n-r)}_t$. %adapted and continuous processes.
%\end{enumerate}
%Then the families %, $P$-a.s.
%%Assume that the real $\Md$-valued random field $(t,\eps,\delta)\mapsto Y^{\eps,}$ is $P$-a.s. continuous. Assume also that, for any $\eps,\delta$ that belongs to a certain neighborhood of $(0,0)$, the process $Y=(Y^{\eps,\delta}_t)_{t\geq 0}$ solves \eqref{eq:SDE_linear_bbb}.
\end{lemma}
\begin{proof}
We prove \eqref{eq:represent_Y_general}-\eqref{eq:sigma_general}-\eqref{eq:mu_general} only for
$n=0,1$. Namely, we show that \eqref{eq:Y0}, \eqref{eq:Y10} and \eqref{eq:Y01} hold up to time
$\tau$, $P$-a.s. The representation for the general term $Y^{(r,n-r)}$ can be proved by induction;
we omit the details for brevity.

Since $Y$ is of the form \eqref{eq:expansion_Yepsdel} then $Y^{(0,0)}_t = Y^{0,0}_t$ for any
$t<\tau$. Moreover, since $Y$ solves the SDE \eqref{eq:SDE_linear_bbb} then $Y^{0,0}\equiv 0$ on
$[0,\tau[$, $P$-a.s. Thus \eqref{eq:Y0} holds up to time $\tau$, $P$-a.s.

Now, \eqref{eq:SDE_linear_bbb} yields
\begin{align}
Y^{\eps,0}_t %& =\int_{0}^t \bigg(    \eps  A^{(j)}_s + \eps \sum_{k=1}^{\infty}
%\frac{\bern_k}{k!}\ad^k_{Y^{\eps,0}_s}\big(A^{(j)}_s\big)    \bigg) \dd W^j_s \\ &\quad -
%\frac{\eps^2}{2} \int_0^t \dD_{Y^{\eps,0}_s}^{-1} \Bigg(
%                {
%                    \dDD_{Y^{\eps,0}_s}\bigg(  \sum_{k=0}^{\infty} \frac{\bern_k}{k!}\ad^k_{Y^{\eps,0}_s}\big(A^{(j)}_s\big)     , \sum_{k=0}^{\infty} \frac{\bern_k}{k!}\ad^k_{Y^{\eps,0}_s}\big(A^{(j)}_s\big)     \bigg)
%                }  \Bigg) \dd s
%  \intertext{($A$ is progressively measurable and bounded)}
               & = \eps \int_{0}^t    A^{(j)}_s  \dd W^{j}_s + \eps R^{\eps}_t , \qquad t\in[0,\tau[, \quad P\text{-a.s.},
               \label{eq:ste123}
\end{align}
%$P$-a.s.,
where
\begin{multline}
 R^{\eps}_t  = \int_{0}^t \bigg(    \sum_{k=1}^{\infty} \frac{\bern_k}{k!}\ad^k_{Y^{\eps,0}_s}\big(A^{(j)}_s\big)    \bigg) \dd W^j_s \\- \frac{\eps}{2}  \int_0^t \dD_{Y^{\eps,0}_s}^{-1} \Bigg(
                %\sum_{j=1}^{q}
                    \dDD_{Y^{\eps,0}_s}\bigg(  \sum_{k=0}^{\infty} \frac{\bern_k}{k!}\ad^k_{Y^{\eps,0}_s}\big(A^{(j)}_s\big)     , \sum_{k=0}^{\infty} \frac{\bern_k}{k!}\ad^k_{Y^{\eps,0}_s}\big(A^{(j)}_s\big)     \bigg)
                 \Bigg) \dd s.
\end{multline}
Note that, again by \eqref{eq:SDE_linear_bbb}, $R^{0}\equiv 0$ $P$-a.s. Moreover, representation
\eqref{eq:expansion_Yepsdel} implies continuity of $\eps\mapsto Y^{\eps,0}_t$ near $\eps=0$, which
in turn implies the continuity of $\eps\mapsto R^{\eps}_t$. Thus we have $\lim\limits_{\eps\to
0}R^\eps_t=R^0_t$
%\begin{equation}
%R^\eps_t \longrightarrow R^{0}_t = 0 \quad \text{as } \eps\to 0, \qquad t\geq 0,
%\end{equation}
$P$-a.s. This, together with \eqref{eq:ste123} and \eqref{eq:expansion_Yepsdel} implies that
\eqref{eq:Y10} necessarily holds, up to time $\tau$, $P$-a.s.

Similarly, \eqref{eq:SDE_linear_bbb} yields
\begin{align}
Y^{0,\delta}_t %& = \int_0^t \bigg(  \delta  B_s + \delta \sum_{k=1}^{\infty} \frac{\bern_k}{k!}\ad^k_{Y^{0,\delta}_s} (B_s ) \bigg)  \dd s
%\intertext{($B$ is progressively measurable and bounded)}
 = \delta \int_0^t  B_s \dd s + \delta Q^{\delta}_t , \qquad t\in[0,\tau[, \quad P\text{-a.s.},
\end{align}
with
\begin{equation}
Q^{\delta}_t  =     \int_0^t \bigg(  \sum_{k=1}^{\infty} \frac{\bern_k}{k!}\ad^k_{Y^{0,\delta}_s}
(B_s ) \bigg)  \dd s
\end{equation}
and the same argument employed above yields \eqref{eq:Y01} up to time $\tau$, $P$-almost surely.
\end{proof}
%\begin{lemma}\label{lem:log_matrix}
%Let $M\in\mathcal{M}_{\mathbb{C}}^{d\times d}$ be invertible with no non-positive real eigenvalues, i.e. $\lambda \in \mathbb{C}\setminus ]-\infty,0]$ for any $\lambda$ eigenvalue of $M$.
% Then $M$ admits a {\blue(unique ?)} logarithm, which is %. In particular, we have
%\begin{equation}\label{eq:log_M}
%\log M = (M - I_d) \int_0^{\infty} \frac{1}{1+\mu} (\mu I_d + M)^{-1} \dd \mu.
%\end{equation}
%\end{lemma}

%We set $\norm{B}_{\infty}=\||B|\|_{L^{\infty}(\R_{\ge0}\times \O)}$.
%We recall the notation  on $\mathcal{M}_{\mathbb{C}}^{d\times d}$.
\begin{lemma}\label{lem:log_matrix}
Let $M\in\mathcal{M}_{\mathbb{C}}^{d\times d}$ be {nonsingular} and such that $\snorm{ M - I_d} < 1$
where $\snorm{\cdot}$ is the spectral norm. Then $M$ has a unique logarithm, which is
\begin{align}\label{eq:log_M}
\log M &=  \sum_{n=1}^{\infty} (-1)^{n+1} \frac{(M-I_d)^n}{n} \\ &=  (M - I_d) \int_0^{\infty}
\frac{1}{1+\mu} (\mu I_d + M)^{-1} \dd \mu .
\end{align}
In particular, we have
\begin{equation}\label{estimate_log}
 \snorm{\log M}\leq - \log\big( 1 - \snorm{ M - I_d }  \big).
\end{equation}
\end{lemma}
%\begin{proof}
%The existence {\blue(and uniqueness ?)} can be found in \citep{gantmakher1959theory}. The representation \eqref{eq:log_M} stems from the factorization $M=V J V^{-1}$ with $J$ in Jordan form.
\proof The first representation is a standard result. The second representation stems from the
factorization $M=V J V^{-1}$ with $J$ in Jordan form, under the assumption that $M$ has no
non-positive real eigenvalues, i.e. $\lambda \in \mathbb{C}\setminus ]-\infty,0]$ for any
$\lambda$ eigenvalue of $M$. This last property, however, is ensured by the assumption $\snorm{ M - I_d
} < 1$. Indeed, the latter implies
\begin{equation}
 \fnorm{M v - v} < 1 , \qquad v\in\R^d,\ \enorm{v}=1,
\end{equation}
which in turn implies that, if $\l$ is a real
eigenvalue of $M$ and $v$ is one of its normalized eigenvectors, then %by \eqref{ae1} we have
  $$1>\fnorm{M v - v}=\enorm{\l v-v}=\abs{\l-1}.$$
%\end{proof}
\endproof
We have one last preliminary lemma, containing some technical results concerning the solutions to
\eqref{eq:SDE_linear_eps}. % and the relative flows $(t,\eps,\delta)\mapsto Y^{\eps,\delta}_t$.
These are semi-standard, in that they can be inferred by combining and adapting existing results
in the literature. %For sake of completeness, we report a sketch of the proof in Appendix \ref{app:proof_lemma}.
%In the following statement,
\begin{lemma}\label{lem:preliminary}
For any $T>0$ and $\e,\d\in\mathbb{C}$, the SDE \eqref{eq:SDE_linear_eps} has a unique strong
solution $(X^{\eps,\delta}_t)_{t\in[0,T]}$. For any $p\ge 1$ and $h>0$ there exists a positive
constant $\kappa$, only dependent on $\|A^{(1)}\|_{T},\dots,\|A^{(q)}\|_{T}$, $\|B\|_{T}$, $d$,
$T$, $h$ and $p$, such that
\begin{align}\label{eq:estimates_holder_eps_delta}
 E\Big[  \fnormp{  X^{\eps,\delta}_t - X^{\eps',\delta'}_{s}} \Big]& \leq
 \kappa {\big( \abs{t-s}^p +\left(\abs{\eps-\eps'}+\abs{\delta-\delta'}\right)^{2p} \big)} ,\\
 \label{eq:estimates_Xepsdel}
 E\Big[  \sup_{0 \leq u\leq t}  \fnormp{  X^{\eps,\delta}_u - X^{\eps,\delta}_{0}  } \Big]
 &\leq  \kappa t^{p} (\abs{\eps} + \abs{\delta})^{2p}, %e^{\kappa (|\eps| + |\delta|) t}
\end{align}
for any $0\le t,s\le T$ and $\eps,\delta,\eps',\delta'\in \mathbb{C} $ with
$\abs{\eps},\abs{\delta},\abs{\eps'},\abs{\delta'}\le h$.

Up to modifications, $(X^{\eps,\delta}_t)_{\e,\d\in\mathbb{C},\, t\in[0,T]}$ is a continuous
process such that:
\begin{itemize}
  \item[i)] for any $t\in[0,T]$, the function
  $(\eps,\delta)\mapsto X^{\eps,\delta}_t$ is holomorphic;
  \item[ii)] the functions $(t,\eps,\delta)\mapsto \partial_{\eps} X^{\eps,\delta}_t$ and $(t,\eps,\delta)\mapsto \partial_{\delta} X^{\eps,\delta}_t$ are
  continuous;
  \item[iii)] for any $p\ge1$ and $h>0$ there exists a positive constant $\kappa$ only dependent on\\
  $\|A^{(1)}\|_{T},\dots,\|A^{(q)}\|_{T}$, $\|B\|_{T}$, $d$, $T$,
$h$ and $p$, such that
\begin{equation}\label{eq:derestimates_Xepsdel}
 E\Big[  \sup_{0 \leq s\leq t} \Big\{  \fnormp{   \partial_{\eps} X^{\eps,\delta}_s  }
 + \fnormp{  \partial_{\delta} X^{\eps,\delta}_s  }  \Big\} \Big] \leq
   \kappa t^p(\abs{\eps} + \abs{\delta})^p,  %{\blue\qquad\text{is this correct??}}
\end{equation}
for any $t\in[0,T]$ and $\abs{\eps},\abs{\delta}\le h$.
\end{itemize}
\end{lemma}
\proof Existence of the solution and estimates
\eqref{eq:estimates_holder_eps_delta}-\eqref{eq:estimates_Xepsdel} of the moments follow from the
results in Section 5, Chapter 2 in \cite{krylov2008controlled} (in particular, see Corollary 5 on
page 80 and Theorem 7 on page 82).

The second part of the statement is a refined version of the Kolmogorov continuity theorem in the
form that can be found for instance in Section 2.3 in \cite{MR3929750}: a detailed proof is
provided in \cite{phdthesis}.
\endproof

\begin{remark}
The existence and uniqueness for the solution to \eqref{eq:SDE_linear_b} is a particular case of
the previous result.
\end{remark}
We are now in the position to prove Theorem \ref{th:convergence}.
\begin{proof}[of Theorem \ref{th:convergence}]
%The strong existence and uniqueness of SDE \eqref{eq:SDE_linear_eps} clearly stems from Lemma \ref{lem:log_matrix}.
We fix {$h>1$,} $T>0$, and let $(X^{\eps,\delta}_t)_{\e,\d\in\mathbb{C},\, t\in[0,T]}$ be the
solution of the SDE \eqref{eq:SDE_linear_eps} as defined in Lemma \ref{lem:preliminary}. Moreover,
for
$t\in\,]0,T]$% and $h>1$
, we set $Q_{t,h}\coloneqq\, ]0,t[ \times B_{h}(0)$ where $B_{h}(0)=\{(\e,\d)\in\mathbb{C}^{2}\mid |(\e,\d)|<h\}$.%,
%with $B$ being an open ball in $\mathbb{C}^2$ centered at $(0,0)$ with radius strictly greater
%than one.
%We also let $(X^{\eps,\delta})_{(\eps,\delta)\in B_2(0)}$ be a family of modifications of the
%solutions to \eqref{eq:SDE_linear_eps} satisfying the properties a), b), c) and d) in Lemma
%\ref{lem:preliminary}-(ii).

\vspace{5pt}

\noindent\underline{Part (i)}: as $X^{\eps,\delta}_0=I_d$, by continuity the random time defined
as
\begin{equation}\label{eq:def_tau}
 \tau \coloneqq \sup \left\{t\in[0,T] \left| \fnorm{ X^{\eps,\delta}_s - I_d} < 1-e^{-\pi} \text{ for any } (s,\eps,\delta)\in Q_{t,h}\right.\right\}
\end{equation}
is strictly positive. Furthermore, again by continuity,
\begin{equation}
 (\tau \leq t) = \bigcup_{ %\substack{
% (s,\eps,\delta)\in ([0,t] \times B_2(0) )\cap \mathbb{Q}^5% \\ (\eps,\delta)\in B_2(0)  \cap \mathbb{Q}^2}
 (s,\eps,\delta)\in \tilde{Q}_{t,h}% \\ (\eps,\delta)\in B_2(0)  \cap \mathbb{Q}^2}
 }  \left( \fnorm{  X^{\eps,\delta}_s -I_d } \geq 1-e^{-\pi }\right),   \qquad t\in[0,T],
\end{equation}
where $\tilde{Q}_{t,h}$ is a countable, dense subset of $Q_{t,h}$, which implies that $\tau$ is a
stopping time.

%there exists a random time $\tau$ such that the cylinder $H_{\tau}:= ]0,\tau[ \times B_{2}(0)$, with $B_{2}(0)$ being the ball in $\mathbb{C}^2$ centered at $(0,0)$ with radius $2$, is such that
%\begin{equation}
% H_{\tau} \subset \{(t,\eps,\delta)\in \R\times\mathbb{C}^2 \text{ such that } \big| X^{\eps,\delta}_t - I_d  \big| < 1 \}.
%\end{equation}
Let $(t,\eps,\delta)\in Q_{\tau,h}$: by Lemma \ref{lem:log_matrix} applied to
$M=X^{\eps,\delta}_t$ we have
\begin{align}\label{eq:Omega_eps_delta}
Y_t^{\eps,\delta} &\coloneqq \log X_t^{\eps,\delta}  =  \sum_{n=1}^{\infty} (-1)^{n+1}
\frac{\big(X_t^{\eps,\delta}-I_d\big)^n}{n} \notag\\&= \big(X_t^{\eps,\delta} - I_d\big)
\int_0^{\infty} \frac{1}{1+\mu} \big(\mu I_d + X_t^{\eps,\delta}\big)^{-1} \dd \mu.
%, \qquad (t,\eps,\delta)\in
%H_{\tau}.
\end{align}
%Also, since $X^{\eps,\delta}_t$ is real for $\eps,\delta\in\R$, then also $Y^{\eps,\delta}_t$ is
%real for any $(t,\eps,\delta)\in [0,\tau[\times \big(B_2(0)\cap \R^2\big)$. In particular, $Y_t =
%Y_t^{1,1}$ is real, and this proves Part (i).
Notice that $X^{\eps,\delta}_t$ (and therefore also $Y^{\eps,\delta}_t$) is real for
$\eps,\delta\in\R$: in particular, $Y_t = Y_t^{1,1}$ is real and this proves Part (i).

\vspace{5pt}

\noindent\underline{Part (ii)}: since $(\eps,\delta)\mapsto X_t^{\eps,\delta}$ is holomorphic, we
can differentiate \eqref{eq:Omega_eps_delta} to infer that $(\eps,\delta)\mapsto
Y_t^{\eps,\delta}$ is holomorphic as well: indeed, we have for $(t,\eps,\delta)\in Q_{\tau,h}$
\begin{align}%\label{eq:Omega_eps_delta}
\partial_{\eps}Y_t^{\eps,\delta} & = \partial_{\eps} X_t^{\eps,\delta}  \int_0^{\infty} \frac{1}{1+\mu} \big(\mu I_d + X_t^{\eps,\delta}\big)^{-1} \dd \mu  \\
&\quad + \big(X_t^{\eps,\delta} - I_d\big)  \int_0^{\infty} \frac{1}{1+\mu} \big(\mu I_d +
X_t^{\eps,\delta}\big)^{-1}  \big( \partial_{\eps}X_t^{\eps,\delta}  \big) \big(\mu I_d +
X_t^{\eps,\delta}\big)^{-1}   \dd \mu ,
\end{align}
and similarly by differentiating w.r.t. to $\delta$. Then the expansion of $Y_t^{\eps,\delta}$ in
power series at $(\eps,\delta)=(0,0)$ is absolutely convergent on $B_{h}(0)$ and the
representation \eqref{eq:expansion_Yepsdel} holds on $Q_{\t,h}$
%for
%any $t\in [0,\tau[$ and $(\eps,\delta)\in B $,
%\begin{equation}\label{eq:power_series}
%Y_t^{\eps,\delta} = \sum_{n=0}^{\infty} \sum_{r=0}^{n} Y^{(r,n-r)}_t \eps^{r} \delta^{n-r}, \qquad t<\tau,\quad (\eps,\delta)\in B .
%\end{equation}
for some random coefficients $Y^{(r,n-r)}_t$. To conclude we need to show that the latter are as
given by \eqref{eq:represent_Y_general}-\eqref{eq:sigma_general}-\eqref{eq:mu_general}. Then
\eqref{eq:convergence} will stem from \eqref{eq:expansion_Yepsdel} by setting
$(\eps,\delta)=(1,1)$.
%\begin{equation}
%\sum_{n=0}^{\infty} \sum_{r=0}^{n} Y^{(r,n-r)}_t \eps^{r} \delta^{n-r}
%\end{equation}
%with $Y^{(r,n-r)}$

%Set the family of events
%\begin{equation}
%F_n := \big\{ \|Y^{\eps,\delta}_t \| <\pi \text{ for any } (t,\eps,\delta)\in [0,\tau[\times B_{{1}/{n}}(0)    \big\}  , \qquad n \in \N.
%\end{equation}
%By the continuity of $(t,\eps,\delta)\mapsto Y_t^{\eps,\delta}$ together with $Y^{0,0}\equiv 0$ $P$-a.s., we have
%\begin{equation}
%F_n \nearrow, \qquad P\big(\cup_{n\in\N} F_n \big)  = 1.
%\end{equation}
In light of %Lemma \ref{eq:lemm_diff}, Lemma \ref{lem:Baker} and
Lemma \ref{lem:log_matrix}, the logarithmic map is continuously twice differentiable on the open
subset of $ \mathcal{M}_{\mathbb{C}}^{d\times d}$ of the matrices $M$ such that $\snorm{ M - I_d } < 1
$: thus $Y_{t}^{\eps,\delta}$ admits an It\^o representation \eqref{eq:general_ito} for
$(t,\eps,\delta) \in Q_{\t,h}$. Then Proposition \ref{prop:Ito 1} together with
\eqref{eq:SDE_linear_eps} yield \eqref{eq:rep_A}-\eqref{eq:rep_B} $P$-a.s. up to $\tau$ {for any
$(\eps,\delta)\in B_h(0)\cap \R^2$}. Furthermore, by estimate \eqref{estimate_log} of Lemma
\ref{lem:log_matrix} we also have $\snorm{Y_{t}^{\eps,\delta}}<\pi$ for $t<\tau$. Therefore, we can
apply Baker's Lemma \ref{lem:Baker} to invert $\dD_{Y^{\eps,\delta}_t}$ in
\eqref{eq:rep_A}-\eqref{eq:rep_B} and obtain that
{$Y^{\eps,\delta}$ solves \eqref{eq:SDE_linear_bbb} up to $\tau$ for any $(\eps,\delta)\in B_h(0)\cap \R^2$}. %$Y_{t}^{\eps,\delta}$ solves \eqref{eq:SDE_linear_bbb} for $(t,\e,\d)\in Q_{\t,h}$. %up to time $\tau$, for any $(\eps,\delta) \in B_{h}(0)\cap \R^2$.
Part (ii) then follows from Lemma \ref{lem:verification}.

%Since the latter coincides with $\sum_{n=0}^{\infty} Y^{(n)}_t$ for $(\eps,\delta)=(1,1)$, this proves \eqref{eq:convergence} and concludes the proof.

\vspace{5pt}

\noindent\underline{Part (iii)}: for $t\le T$ let
  $$%\sup_{\substack{\eps,\delta\in B  \\ 0\leq s\leq t}}|X^{\eps,\delta}_{s} - I_d|,\qquad
  f_t(\eps,\delta)\coloneqq \max_{ s\in[0,t] }\fnorm{X^{\eps,\delta}_{s} - I_d},\qquad M_{t}\coloneqq\sup_{(\eps,\delta)\in
  B_{h}(0)}f_t(\eps,\delta).
  %\qquad 0<t\le T.
  $$
By definition \eqref{eq:def_tau}, we have%\red{(the first is an equality or inequality?)}
  \begin{equation}\label{eq:mark_property}
 P (\tau \leq t)\le P\left(M_{t}\geq 1-e^{-\pi}\right)\le \frac{1}{\left(1-e^{-\pi}\right)^2}{E\left[M^2_{t}\right]},
\end{equation}
and therefore \eqref{eq:estimate_tau_conv_b} follows by suitably estimating
$E\left[M^2_{t}\right]$. To prove such an estimate we will show in the last part of the proof that
$f_t$ belongs to the Sobolev space $W^{1,2 p}(B_{h}(0))$ for any $p\geq 1$ and we have
\begin{equation}\label{eq:sobolev}
 E\big[ \|  f_t   \|^{2 p}_{W^{1,2p}(B_{h}(0))}  \big]  \leq   %{\blue d^p}
 C  t^p,% \red{e^{C t}},
 \qquad t\in[0,T],
\end{equation}
where the positive constant $C$ depends only on $\|A^{(1)}\|_{T},\dots,\|A^{(q)}\|_{T}$,
$\|B\|_{T}$, $d$, $T$, $h$ and $p$. Since $f_t\in W^{1,2 p}(B_{h}(0))$ and
$B_{h}(0)\subseteq\mathbb{R}^{4}$, by Morrey's inequality (cf., for instance, Corollary 9.14 in
\cite{MR2759829}) for any $p>2$ we have
\begin{equation}\label{ae5}
  M_{t}\le c_{0} \|f_t\|_{W^{1,2p}(B_{h}(0))},
\end{equation}
where $c_{0}$ is a a positive constant, dependent only on $p$ and $h$ (in particular, $c_{0}$ is
independent of $\o$).
%Choosing $p>2$ %($\mathbb{C}^{2}$ is isomorphic to $\R^4$)
%Applying Morrey's inequality on $B_{h}(0)\subseteq\mathbb{R}^{4}$ with $p>2$ (cf., for instance,
%Corollary 9.14 in \citep{MR2759829}) and Jensen's inequality we obtain
Combining \eqref{eq:sobolev} with \eqref{ae5}, for a fixed $p>2$ we have
\begin{align}
 E\left[M^2_{t}\right] &\leq c^{2}_{0} E\big[ \|f_t\|^{2}_{W^{1,2p}(B_{h}(0))}  \big]\le
\intertext{(by H\"older inequality)}
 &\leq %{\blue d}\,
 c^{2}_{0}C t,\qquad t\in[0,T].
\end{align}
%for some positive constant $C$ that depends only on $h$ and $p$.
This last estimate, combined with \eqref{eq:mark_property}, proves \eqref{eq:estimate_tau_conv_b}. % stems from applying Jensen inequality and combin

To conclude, we are left with  the proof of \eqref{eq:sobolev}. First we have
\begin{equation}\label{eq:estimate_L2p}
E\bigg[ \int_{B_{h}(0)} | f_t(\eps,\delta) |^{2p}  \dd \eps\, \dd \delta \bigg] %&
 = \int_{B_{h}(0)} E\big[ | f_t(\eps,\delta) |^{2p} \big] \dd \eps\, \dd \delta
%\intertext{(by the estimate in Lemma \ref{lem:preliminary}-(i))}
%&
\leq %{\blue d^p}
 C%_p
  %e^{C t}}
  t^p,
\end{equation}
where we used the estimate \eqref{eq:estimates_Xepsdel} of Lemma \ref{lem:preliminary} in the last
inequality.
Fix now $t\in\,]0,T]$, $(\eps,\delta),(\eps',\delta')\in B_{h}(0)  %\cong \R^4
$ such that $f_t(\eps',\delta')\leq f_t(\eps,\delta)$ and set
\begin{equation}
  \bar{t} \in\underset{0\leq s\leq t}{\arg\max}  \fnorm{X^{\eps,\delta}_{s} - I_d},\qquad \tilde{t}
  \in\underset{0\leq s\leq t}{\arg\max}  \fnorm{X^{\eps',\delta'}_{s} - I_d} .
\end{equation}
Note that the $\arg\max$ above do exist in that the process $g_s(\eps,\delta)\coloneqq
X^{\eps,\delta}_{s} - I_d$
is continuous in $s$ %Denoting by $g_t(\eps,\delta):= X^{\eps,\delta}_{s}(\omega) - I_d$,
and we have
\begin{align*}
\abs{f_t(\eps,\delta)-f_t(\eps',\delta')}
&= \abs{ \fnorm{g_{\bar{t}}(\eps,\delta)} - \fnorm{g_{\tilde{t}}(\eps',\delta') } }%{|(\eps,\delta)-(\eps',\delta')|}
\leq{\abs{ \fnorm{g_{\bar{t}}(\eps,\delta)} - \fnorm{g_{\bar{t}}(\eps',\delta')} } }%{|(\eps,\delta)-(\eps',\delta')|}
 \\&\leq  \fnorm{g_{\bar{t}}(\eps,\delta) - g_{\bar{t}}(\eps',\delta')}
  \leq \sup_{0\leq s\leq t} \fnorm{g_{s}(\eps,\delta) - g_{s}(\eps',\delta')}
    \\&\leq \enorm{(\eps,\delta)-(\eps',\delta')} \sup_{0\leq s\leq t}
  \sup_{\substack{ |\bar{\eps} - \eps | \leq | \eps' - \eps|   \\  |\bar{\delta} - \delta | \leq | \delta' - \delta|   }}   \fnorm{\nabla    %_{\eps,\delta}
  g_{s}(\bar{\eps},\bar{\delta}) },
\end{align*}
where $\nabla=\nabla_{\!\e,\d}$.
%for some $\bar{\eps},\bar{\delta}\in\mathbb{C}^2$ such that $|\bar{\eps} -\eps|\leq |\eps' -\eps|$ and $|\bar{\delta} -\delta |\leq | \delta' -\delta |$.
This, as $(s,\eps,\delta)\mapsto \nabla    %_{\eps,\delta}
  g_{s}({\eps},{\delta})$ is continuous on $Q_{t,h}$, implies $f_t\in W^{1,2 p}(B_{h}(0))$ and yields the key inequality
\begin{equation}
 \enorm{ \nabla f_t(\eps,\delta) } \leq \sup_{ 0\leq s\leq t } %\big| \nabla ( X^{\eps,\delta}_{s}(\omega) - I_d ) \big|,
 \fnorm{ \nabla X^{\eps,\delta}_{s} },
 \qquad (\eps,\delta)\in B_{h}(0). %{\blue\qquad !!!}
\end{equation}
Therefore, we have
%\begin{equation}
%E\bigg[ \int_{B } | \partial_{u_{\eps}} f_t(\eps,\delta) |^{2p}  \dd \eps\, \dd \delta \bigg] = \int_{B } E\big[ | \partial_{u_{\eps}} f_t(\eps,\delta) |^{2p} \big] \dd \eps\, \dd \delta
%\leq   \int_{B } E\bigg[ \sup_{ 0\leq s\leq t }|\partial_{u_{\eps}} X^{\eps,\delta}_{s} |^{2p} \bigg]   \dd \eps\, \dd \delta \leq {{\blue d}\,C t\, e^{C t}},
%\end{equation}
\begin{align}
 E\bigg[ \int_{B_{h}(0)} \enorm{ \nabla f_t(\eps,\delta) }^{2p}  \dd \eps\, \dd \delta \bigg] &=
 \int_{B_{h}(0)} E\big[ \enorm{ \nabla f_t(\eps,\delta) }^{2p} \big] \dd \eps\, \dd \delta
 \\&\leq   \int_{B_{h}(0)} E\bigg[ \sup_{ 0\leq s\leq t }\fnormp{\nabla X^{\eps,\delta}_{s} } \bigg]   \dd \eps\, \dd \delta \leq %{\blue d^p}\,
 C t^p,
\end{align}
where we used the estimate \eqref{eq:derestimates_Xepsdel} of Lemma \ref{lem:preliminary} in the
last inequality.
%Similar arguments yield the same estimates for
%\begin{equation}
%E\bigg[ \int_{B } | \partial_{v_{\eps}} f_t(\eps,\delta) |^{2p}  \dd \eps\, \dd \delta \bigg], \ E\bigg[ \int_{B } | \partial_{u_{\delta}} f_t(\eps,\delta) |^{2p}  \dd \eps\, \dd \delta \bigg],\ \text{and }  E\bigg[ \int_{B } | \partial_{v_{\delta}} f_t(\eps,\delta) |^{2p}  \dd \eps\, \dd \delta \bigg].
%\end{equation}
%Such estimates
This, together with \eqref{eq:estimate_L2p}, proves \eqref{eq:sobolev} and conclude the proof.
\end{proof}

%\newpage

%The following statement holds.%, which is a recursive representation for the general term $Y^{(i,j)}$ in \eqref{eq:expansion_Yepsdel}.

\section{Numerical tests {and applications to SPDEs}}\label{sec:num_tests}
We present here some numerical tests in order to confirm the accuracy of the approximate solutions to \eqref{eq:SDE_linear_b}
stemming from the truncation of %such an
the series \eqref{eq:convergence}. We also show how this approximation can be applied to approximate the solutions to stochastic partial differential equations (SPDEs) of parabolic type.

We consider two examples {of SDEs (one in Section \ref{sec:ABconstant} and one in Section \ref{subsec:Auppertriang})}, for which we compute the first three terms of the ME given by \eqref{eq:stoch_magnus_exp}-\eqref{eq:represent_Y_general} and present numerical experiments to test the accuracy of the approximate solutions to \eqref{eq:SDE_linear_b} stemming from it. In both cases we consider $j=1$ in \eqref{eq:SDE_linear_b} and replace $A^{(1)}$ with $A$ to shorten notation. The first example will be for %deterministic and
constant matrices $ A$ and $B$. %in \eqref{eq:SDE_linear_b} with $j=1$. % and $d=2$.
%The second one will concern itself with the pure stochastic case, i.e.
In the second one we will consider $B\equiv 0$ and {a} deterministic upper diagonal $A_t$. % a deterministic upper diagonal matrix.
%one deterministic and time dependent two times two upper-triangular matrix $A_t$.
 For each numerical test we will implement the exponential of the truncated ME up to order $n=1,2$ and $3$, i.e.
 \begin{equation}\label{eq:magnus_approx}
 X^{(n)}\coloneqq e^{\sum_{i=1}^{n}Y^{(i)}}, \qquad n=1,2,3,
\end{equation}
and compare it with a benchmark solution to \eqref{eq:SDE_linear_b}. %In particular, we will plot the trajectories of one component of the solution, and provide a table with different type of errors.
In Section \ref{sec:par_spdes} we turn our attention to the application of the ME to the numerical resolution of SPDEs. In particular, in the numerical tests we will make use of the ME for constant matrices discussed in Section \ref{sec:ABconstant}.%
%present some numerical tests in the case of constant coefficients, and thus we will empl
%After the first example of an SDE with constant coefficients, we will show how one can apply this special case of the ME to solving SPDEs numerically.
%The section will be structured as follows: First of all, we fix some notations and define the errors we consider in the numerical experiments, then there will be the two examples. For each example we {\blue show the expansion formulas and} define the necessary parameters for the numerical methods. After that there will be a plot of one solution in the solution matrix for one trajectory over all times and a table with some errors for each method we use.
\paragraph*{Error and notations.$\ {}$}
Throughout this section we will employ the following tags:
\begin{itemize}
\item[-] \verb+euler+ for the solution obtained with Euler-Maruyama scheme,
which was implemented with Matlab's pagefun for the matrix multiplication on a single GPU and
vectorized over all samples;
\item[-] \verb+exact+ to denote the time-discretization of an explicit solution, if available;
\item[-] \verb+m1+,
\verb+m2+ and \verb+m3+ for the time-discretization of the Magnus approximations in \eqref{eq:magnus_approx}, up to order 1,2 and 3, respectively. %Namely, \verb+Mn+ refers to the time-discretization of $e^{Y^{(1)}_t + \cdots + Y^{(n)}_t}$.
\end{itemize}
%If available, we will use explicit solutions as a benchmark, otherwise the Euler-Maruyama scheme. The naming of the methods in the pictures and tables will be \verb+exact+ if an explicit solution is available, \verb+euler+ for the Euler-scheme, \verb+m1+,
%\verb+m2+ and \verb+m3+ for the Magnus scheme of order 1,2 and 3, resp.

For the {numerical} error analysis {in the SDE examples }we will make use of the following norms. Denoting by $X^{\text{ref}}$ and
by $X^{\text{\text{app}}}$ a benchmark and an approximate solution, respectively, to
\eqref{eq:SDE_linear_b} and by {$\left(t_k\right)_{k=0,\dots,N}$} a homogeneous discretization of $[0,t]$, %and by $(\omega_m)_{m\in M}$ a family of independent realizations of the discretized Brownian trajectories,
we consider the random variable% an approximate solution.
                \begin{align}  %\hspace{-15pt}
            \text{Err}_t \coloneqq     \frac{\Delta}{t}
                    \sum_{k=1}^N
                            \frac{\fnorm{ X^{\text{ref}}_{t_k}- X^{\text{\text{app}}}_{t_k}} }{\fnorm{ X^{\text{ref}}_{t_k}}}\approx
             \frac{1}{t}
                \int_{0}^{t}
                % \inf_{\omega\in\Omega :  | X^{\text{ref}}_s | > 0.1}
                {
                     \frac{
                        \fnorm{ X^{\text{ref}}_s  -  X^{\text{\text{app}}}_s   }}{\fnorm{ X^{\text{ref}}_s }}
                }  %\bigg| \big| X^{\text{ref}}_s \big| > 0.1
             \dd s  \qquad  \text{with $\Delta=\frac{t}{N}$} ,
                \end{align}
namely a discretization of the time-averaged relative error on the interval $[0,t]$. This is a way to measure the error on the whole trajectory as opposed to the error at a specific given time. Then we use Monte Carlo simulation, with $M$ independent realizations of the discretized Brownian trajectories, to approximate the distribution of $\text{Err}_t$. %expected value of :

The matrix norm above is the Frobenius norm. %, i.e. Euclidean entry-wise norm.
In the following tests, \verb+m1+, \verb+m2+ and \verb+m3+ will always play the role of
$X^{\text{\text{app}}}$, \verb+exact+ always the role of $X^{\text{ref}}$, whereas \verb+euler+
will be either $X^{\text{\text{app}}}$ or $X^{\text{ref}}$ depending on whether \verb+exact+ is
available or not.
%The stochastic ME will by its very nature be close to the exact solution at the start and will lose accuracy at terminal time. Therefore, a consideration of plain minimal and maximal errors would most likely only reflect in the errors at the start and end of the time line. Hence, we decided to choose errors for comparison, which take all times into account and will be averaged over all trajectories to avoid a favoritism of an exceptionally good or bad path in terms of differing from the exact solution.
%Conclusively, we choose the $L^1$ mean from above for the absolute errors and to be consistent we defined the relative errors in the same manner.

We used for the calculations {\verb+Matlab R2021a+} with Parallel Computing Toolbox running on
Windows 10 Pro, on a machine with the following specifications: processor Intel(R) Core(TM)
i7-8750H @ 2.20\,GHz,  2x32 GB (Dual Channel) Samsung SODIMM RAM @ 2667\,MHz, and a NVIDIA GeForce
RTX 2070 with Max-Q Design (8\,GB GDDR6 RAM).
%{\color{magenta} We used for the calculations \verb+Matlab R2019a+ with Parallel Computing Toolbox on a Gigabyte Aero-15-X9 Laptop with Windows 10 Pro with the following specifications: processor Intel(R) Core(TM) i7-8750H @ 2.20\,GHz,  2x16 GB (Dual Channel) Samsung SODIMM RAM @ 2667\,MHz, and a NVIDIA GeForce RTX 2070 with Max-Q Design (8\,GB GDDR6 RAM).
%}
Also, we will make use of the Matlab built-in routine \verb+expm+ for the computation of the
matrix exponential. As it turns out, this represents the most expensive step in the implementation
of the Magnus approximation. However important, the pursue of optimized method for the matrix
exponentiation is an extended topic of separate interest, which goes beyond the goals of this
paper. Therefore, here we will limit  {ourselves} to pointing out, separately, the computational
times for the approximations of the logarithm and of the matrix exponential.

In the implementation we simulate the Brownian motion first and use it as an input for each scheme
to be able to compare the trajectories of each scheme amongst each other.

%{\color{blue} [Andrea: Sure we want to say the following?]}\\ {\color{red} In most of the
%numerical examples we will be limited to using only 1000 samples due to limiting hardware
%components, i.e. 8\,GB RAM on the GPU. \sout{The method causing this limitation is
%\protect\UseVerb{euler} out of necessity of using a finer time grid compared to the Magnus methods
%to achieve the same accuracy. This can be explained by the theoretical convergence rate of the
%Euler-Maruyama scheme, which is of order one-half, while the convergence of the numerical
%deterministic integration, using standard Riemann-sums, is of order one, which is the method of
%our choice in the Magnus methods.}}
% : time for the discretized  %because we do not want to address in these experiments problems resulting from implementing a numerical method for computing the matrix exponential.

%{\blue The titles in plots are generated automatically by the \verb+Matlab+ code and look like
%$X(i,j,t,\omega)$, where $i,j$ is the position in the solution matrix $X_t$. For
%$t$ there will be $:$, which means, that the plot will be shown for all times and
%for $\omega$ there will be a number indicating which trajectory of the solution we are considering, e.g. $X(1,2,:,1)$ will be the first trajectory of the entry in the first row and second column over all times. (Any hope to remove $X(1,2,:,1)$ from the plots altogether?)}
\vspace{5pt}
\subsection{Example: constant $A$ and $B$.}\label{sec:ABconstant}

%%%%%%%%%%%%%%%%%%%%%%%%%%%%%%%%%%%%%%%%%%%%%%%%%%%%%%%%%%%%%%%%%%%%%%%%%%%%%%%%%%%%%%%%%%%%%%%
%\input{Numerics/AB_const/AB_const_commands.tex}
%\subparagraph*{Expansion formulas.}
 With a slight abuse of notation, we consider $A_t\equiv A$ and $B_t\equiv B$. Recall that, if $A$
and $B$ do not commute, there is in general no closed-form solution to \eqref{eq:SDE_linear_b}.
The first three terms of the ME read as
\begin{align}
    \Mlog_t^{(1)}&=B t + A W_t,\qquad
    \Mlog_t^{(2)}=\comm{A}{B}\left(\frac{1}{2}tW_t-\int_{0}^{t}{W_s ds}\right)
    -\frac{1}{2}A^2 t,\\
    \Mlog_t^{(3)}&=\comm{\comm{B}{A}}{A}\bigg(
            \frac{1}{2} \int_{0}^{t}{W_s^2 ds}
            -\frac{1}{2} W_t \int_{0}^{t}{W_s ds}
            +\frac{1}{12} tW_t^2 \bigg) \\
            &\quad
            +\comm{\comm{B}{A}}{B}\bigg(
             \int_{0}^{t}{sW_s ds}
             -\frac{1}{2}  t\int_{0}^{t}{W_s ds}
            -\frac{1}{12} t^2 W_t\bigg). \label{eq:ab_const_y3}
%\\ \Mlog_t^3&=
%           -\frac{1}{2}\comm{\comm{B}{A}}{A} W_t \int_{0}^{t}{W_s ds}
%           +\frac{1}{2}\comm{\comm{B}{A}}{A} \int_{0}^{t}{W_s^2 ds}
%           +\frac{1}{12}\comm{\comm{B}{A}}{A} tW_t^2 \\
%           &\quad
%           -\frac{1}{2} \comm{\comm{B}{A}}{B} t\int_{0}^{t}{W_s ds}
%           +\comm{\comm{B}{A}}{B} \int_{0}^{t}{sW_s ds}
%           -\frac{1}{12}\comm{\comm{B}{A}}{B} t^2 W_t.
\end{align}
%%%%%%%%%%%%%%%%%%%%%%%%%%%%%%%%%%%%%%%%%%%%%%%%%%%%%%%%%%%%%%%%%%%%%%%%%%%%%%%%%%%%%%%%%%%%%%%
%\subparagraph*{Numerical test.}
We point out that, in this case, all the stochastic integrals appearing in the ME can be solved in terms of Lebesgue integrals by using It\^o's formula. Therefore, in order to discretize $Y^{(n)}$ it is not necessary to approximate stochastic integrals. This allows to use a sparser time grid compared to the Euler method, for which the discretization of stochastic integrals is necessary. {In particular, the theoretical speed of convergence with respect to the time-step is of order $\sqrt{\Delta}$ for Euler-Maruyama scheme and of order $\Delta$ for deterministic Euler, which is the scheme used to discretize the Lebesgue integrals in the Magnus expansion above.} {In the following numerical tests, we discretize in time with mesh $\Delta$ equal to $10^{-4}$ for \verb+euler+ and equal to
$\sqrt{\Delta}=10^{-2}$ for \verb+m1+, \verb+m2+ and \verb+m3+. Note that, as it is confirmed by the results in Table \ref{tab:AB_const}, choosing a finer time-discretization for \verb+euler+ (our reference method here) is essential in order to make it comparable with \verb+m3+.} {Furthermore, in the example of Section \ref{subsec:Auppertriang}, where an explicit solution is available, we show (see Tables \ref{tab:B0_fix} and \ref{tab:B0_var}) that choosing a sparser time-grid (say $\Delta=10^{-3}$) the Euler-Maruyama method incurs a sensitive loss of precision.}

It is also clear that the implementation is totally parallelizable, in that $Y^{(1)}$,
$Y^{(2)}$ and $Y^{(3)}$ do not depend on each other and thus they can be computed in parallel. More importantly, the discretization of the integrals in each $Y^{(n)}$ can be parallelized {as the latter are explicit and not implicitly defined through a differential equation}.

%{In particular, in this formula one can see clearly, that for a numerical implementation it is totally parallelizable: Each term $Y^{(1)}$, $Y^{(2)}$ and $Y^{(3)}$ does not depend on each other and can be computed in parallel, as well as in each term all times and simulations can be computed in parallel.}
 %, because we want our reference method for the error analysis to be more accurate, than the Magnus methods. We will see in Table \ref{tab:AB_const} that for small times the error of the third order stochastic ME to the \verb+euler+ method is approximately 0.13\,\%, which confirms the need to use a finer time-discretization for the \verb+euler+ method.}
%Also, w
We choose $A$ and $B$ at random and normalize them by their spectral norms. In particular,
the results below refer to
\begin{equation}
A= \left( \begin{matrix} 0.335302 & -0.645492 \\ -0.264419 & 0.634641
\end{matrix}\right), \qquad
B= \left( \begin{matrix} -0.0572262 & 0.0493763 \\ -0.665366 & 0.742744
\end{matrix}\right).
\end{equation}
%\begin{align*}
%   &A=
%   \ABconstA
%   &&
%   &B=
%   \ABconstB.
%\end{align*}
%Also, we discretize in time with mesh $\Delta$ equal to $10^{-4}$ for \verb+euler+ and equal to $10^{-3}$ for \verb+m1+, \verb+m2+ and \verb+m3+. {\blue Note we use a smaller $\Delta$ for \verb+euler+ compared to the MEs. This is because the Euler method requires a smaller time-step in order to reach a good accuracy level, due the discretization of the stochastic integrals.}
In Figure \ref{fig:AB_const} we plot one realization of the trajectories of the top-left component
$(X_t)_{11}$, computed with the methods above, up to time $t=0.75$.
\begin{figure}
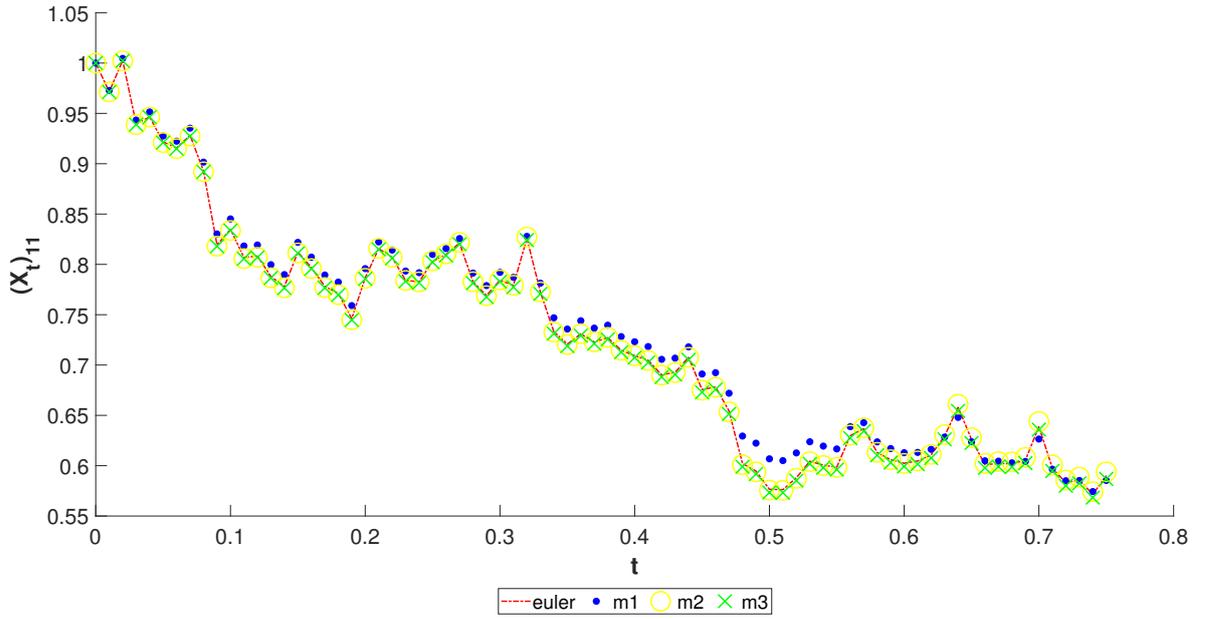
%
\ABconstTrajectory%
\caption{$A$ and $B$ constant. One realization of the trajectories of the top-left component $(X_t)_{11}$, computed with \protect\UseVerb{euler}, \protect\UseVerb{m1}, \protect\UseVerb{m2}, \protect\UseVerb{m3}.}%
\label{fig:AB_const}%
\end{figure}
In Table \ref{tab:AB_const} we show the expectations $E[\text{Err}_t]$ for different values of
$t$, with \verb+euler+ as benchmark solution, computed via Monte Carlo simulation with $10^{3}$
samples. The same samples are used in Figure \ref{fig:AB_constcds} %\ref{fig:cdf1}
to plot the empirical CDF of $\text{Err}_t$.
%corresponding to the samples used to compute the expectations.
\begin{table}
\centering
    \caption{$A$ and $B$ constant. Values of $E[\text{Err}_t]$ (in percentage) for \protect\UseVerb{m1}, \protect\UseVerb{m2}, \protect\UseVerb{m3}, with \protect\UseVerb{euler} as benchmark solution, obtained with $10^{3}$ samples.
}
\begin{tabular}{ccccccc}
Method               & $t=0.25$    & $t=0.5$     &  $t=0.75$   &  $t=1$          & $t=2$
&  $t=3$\\
\hline
\multicolumn{7}{c}{\bf Euler $\Delta=10^{-4}$, Magnus $\Delta=10^{-2}$}\\
\protect\UseVerb{m1} & $4.59\,\%$  & $9.6\,\%$  & $14.8\,\%$  & $21.1\,\%$  & $49.5\,\%$
& $86.1\,\%$   \\
\protect\UseVerb{m2} & $0.217\,\%$ & $0.503\,\%$ & $0.951\,\%$ & $1.55\,\%$   & $5.29\,\%$
& $10.9\,\%$  \\
\protect\UseVerb{m3} & $0.176\,\%$ & $0.256\,\%$ & $0.371\,\%$ & $0.543\,\%$ & $2.03\,\%$
& $5.25\,\%$ %\\
%\hline
%\multicolumn{7}{c}{\bf Euler $\Delta=10^{-3}$, Magnus $\Delta=10^{-2}$}\\
%\protect\UseVerb{m1} & $4.55\,\%$  & $9.48\,\%$  & $15.1\,\%$  & $21\,\%$  & $48.5\,\%$
%& $87.9\,\%$   \\
%\protect\UseVerb{m2} & $0.454\,\%$ & $0.759\,\%$ & $1.28\,\%$ & $1.87\,\%$   & $5.36\,\%$
%& $10.7\,\%$  \\
%\protect\UseVerb{m3} & $0.43\,\%$ & $0.588\,\%$ & $0.795\,\%$ & $0.993\,\%$ & $2.44\,\%$
%& $5.45\,\%$
\end{tabular}
\label{tab:AB_const}
\end{table}
%\UseVerb{term}
%In the case $t\in [0,\ABconstT]$
%with $\ABconstNfine$ time-steps and samples for the Euler-method as a reference and
%$\ABconstN$
%time-steps and samples for all Magnus-orders we have the Figure \ref{fig:AB_const}\\
%The corresponding errors for the 4 different solutions in the solution matrix are in \ref{tab:AB_const}.
%\begin{table}%
%   \ABconsttotalerrtab
%   \centering
%\caption{$A$ and $B$ constant: values of $\Eb[\text{Err}_t]$ (in percentage) for \protect\UseVerb{m1}, \protect\UseVerb{m2}, \protect\UseVerb{m3}, using \protect\UseVerb{euler} as benchmark solution.
%}
%\label{tab:AB_const}
%\end{table}
\begin{figure}
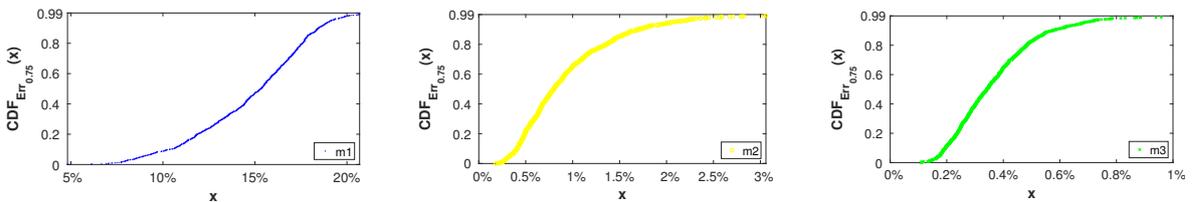
\label{fig:cdf1}%
    \begin{center}
    \begin{minipage}[c][][c]{.30\linewidth}
        \ABconstCDSmone%
    \end{minipage}\hfill
    \begin{minipage}[c][][c]{.30\linewidth}
        \ABconstCDSmtwo%
    \end{minipage}\hfill
    \begin{minipage}[c][][c]{.30\linewidth}
        \ABconstCDSmthree%
    \end{minipage}%
    \end{center}
\caption{$A$ and $B$ constant. Empirical CDF of $\text{Err}_t$, at $t=0.75$, for
\protect\UseVerb{m1}, \protect\UseVerb{m2}, \protect\UseVerb{m3}, with \protect\UseVerb{euler} as
benchmark solution, obtained with $10^{3}$ samples.}
\label{fig:AB_constcds}%
\end{figure}
{The computational times for the $10^3$ sampled trajectories of $X$, up to time $t=1$, computed with \protect\UseVerb{m1}, \protect\UseVerb{m2}, \protect\UseVerb{m3} and \protect\UseVerb{euler} are reported in Table \ref{tab:ctimes_AB_const_T1_d2_N1000_M1000}. For the Magnus methods we separate the time to compute the approximate logarithm from the one to compute the matrix exponential.}
%{\color{red}\sout{
%which is given in
%Table \ref{tab:ctimes_AB_const_T1_d2_N1000_M1000} in more detail, is roughly
%$7$
%seconds
%for \verb+euler+ and $0.6$ seconds for either \verb+m1+, \verb+m2+ or \verb+m3+. The latter,
%however, is divided as follows: nearly $0.03$ seconds to compute the %\blu{\sout{approximate logarithm}
%ME %}
%and nearly $0.56$ seconds to compute the matrix exponential with the Matlab function \verb+expm+.
%}}
\begin{table}[!ht]
\centering
    \caption{$A$ and $B$ constant. Computational times for $10^3$ sampled trajectories of $X$, up to time
    $t=1$, computed with \protect\UseVerb{m1}, \protect\UseVerb{m2}, \protect\UseVerb{m3} and \protect\UseVerb{euler}.
}
\begin{tabular}{*{4}{c}}
\text{Method} &\text{Log} &\text{Matrix Exp} &\text{Total}\\
\hline
\protect\UseVerb{euler} $\Delta=10^{-4}$ & 0 & 0 & 6.7784 \\
\protect\UseVerb{m1} $\Delta=10^{-2}$ & 0.0093466 & 0.535331 & 0.544678 \\
\protect\UseVerb{m2} $\Delta=10^{-2}$ & 0.0221759 & 0.569511 & 0.591687 \\
\protect\UseVerb{m3} $\Delta=10^{-2}$ & 0.0475184 & 0.584098 & 0.631616 \\
\end{tabular}
\label{tab:ctimes_AB_const_T1_d2_N1000_M1000}
\end{table}

\begin{remark}\label{rem:compTimes}%
    {%Let us point out the implications of the computational times in
%        Table \ref{tab:ctimes_AB_const_T1_d2_N1000_M1000} and Table \ref{tab:ctimes_B0_T1_d2_N1000_M1000}.
                {We can see from Table \ref{tab:ctimes_AB_const_T1_d2_N1000_M1000} that the Magnus methods \protect\UseVerb{m1}, \protect\UseVerb{m2} and
        \protect\UseVerb{m3} are significanty faster than \protect\UseVerb{euler}. The reason for this is two-fold: on the one hand, we have the possibility of parallelizing  the Magnus methods over both time and
        samples, while \protect\UseVerb{euler} is only parallelizable over all samples,
                and on the other hand, {we can discretize the Magnus expansion with a time-step that is the square root of the one used for Euler-Maruyama, due to the different rates of convergence.}
%                we have the advantage of using only a square-root of the necessary time step size compared to \protect\UseVerb{euler} due to the different convergence rate of the chosen deterministic integration scheme (order 1) and the
%                theoretical convergence rate of the Euler-Maruyama scheme (order 0.5).
                }

        In our numerical experiments we already use 6 CPU cores to parallelize the computation of the matrix exponential on the CPU, while we use one GPU to compute the Magnus logarithm. For \protect\UseVerb{euler} we speed up in each iteration the matrix multiplications by using \verb+pagefun+ on a GPU to parallelize over all samples.
        As for \protect\UseVerb{m1}, \protect\UseVerb{m2} and
        \protect\UseVerb{m3}, if we were to increase the number of CPU cores to, say, 12, we could see an approximate reduction in the
        computation time of matrix exponentiation by half (plus overhead), making it about {$24$ times as
        fast as the Euler method.}

                {        Now, the very nature of \protect\UseVerb{euler} (see Example \ref{exm:moments} together with Table \ref{tab:moments}) as an iterative scheme yields another advantage of the Magnus methods; %over it, n
        namely, that the computation of the logarithm is very fast and if one needs only the solution of the SDE at the terminal time then one has to compute the matrix exponential only at a single time. %point over all samples.
        Let us consider Table \ref{tab:ctimes_AB_const_T1_d2_N1000_M1000} for the moment. In this particular experiment it would mean that we can divide the computational time of the matrix exponentiation by approximately $\Delta^{-1}=10^2$ without increasing the CPU core count. Hence, the Magnus methods would
        require approximately only $0.04$ seconds plus effects from distributing the memory to the different processors. The \protect\UseVerb{euler} method, in contrast, does not benefit
        from this because, as an iterative method, it must fully evaluate the trajectories.}

        Such situations are not uncommon; for example, in mathematical finance pricing a European call option depends only on the terminal time of the underlying process, giving the Magnus methods a tremendous advantage even without increasing CPUs or GPUs. We will illustrate such a situation in Example \ref{exm:moments} together with Table \ref{tab:moments}.
        In calibration procedures, such as fitting a model to data at few points in time, the Magnus method also excels for the same reason.
    }
\end{remark}
\begin{table}%
    \caption{$A$ and $B$ constant. Computational times and values of first, second and third moment at the terminal time $t=1$ for \protect\UseVerb{m1}, \protect\UseVerb{m2}, \protect\UseVerb{m3}, using $\Delta = 10^{-2}$, and \protect\UseVerb{euler}, using $\Delta=10^{-4}$, obtained with $10^3$ samples.}
    \begin{tabularx}{\linewidth}{c*{4}{X}c}
        \textbf{Method} & $E\left[\left(\left(X_t\right)_{11}\right)^k\right]$
                                                & $E\left[\left(\left(X_t\right)_{12}\right)^k\right]$
                                                & $E\left[\left(\left(X_t\right)_{21}\right)^k\right]$
                                                & $E\left[\left(\left(X_t\right)_{22}\right)^k\right]$
                                                &\text{\textbf{Total time}}\\
        \hline
        \multicolumn{6}{c}{\bf First moment, $k=1$}\\
                \protect\UseVerb{euler} & 0.884995 & 0.136974 & -0.913738 & 1.99784 & 6.83223 \\
                \protect\UseVerb{m1} & 1.23538 & -0.510346 & -1.38672 & 2.88552 & 0.11308 \\
                \protect\UseVerb{m2}  & 0.92461 & 0.0488442 & -0.889341 & 2.00131  & 0.16783 \\
                \protect\UseVerb{m3} & 0.886685 & 0.132748 & -0.910886 & 1.9915 & 0.185966 \\
        \hline
        \multicolumn{6}{c}{\bf Second moment, $k=2$}\\
                \protect\UseVerb{euler} & 1.20982 & 1.09315 & 1.78757 & 7.06842 & 6.83263 \\
                \protect\UseVerb{m1} & 2.49141 & 3.21348 & 3.99804 & 15.4156 & 0.113674 \\
                \protect\UseVerb{m2}  & 1.31038 & 1.18291 & 1.7166 & 7.07746 & 0.168869 \\
                \protect\UseVerb{m3}  & 1.21421 & 1.09186 & 1.77727 & 7.00593 & 0.186955 \\
        \hline
        \multicolumn{6}{c}{\bf Third moment, $k=3$}\\
                \protect\UseVerb{euler}  & 2.62519 & -3.20706 & -5.74199 & 40.8729 & 6.83229 \\
                \protect\UseVerb{m1}  & 8.21939 & -20.9025 & -18.9891 & 136.058 & 0.1131 \\
                \protect\UseVerb{m2}  & 2.95392 & -4.05804 & -5.44452 & 40.8636 & 0.168065 \\
                \protect\UseVerb{m3}  & 2.6546 & -3.27689 & -5.70576 & 40.2687 & 0.185915
    \end{tabularx}
    \centering
    \label{tab:moments}
\end{table}
\begin{example}\label{exm:moments}%
    %\color{red}
    In this example we want to demonstrate the benefit, explained in Remark \ref{rem:compTimes}, of using the Magnus methods compared to iterative schemes, such as the Euler method, when calculating the first, second and third element-wise moments of the terminal value of a matrix-valued SDE.
    %To be more precise
    Precisely, we evaluate %$E\left[F(X_t)\right]$, where $F:\R^{d\times d}\rightarrow \R^{d\times d}, F(X)\coloneqq f(X_{ij})$, $i,j=1,\dots,d$, and $f:\R\rightarrow\R,\ f(x)\coloneqq x^k$,
   $E\big[((X_t)_{ij})^k\big]$ for $i,j=1,\dots,d$ and $k=1,2,3$.
%     {\color{red} 
%     \sout{Compared to the parameters we used in Table \ref{tab:AB_const}, we will increase the sample size to $10^4$ for all methods, but keep the same step sizes.}
%     }
        {We will keep the same parameters as in Table \ref{tab:AB_const}.}
%{\color{blue} In order to make the comparison of the computational times fairer, regardless of
%accuracy, with respect to the Euler method, we keep $T=1$ but alter, compared to the parameters
%used in Table \ref{tab:AB_const}, $\Delta$ to $10^{-3}$ and the samples to $10^4$ for \verb+m1+,
%\verb+m2+,    \verb+m3+ and \verb+euler+.}

The results of this example are summarized in Table \ref{tab:moments}. In this table columns 2--4
contain the values of the element-wise moments at the terminal time of the solution to the SDE
with constant coefficients starting with the upper left corner of the solution matrix, then the
upper right, lower left and lower right, respectively. In the last column we present the
computational times in seconds.

    %{\color{blue} \sout{We can see that the theoretically conjectured times from Remark \ref{rem:compTimes} for evaluating matrix exponential differs from the real times displayed in Table \ref{tab:moments}, because of various reasons, e.g. distributing memory to the different processors, idle time of a processor that has already finished, etc.}}
        {
        %\color{blue} %We can see that when each sub-table for $k=1,2,3$ is considered separately,
        The values of the moments do not differ significantly between \protect\UseVerb{euler}  and \protect\UseVerb{m3}, %However, due to the increased step-size, it could well be that the values produced by \verb+m3+ are more accurate.
    and remarkably %, compared to the Euler method,
    the Magnus methods %yield a roughly 2-times speed-up
    are roughly 35 times as fast in this particular example. We stress again at this point that a coarser time-grid for \protect\UseVerb{euler} would not be comparable to the accuracy of \protect\UseVerb{m3}.

%        \sout{Decreasing the step-size for both methods, on the other hand, will favor the Magnus methods, for
%example, if we change $\Delta$ to $10^{-4}$ %in another experiment with $10^4$ samples,
%then
%\protect\UseVerb{m3} will take about $2.49$ seconds and \protect\UseVerb{euler} $6.11$ seconds, yielding a 2.5-fold
%acceleration.}
}
\end{example}
%{\blue In the case $T=1$ with $10^3$ time-steps and samples for Magnus and
%$10000$ time-steps and $10^3$ samples for Euler it took roughly $84$ seconds to compute the Euler-method, $2$ seconds for the Magnus-logarithm and $46$ seconds for the Matrix exponential, i.e. in total for each Magnus order separately $48$ seconds to compute the whole time evolution for all samples.
%The Magnus-methods are faster than the Euler-method in this case despite the fact, that we still use \verb+Matlab+'s \verb+expm+. (is this the computation time for the whole trajectory up to a certain time?)}
{In the interesting paper \cite{Lord2018} a non-linear extension in the case of commuting $A$ and $B$ can be found and applications to SPDEs via space discretizations are discussed, which is the same approach we take in the next subsection with the ME.}

\subsection{Applications to %\blu{parabolic (remove?)}
SPDEs}\label{sec:par_spdes} The aim of this subsection is to apply {the previously derived ME} for {the
numerical solution} of parabolic stochastic partial differential equations (SPDEs). {We derive an approximation scheme for the general case of variable coefficients, which we only test in the case of the stochastic heat-equation (Example \ref{sec:spde_numerics}), for which an exact solution is available.}

\subsubsection{Stochastic Cauchy problem and fundamental solution}\label{subsec:SPDEs}
Let {$(\Omega,\F,P, (\F_t)_{t\geq 0})$} be a filtered probability space endowed with a real
Brownian motion $W$. We consider the %parabolic SPDE of the type
stochastic Cauchy problem
\begin{equation}\label{eq:SPDE}
\begin{cases}
 \dd u_t(x) = \Lbf_t u_t (x)   \dd t + \Gbf_t u_t(x) \dd W_t, \qquad t> 0,\ x\in \R,\\
 u_0 = \varphi,
\end{cases}
\end{equation}
where $\Lbf_t $ is the %second-order
elliptic linear operator acting as
\begin{equation}
\Lbf_t u_t(x) = \frac{1}{2} \abf_t(x) \partial_{xx}u_t(x) + \bbf_t(x) \partial_{x}u_t(x) + \cbf_t(x)u_t(x) ,
\end{equation}
and $\Gbf_t$ is the %second-order
first-order linear operator acting as
\begin{equation}
\Gbf_t u_t(x) = \sbf_t (x) \partial_x u_t(x) +\gbf_t(x) u_t(x)  .
\end{equation}
The coefficients $(\abf, \bbf, \cbf, \gbf, \sbf)$ are random fields indexed by $(t,x)\in
[0,\infty[\times \R$ and the initial datum $\varphi$ is a random field on $\R$. {A classical
solution to \eqref{eq:SPDE} is understood here as a predictable and almost-surely continuous random field
$u=u_t(x)$ over $[0,\infty[\times \R$,} such that $u_t \in C^{2}(\R)$ a.s. for any $t>0$ and
\begin{equation}
  u_t (x) = \phi(x) + \int_{0}^t   \Lbf_{\tau} u_{\tau} (x)  \dd \tau +
  \int_0^t  \Gbf_{\tau} u_{\tau}(x)   \dd W_{\tau}, \qquad t\geq 0,\ x\in\R. %\ x\in\R.
\end{equation}
%for any $x\in\R$.
%Under suitable measurability, regularity and boundedness assumptions for such random fields, Eq. \eqref{eq:SPDE} admits a unique (up to indistinguishability) solution.  %satisfying suitable regularity and measurability assumptions that will be specified later on.
%Solutions to \eqref{eq:SPDE} are understood here as suitably regular random fields over $[0,\infty[\times \R$ such that
%\begin{equation}
%u_t (x) = u_0 + \int_{0}^t   \Lbf_{\tau} u_{\tau} (x)  \dd \tau + \int_0^t  \Gbf_{\tau} u_{\tau}(x)   \dd W_{\tau}, \qquad t\geq 0, \ x\in\R.
%\end{equation}
There is a vast literature on stochastic SPDEs and problems of the form \eqref{eq:SPDE}, under
suitable measurability, regularity and boundedness assumptions on the coefficients and on the
initial datum: see, for instance, \cite{MR0501350}, \cite{Mikulevicius}, \cite{Chowbook},
\cite{PascucciPesce} and the references therein.

Note that, in analogy with deterministic PDEs, the solution of the Cauchy problem \eqref{eq:SPDE}
can be written, in some cases, as a %space-
convolution of the initial datum with a \emph{stochastic
fundamental solution} $p(t,x;0,\xi)$, i.e.
\begin{equation}
 u_t(x) = \int_{\R} p(t,x;0,\xi) \phi(\xi)  \dd \xi, \qquad (t,x)\in\,]0,\infty[\times \R,%\qquad P\text{-a.s.},
\end{equation}
with $p(t,x;0,\xi)$ being a random field that solves the SPDE in \eqref{eq:SPDE} with respect to
the variables $(t,x)$ and which approximates a Dirac delta centered at $\xi$ as $t$ approaches
$0$. %We refer to \citep{} for a precise definition, and for existence-uniqueness results for such
%fundamental solution.

\subsubsection{Finite-difference Magnus scheme}\label{sec:finitediff_magnus}
We employ the stochastic ME to develop an approximation scheme for the Cauchy problem %the SPDE
\eqref{eq:SPDE}. Our goal here is only to hint at the possibility that the stochastic ME is a
useful tool for the numerical solution of SPDEs. Therefore, we keep the exposition at a heuristic
level and postpone the rigorous study of the problem for further research.

The idea is to apply finite-difference space-discretization for the operators $\Lbf$ and $\Gbf$,
and then ME to solve the resulting linear (matrix-valued) It\^o SDE. We fix a bounded interval
$[a,b]$ and use the following notation: for a given { $d\in\N$}, we denote by {$\varsigma_d$} a
mesh of {$d+2$} equidistant points in $[a,b]$, i.e.
\begin{equation}
{ \varsigma_d = \{  x^d_i \mid x^d_i = a + i h ,\ i=0,\dots,  d+1 \}, \qquad
h\coloneqq\frac{b-a}{d+1} },
\end{equation}
and for any random field ${\bf f}(x)$, $x\in \R$, we denote by ${\bf f}^{d}= ({\bf f}^{{d}}_{0},
\dots, {\bf f}^{{d}}_{{d+1}})$ the random vector whose components correspond to ${\bf f}$
evaluated at the points of the mesh, namely
\begin{equation}
{\bf f}^{{d}}_{i} = {\bf f}(x^{{d}}_i), \qquad i=0,\dots, {d+1}.
\end{equation}
%
%
%and the random vector $u^m_t =  (u^m_{t,0}, \cdots, u^m_{t,m})_{t> 0}$ whose components correspond to the solution $u_t$ evaluated at the points of the mesh, namely
%\begin{equation}
%u^m_{t,i} = u_t (x^m_i), \qquad i=0,\cdots, m.
%\end{equation}
Following the classical centered finite-difference discretization, we approximate the spatial derivatives in each point as
\begin{equation}
 \partial_x u_t (x^{{d}}_i) \approx \frac{u^{{d}}_{t,i} - u^{{d}}_{t,i-1}}{h} ,\quad \partial_{xx} u_t (x^{{d}}_i) \approx \frac{u^{{d}}_{t,i+1} - 2 u^{{d}}_{t,i} + u^{{d}}_{t,i-1}}{h^2}, \qquad i=1,\dots, {d} ,
\end{equation}
to obtain the %finite-difference
system of It\^o SDEs
\begin{equation}\label{eq:SPDE_discrete}
\begin{cases}
 \dd u^{{d}}_{t,i} = (\Lbf^{{d}}_t u^{{d}}_t)_i   \dd t + (\Gbf^{{d}}_t u^{{d}}_t)_i \dd W_t,&%\qquad t> 0,
 \\
 u^{{d}}_{0,i} = \phi^{{d}}_i,&
\end{cases}
\end{equation}
for $i=1,\cdots, {d}$, where $\Lbf^{{d}}_t$ and $\Gbf^{{d}}_t$ are now the %linear
operators %from $\R^m$ onto itself
acting as
\begin{align}
(\Lbf^{{d}}_t  u^{{d}}_t)_i  &=  \frac{1}{2} \abf^{{d}}_{t,i} \frac{u^{{d}}_{i+1} - 2 u^{{d}}_{i}
+ u^{{d}}_{i-1}}{h^2} + \bbf^{{d}}_{t,i} \frac{u^{{d}}_i - u^{{d}}_{i-1}}{h} + \cbf^{{d}}_{t,i}
u^{{d}}_{i},  \\
 ( \Gbf^{{d}}_t u^{{d}}_t)_i   & =   \sbf^{{d}}_{t,i}  \frac{u^{{d}}_i - u^{{d}}_{i-1}}{h} +\gbf^{{d}}_{t,i} u^{{d}}_{i}  .
\end{align}
By imposing {some boundary conditions, for instance}
\begin{equation}\label{ae10}
 u^{{d}}_{t,0} = u^{{d}}_{t,{{d+1}}} = 0, \qquad t>0,
\end{equation}
the {system of SDEs} \eqref{eq:SPDE_discrete} can be cast in the framework of the previous
section. More precisely, under condition \eqref{ae10}, system \eqref{eq:SPDE_discrete} is
equivalent to
\begin{equation}\label{eq:linear_SDE_SPDE}
\begin{cases}
 \dd \bar{u}^{{d}}_{t} =   A^{{d}}_t \bar{u}^{{d}}_t  \dd t + B^{{d}}_t \bar{u}^{{d}}_t \dd W_t,&\\
 \bar{u}^{{d}}_{0} =  \bar{\phi}^{{d}},&
\end{cases}
\end{equation}
%in the following way. Set the $(m-1)$-dimensional random vector
where we set
\begin{equation}
 \bar{u}^{{d}}_t = (u^{{d}}_{t,1},\dots, u^{{d}}_{t,{{d}}}),\qquad \bar{\phi}^{{d}} = (\phi^{{d}}_{1}, \dots, \phi
^{{d}}_{{{d}}}), %\qquad t\geq 0,
\end{equation}
and $A^{{d}}_t,B^{{d}}_t$ are the random tridiagonal $({{d}}\times {{d}})$-matrices given by
%\vspace{5pt}
\begin{equation}
A^{{d}}_t = \left( \begin{matrix}
-\frac{\abf^{{d}}_{t,1}}{h^2} + \frac{ \bbf^{{d}}_{t,1}}{h} +\cbf^{{d}}_{t,1} & \frac{1}{2} \frac{\abf^{{d}}_{t,1}}{h^2}   & %0 &
 \cdots & 0  \vspace{20pt}  \\
\frac{1}{2} \frac{\abf^{{d}}_{t,2}}{h^2} -\frac{\bbf^{{d}}_{t,2}}{h} &
-\frac{\abf^{{d}}_{t,2}}{h^2} + \frac{ \bbf^{{d}}_{t,2}}{h} +\cbf^{{d}}_{t,2} &
%\frac{1}{2} \frac{\abf^m_{t,2}}{h^2}  &
\ddots & \vdots  \vspace{20pt}  \\
%0& \frac{1}{2} \frac{\abf^m_{t,3}}{h^2} -\frac{\bbf^m_{t,3}}{h} & -\frac{\abf^m_{t,3}}{h^2} + \frac{ \bbf^m_{t,3}}{h} +\cbf^m_{t,3} & \ddots  & 0  \vspace{15pt}  \\
\vdots& \ddots&
%\ddots &
 \ddots & \frac{1}{2} \frac{\abf^{{d}}_{t,{{d -1}}}}{h^2}  \vspace{20pt}   \\
0& \cdots&
%0 &
\frac{1}{2} \frac{\abf^{{d}}_{t,{{d}}}}{h^2} -\frac{\bbf^{{d}}_{t,{{d}}}}{h} &
-\frac{\abf^{{d}}_{t,{{d}}}}{h^2} + \frac{ \bbf^{{d}}_{t,{{d}}}}{h} +\cbf^{{d}}_{t,{{d}}}
\\
\end{matrix}
\right),
\end{equation}
\vspace{5pt}
\begin{equation}
B^{{d}}_t =\left( \begin{matrix}
 \frac{ \sbf^{{d}}_{t,1}}{h} +\gbf^{{d}}_{t,1} & 0   & % 0 &
  \cdots & 0  \vspace{15pt}  \\
 -\frac{\sbf^{{d}}_{t,2}}{h} &  + \frac{ \sbf^{{d}}_{t,2}}{h} +\gbf^{{d}}_{t,2} & % 0  &
  \ddots & \vdots  \vspace{15pt}  \\
%0& -\frac{\sbf^m_{t,3}}{h} &  + \frac{ \sbf^m_{t,3}}{h} +\gbf^m_{t,3} & \ddots  & 0  \vspace{15pt}  \\
\vdots& \ddots& %\ddots &
 \ddots & 0  \vspace{15pt}   \\
0& \cdots& %0 &
 -\frac{\sbf^{{d}}_{t,{{d}}}}{h} &   + \frac{ \sbf^{{d}}_{t,{{d}}}}{h} +\gbf^{{d}}_{t,{{d}}}
\end{matrix}\right).\vspace{5pt}
\end{equation}
%For a given initial condition $u^{{d}}_0$,
Now, the solution to \eqref{eq:linear_SDE_SPDE} can be written as
\begin{equation}
\bar{u}^{{d}}_t = X^{{d}}_t \bar{u}^{{d}}_0, \qquad t\geq 0,
\end{equation}
where $X^{{d}}$ is in turn the solution to the $\mathcal{M}^{{{d}}\times {{d}}}$-valued It\^o SDE
\begin{equation}\label{eq:SDE_linear_bis}
\begin{cases}
 \dd X^{{d}}_t = A^{{d}}_t X^{{d}}_t \dd t + B^{{d}}_t X^{{d}}_t \dd W_t,\\
 X^{{d}}_0=I_{{{d}}}.
\end{cases}
\end{equation}

\begin{remark}
The components of $X^{{d}}$ %to \eqref{eq:SDE_linear_bis}
can be regarded as approximations of the integrals of the fundamental solution of the SPDE in
\eqref{eq:SPDE}, when it exists, on each sub-interval {$[\frac{1}{2}(x^{{d}}_{j-1}+x^{{d}}_{j}), \frac{1}{2}(x^{{d}}_{j}+x^{{d}}_{j+1})]$}, namely
\begin{equation}\label{eq:int_fund_sol}
(X^{{d}}_t)_{i,j}  \approx    \int%\limits
_{\frac{1}{2}(x^{{d}}_{j-1}+x^{{d}}_{j})}^{\frac{1}{2}(x^{{d}}_{j}+x^{{d}}_{j+1})}
p\big(t,x^{{d}}_i;0,\xi\big) \dd \xi  =:  (\mathcal{I}^{{d}}_t)_{i,j}, \qquad i,j =
1,\dots,{{d}}.%,\qquad t>0.
\end{equation}
\end{remark}

%\subsection{A numerical test}\label{sec:spde_numerics}
%\input{Numerics/SPDE/SPDE_commands.tex}

\begin{example}\label{sec:spde_numerics}
We consider a special case of \eqref{eq:SPDE} % for our numerical test, namely
with $\abf_t\equiv \abf >0$, $\bbf,\cbf,\gbf \equiv 0$ and $\sbf_t\equiv \sbf > 0$. Hence, we
consider the stochastic heat equation
\begin{equation}%\label{eq:SPDE}
\dd u_t = \frac{\abf}{2} \partial_{xx}u_t(x)   \dd t + \sbf \partial_x u_t(x) \dd W_t, \qquad t> 0,\ x\in \R,
\end{equation}
with $\abf>\sbf^2$, whose stochastic fundamental solution is given explicitly by
\begin{equation}\label{eq:exp_fund_sol}
 p(t,x;0,\xi)\coloneqq \frac{1}{\sqrt{2\pi (\abf- \sbf^2)t%(t-s)
 }}  \exp\bigg(   - \frac{(   x + \sbf W_t%(W_t - W_s)
 - \xi   )^2}{2 (\abf- \sbf^2)t %(t-s)
 }    \bigg),
 \qquad t>0%s
 ,\ x,\xi \in\R.
\end{equation}
%For a given $m\in \N$, the solution to \eqref{eq:SDE_linear_bis} must be interpreted as an approximation of the fundamental solution, in the sense that
%\begin{equation}
%(X^m_t)_{i,j}  \approx  \int_{x^m_{j-1}}^{x^m_j} p(0,\xi;t,x^m_i) \dd \xi, \qquad i,j \in \{ 1,\cdots,m-1  \},\qquad t>0.
%\end{equation}
%x^m_i = a + i h
The matrices $A_t^{{d}}$ and $B_t^{{d}}$ in \eqref{eq:SDE_linear_bis} now read as
\begin{equation}
A^{{d}}_t \equiv\frac{\abf}{h^2}\left( \begin{matrix} -1  & \frac{1}{2}    & \cdots & 0
\vspace{0pt}  \\ \frac{1}{2}& - 1 &  \ddots & \vdots  \vspace{0pt}  \\
%0& \frac{1}{2}& -1 & \ddots  & 0  \vspace{0pt}  \\
\vdots& \ddots&  \ddots & \frac{1}{2}   \vspace{0pt}   \\
0& \cdots&  \frac{1}{2}  &   -1   \\
\end{matrix}\right), \qquad
B^{{d}}_t \equiv\frac{\sbf}{h}\left( \begin{matrix} 1  & 0     & \cdots & 0  \vspace{0pt}  \\ -1&
1   & \ddots & \vdots  \vspace{0pt}  \\
%0& -1 & 1 & \ddots  & 0  \vspace{0pt}  \\
\vdots& \ddots&  \ddots & 0   \vspace{0pt}   \\
0& \cdots&  -1  &   1   \\
\end{matrix}\right).
\end{equation}
In particular, they do not commute and are constant for fixed $d$.

\begin{table}[!ht]
\centering
    \caption{SPDE. Values of $\Eb[\text{Err}^{{d}}_t]$ (in percentage) for
    \protect\UseVerb{m1} and \protect\UseVerb{m3}, with ${d}=50$, obtained with $50$ independent samples. Parameters as in \eqref{eq:parameters_spdes}.
}
\begin{tabular}{cccccc}
Method                                  &  $t=0.1$  & $t=0.2$  & $t=0.3$  & $t=0.4$  & $t=0.5$  \\
\hline
\protect\UseVerb{euler}  & 9.1364\,\%  & 5.5467\,\% & 5.3231\,\% & 4.8377\,\% & 4.5829\,\% \\
\protect\UseVerb{m1}   & 8.0746\,\%  & 5.3243\,\% & 4.9617\,\% & 4.7273\,\% & 5.3065\,\% \\
%\protect\UseVerb{m2}    & 9.1337\%  & 5.5296\% & 5.3311\% & 4.8315\% & 4.5700\% \\
\protect\UseVerb{m3}    & 9.1337\,\%  & 5.5296\,\% & 5.3310\,\% & 4.8314\,\% & 4.5704\,\%
\end{tabular}
 \label{tab:spde1}
\end{table}
\begin{table}[!ht]
\centering
    \caption{SPDE. Values of $\Eb[\text{Err}^{{d}}_t]$ (in percentage) for
    \protect\UseVerb{m1} and \protect\UseVerb{m3}, with ${d}=100$, obtained with $50$ independent samples. Parameters as in \eqref{eq:parameters_spdes}.
}
\begin{tabular}{cccccc}
Method                                  &  $t=0.1$  & $t=0.2$  & $t=0.3$  & $t=0.4$  & $t=0.5$  \\
\hline
\protect\UseVerb{euler}  & 4.2053\,\%  & 3.4600\,\% & 2.8214\,\% & 2.3524\,\% & 2.0370\,\% \\
\protect\UseVerb{m1}     & 4.4452\,\%  & 5.1232\,\% & 4.7061\,\% & 4.8807\,\% & 4.8397\,\% \\
%\protect\UseVerb{m2}    & 4.2576\%  & 3.4543\% & 2.8172\% & 2.3598\% & 2.0968\% \\
\protect\UseVerb{m3}    & 4.2576\,\%  & 3.4543\,\% & 2.8172\,\% & 2.3598\,\% & 2.0467\,\%
\end{tabular}
 \label{tab:spde2}
\end{table}
\begin{table}[!ht]
    \caption{SPDE. Values of $\Eb[\text{Err}^{{d}}_t]$ (in percentage) for
    \protect\UseVerb{m1} and \protect\UseVerb{m3}, with ${d}=200$, obtained with $50$ independent samples. Parameters as in \eqref{eq:parameters_spdes}.
}
\centering
\begin{tabular}{cccccc}
Method                                  &  $t=0.1$  & $t=0.2$  & $t=0.3$  & $t=0.4$  & $t=0.5$  \\
\hline
\protect\UseVerb{euler}  & 2.1832\,\%  & 1.4403\,\% & 1.4190\,\% & 1.2174\,\% & 1.1532\,\% \\
\protect\UseVerb{m1}    & 4.9891\,\%  & 5.0444\,\% & 4.8249\,\% & 5.0042\,\% & 5.2603\,\% \\
%\protect\UseVerb{m2}    & 2.1690\%  & 1.4366\% & $\inf$   & $\inf$   & $\inf$ \\
\protect\UseVerb{m3}   & 2.1690\,\%  & 1.4364\,\% & 1.4467\,\% & 1.2140\,\% & 1.1420\,\%
\end{tabular}
\label{tab:spde3}
\end{table}

In the next numerical test we compare the approximate solutions to \eqref{eq:SDE_linear_bis},
obtained with the stochastic ME {in the special case of constant coefficients \eqref{eq:ab_const_y3}}, with the $\mathcal{M}^{{{d}}\times{{d}}}$-valued stochastic
process $\mathcal{I}^{{d}}$, whose components are given by the integral in \eqref{eq:int_fund_sol}
with $p$ as in \eqref{eq:exp_fund_sol}. In doing this, we shall keep in mind that the difference
between the latter quantities can be decomposed into two errors, namely: the one between $X^{{d}}$
and its approximation, and the one between $X^{{d}}$ and ${\mathcal{I}^d}$. In turn, the latter is
the result of both space-discretization and the error that stems by imposing null boundary
conditions (see \eqref{ae10}). In particular, this last error cannot be reduced by refining the
space-grid. Therefore, the analysis should be restricted to the ``central" components of
$\mathcal{I}^{{d}}$, namely those which do not depend on the values of the fundamental solution in
the vicinity of the boundary $\{a,b\}$. This motivates the definition that follows. For a given
$\kappa \in \N$ with $\kappa < {d}$, and a given approximation $X^{{{d}},\text{app}}$ of
$X^{{{d}}}$, we define the process
                \begin{align}\label{eq:err_spde_def}  %\hspace{-15pt}
            \text{Err}^{{d}}_t \coloneqq     %\frac{\Delta}{t}
    %                \sum_{k=1}^N
                                                \frac{\fnorm{ \tilde{\mathcal{I}}^{{{d}},\kappa}_{t}- \tilde{X}^{{{d}},\kappa,\text{\text{app}}}_{t}} }{\fnorm{ \tilde{\mathcal{I}}^{{{d}},\kappa}_{t}}}
%                                                \approx
%             \frac{1}{t}
%                \int_{0}^{t}
%                {
%                     \frac{
%                        \big| X^{\text{ref}}_s  -  X^{\text{\text{app}}}_s   \big|}{\big| X^{\text{ref}}_s  \big|}
%                }               \dd s  \qquad  \text{with $\Delta=\frac{t}{N}$}
                , \qquad t>0,
                \end{align}
where $\tilde{\mathcal{I}}^{{{d}},\kappa}_{t}$ and $\tilde{X}^{{{d}},\kappa,\text{\text{app}}}_{t}$ are the projections on $\mathcal{M}^{\kappa\times {{d}}}$ obtained by selecting the central $\kappa$ rows %and $\kappa$ columns
of $\mathcal{I}^{{{d}}}_{t}$ and $X^{{{d}},\text{\text{app}}}_{t}$, respectively. The matrix norm
above is the Frobenius norm. The role of $X^{{{d}},\text{app}}$ will be played by the
time-discretization of the truncated ME \eqref{eq:logarith_real_2}-\eqref{eq:convergence}. In
particular, we will denote by \verb+m1+ and \verb+m3+ the discretized first and third-order MEs of
$X^{{d}}$, respectively. We will not consider the second-order Magnus approximation $\verb+m2+$ as
it appears less stable than the others. Note that, being $A^{{d}}_t$ and $B^{{d}}_t$ constant
matrices, the first three terms of the ME are given explicitly by \eqref{eq:ab_const_y3}.

In the numerical experiments we set
\begin{equation}\label{eq:parameters_spdes}
a=-2,\ b=2,\quad {\bf a}=0.2,\quad \sbf =0.15 .
\end{equation}
Setting the parameter $\kappa$ in \eqref{eq:err_spde_def}, which determines the number of rows
that are taken into account to asses the error,  as $\kappa = \floor*{ {{d}}/2 } $, we study the
expectation of $\text{Err}^{{{d}}}_t$ up to $t=0.5$. Such choice for $\kappa$ and $t$ allows us to
study the error in a region that is suitably away from the boundary. Indeed, choosing $\kappa$ as
above implies $x^{{d}}_i$ in \eqref{eq:int_fund_sol} ranging roughly from $-1$ to $1$. On the
other hand, the standard-deviation parameter associated to the Gaussian density
\eqref{eq:exp_fund_sol} at $t=0.5$ is roughly $0.30$, while the mean parameter is $0.15\times
W_{0.5}$, whose standard deviation is in turn roughly $0.10$. Therefore, both
$(\tilde{\mathcal{I}}^{{{d}},\kappa}_{t})_{i,1}$ and
$(\tilde{\mathcal{I}}^{{{d}},\kappa}_{t})_{i,{{d}}}$ are likely to be very close to zero, thus
meeting the null boundary condition implied by \eqref{ae10}.

In Tables \ref{tab:spde1}, \ref{tab:spde2}, \ref{tab:spde3}, we report the approximate values of
$\Eb[\text{Err}^{{d}}_t]$ for ${{d}}=50, 100$ and $200$, respectively. These were obtained via
simulation of $50$ %\blu{\sout{independently-sampled}}
trajectories of $W$ %\blu{\sout{, discretized}} %respectively. In every simulation the trajectories have been discretized
with time step-size $\Delta = 10^{-4}$. {%\color{red}
Now, let us inspect the Tables
\ref{tab:spde1}--\ref{tab:spde3} in more detail. As a reminder, these results were obtained by
using the \verb+exact+ solution as a reference, which is available in this particular example.
In Table \ref{tab:spde1} it is noticeable that \verb+euler+ and \verb+m3+ can exhibit worse
results for small times compared to \verb+m1+. This is due to the coarse space approximation with
only $52$ space grid points.
Increasing the number $d$ of grid points improves the error of \verb+euler+ and \verb+m3+ for all
displayed times, which can be seen in Tables \ref{tab:spde2} and \ref{tab:spde3} by comparing each
column for the same final time. % in all tables. %Another notable feature can be seen when comparing \verb+euler+ and \verb+m3+ in all three tables.
%The third-order Magnus expansion has the same magnitude of error as the \verb+euler+ scheme for
%all final times.
Finally, notice that the third-order Magnus expansion has the same magnitude of error as the
\verb+euler+ scheme for all final times.}
\end{example}

\subsection{Example: $j=1$, $B=0$ and $A_t$ upper triangular.}\label{subsec:Auppertriang}
We now test the ME on an
%{Now, we consider an example, where we want to apply the ME to an
SDE with time-dependent coefficients and with known explicit solution. %}
%\vspace{5pt}
%\paragraph*{\bf Example: $j=1$, $B=0$ and $A_t$ upper triangular.$\ {}$}
%\input{Numerics/B0_fix/B0_fix_commands.tex}
%\input{Numerics/B0_var/B0_var_commands.tex}
    %We s
    Set% $B\equiv 0$ and
    \begin{equation}\label{eq:AB_upper_triangular}
A_t= \left( \begin{matrix} 2 & t \\ 0 & -1
\end{matrix}\right), \qquad B_t\equiv 0.
\end{equation}
    In this case \eqref{eq:SDE_linear_b} admits an explicit solution, which can be obtained by using It\^o's formula, given by
%    given by{\color{magenta}, which can be obtained by using Yoeurp and Yor's formula
    \begin{equation}
    X_t=
    \left(
        \begin{array}[c]{cc}
            %X11
            e^{
                    2\left(W_t - t\right)
            }
            &
            %X12
            e^{
                    2\left(W_t - t\right)
            }\left(
                    \int_{0}^{t}{
                            s\, e^{
                                    -3 W_s + \frac{3}{2} s
                            }\,
                            dW_s
                    }-2
                    \int_{0}^{t}{
                            s\, e^{
                                    -3 W_s + \frac{3}{2} s
                            }\,
                            ds
                    }
            \right)
            \\
            %X21
            0
            &
            %X22
            e^{
                    -(W_t + \frac{1}{2}t)
            }
        \end{array}
    \right).
\end{equation}
%}
The first three terms of the ME read as
\begin{align}
    \Mlog_t^{(1)}& =
    %-\frac{1}{2}
                    %\left[
                            %\begin{array}[c]{cc}
                                    %4t
                                        %& \frac{1}{2}t^2\\
                                    %0 & 1t
                            %\end{array}
                    %\right]+
                    \left(
                            \begin{array}[c]{cc}
                                    2 W_t
                                    & tW_t-\int_{0}^{t}{W_s ds}\\
                                    0       & - W_t
                            \end{array}
                    \right),
\\
    \Mlog_t^{(2)}&=
    -\frac{1}{2}
                    \left(
                            \begin{array}[c]{cc}
                                    4t
                                        & \frac{1}{2}t^2\\
                                    0 & t
                            \end{array}
                    \right)
                    -\frac{3}{2}
                            \left(
                                    \begin{array}[c]{cc}
                                            0 & W_t\int_{0}^{t}{W_u du}-
                                                            \int_{0}^{t}{W_u^2 du}\\
                                            0       & 0
                                    \end{array}
                            \right),
\\                  \Mlog_t^{(3)}&=
                    %\left[
                            %\begin{array}[c]{cc}
                                    %0 & \frac{15}{16}\left(\int_{0}^{t}{
                                                            %W_s
                                                            %ds
                                                            %}\right)^2-\frac{7}{8}\left(
                                                            %t
                                                                    %\int_{0}^{t}{W_s^2 ds}-
                                                            %\int_{0}^{t}{sW_s^2 ds}
                                                            %\right)-\frac{1}{24}
                                                            %t^3\\
                                    %0 & 0
                            %\end{array}
                    %\right]\\&+
                    \left(
                            \begin{array}[c]{cc}
                                    0 & \frac{3}{4}(
                                                    t-W_t^2)\int_{0}^{t}
                                                                    W_u
                                                                    du
                                                            -
                                                            \frac{3}{2}
                                                            \int_{0}^{t}{
                                                                    W_u
                                                                    u
                                                                    du
                                                            }
                                    \\
                                    0 & 0
                            \end{array}
                    \right)+
                                        \left(
                            \begin{array}[c]{cc}
                                    0 &
                                    \frac{9}{4}
                                                    W_t\int_{0}^{t}{
                                                            W_u^2du
                                                    }-\frac{3}{2}
                                                    \int_{0}^{t}{
                                                            W_u^3 du
                                                    }
                                    +
                                    \frac{3}{8}
                                            t^2W_t\\
                                    0 & 0
                            \end{array}
                    \right).
            \end{align}
Again, all the stochastic integrals appearing in the ME can be solved in terms of Lebesgue
integrals by using It\^o's formula, which allows us to
use one more time a sparser time grid compared to the Euler method and the discretized exact solution. %, for which the discretization of stochastic integrals is necessary.
{In the following numerical tests, we discretize in time with mesh $\Delta$ equal to
$10^{-4}$ for \verb+exact+ and equal to $10^{-2}$ for \verb+m1+, \verb+m2+ and
\verb+m3+. For \verb+euler+ we run two experiments with mesh equal to $10^{-4}$ and $10^{-3}.$} {Note that \verb+euler+ serves here as an alternative approximation and that choosing
a finer time-discretization for \verb+euler+ and \verb+exact+ (our reference method here) is again
essential in order to make them comparable with \verb+m3+.

In Table \ref{tab:B0_fix} we show the expectations $E[\text{Err}_t]$ for different values of $t$,
with \verb+exact+ as benchmark solution, computed via Monte Carlo simulation with $10^{3}$
samples. The same samples are used in Figure \ref{fig:B0_fixcds} to plot the empirical CDF of
$\text{Err}_t$. {It is clear from the results that the time-step size $\Delta=10^{-3}$ is not small enough in order for \verb+euler+ to yield accurate results. Also note that \verb+m3+ outperforms \verb+euler+ with $\Delta=10^{-4}$ up to $t=0.75$.
%Let us direct our focus for the moment on the columns for $t=0.25$ up to $t=1$. We can see that \verb+m3+ is more accurate than \verb+euler+ in the case $\Delta=10^{-3}$, justifying the usage of a finer time grid for \verb+euler+. But even for \verb+euler+ with a finer step size equal to $\Delta=10^{-4}$, \verb+m3+ outperforms \verb+euler+ till $t=0.75$.
}
\begin{table}
\centering
    \caption{$B_t$ and $A_t$ as in \eqref{eq:AB_upper_triangular}. Values of $E[\text{Err}_t]$ (in percentage) for
    \protect\UseVerb{euler}, \protect\UseVerb{m1}, \protect\UseVerb{m2}, \protect\UseVerb{m3}, with \protect\UseVerb{exact} %, using $\Delta=10^{-4}$,
    as benchmark solution, obtained with $10^3$ samples.
}
\begin{tabular}{*{7}{c}}
Method  &
        $t=0.25$    & $t=0.5$     & $t=0.75$    & $t=1$      & $t=2$          &  $t=3$    \\
\hline
\protect\UseVerb{euler} $\Delta=10^{-4}$ &
        $0.503\,\%$ & $0.716\,\%$ & $0.883\,\%$ & $1.08\,\%$ & $2.07\,\%$ & $3.3\,\%$ \\
\protect\UseVerb{euler} $\Delta=10^{-3}$ &
        $1.56\,\%$  & $2.12\,\%$  & $2.56\,\%$  & $2.92\,\%$ & $4.16\,\%$ & $5.6\,\%$ \\
\protect\UseVerb{m1} $\Delta=10^{-2}$ &
        $19.6\,\%$  & $42.9\,\%$  & $70.8\,\%$  & $108\,\%$  & $409\,\%$  & $1490\,\%$ \\
\protect\UseVerb{m2} $\Delta=10^{-2}$ &
        $0.147\,\%$ & $0.659\,\%$ & $1.71\,\%$  & $3.36\,\%$ & $11.3\,\%$ & $20.6\,\%$ \\
\protect\UseVerb{m3} $\Delta=10^{-2}$ &
        $0.117\,\%$ & $0.309\,\%$ & $0.77\,\%$  & $1.61\,\%$ & $6.44\,\%$ & $13.9\,\%$ \\
\end{tabular}
\label{tab:B0_fix}
\end{table}
\begin{table}[!ht]
\centering
    \caption{Computational times (seconds) for $10^3$ sampled trajectories of $X$, up to time $t=1$,
        choosing $\Delta %= \sqrt{10^{-4}}
        = 10^{-2}$ for
    \protect\UseVerb{m1}, \protect\UseVerb{m2} and \protect\UseVerb{m3}, and
        $\Delta = 10^{-4}$ for \protect\UseVerb{euler} and \protect\UseVerb{exact}.
}
\begin{tabular}{@{}*{4}{c}}
\multicolumn{4}{c}{$B_{t}$ and $A_t$ as in \eqref{eq:AB_upper_triangular}}\\
\multicolumn{4}{c}{}\\ \text{Method} &\text{Log} &\text{Matrix Exp} &\text{Total}\\ \hline
\protect\UseVerb{exact}  & 0 & 0 & 0.70544 \\
\protect\UseVerb{euler}  & 0 & 0 & 4.59658 \\
\protect\UseVerb{m1}  & 0.0186238 & 0.51252 & 0.531143 \\
\protect\UseVerb{m2}  & 0.0245188 & 0.517689 & 0.542207 \\
\protect\UseVerb{m3}  & 0.0441915 & 0.530973 & 0.575165 \\
\end{tabular}%
\begin{tabular}{|*{3}{c}@{}}
\multicolumn{3}{|c}{$B_{t}$ and $A_t$ as in \eqref{eq:AB_upper_triangular}}\\
\multicolumn{3}{|c}{normalized by its spectral norm}\\ \text{Log} &\text{Matrix Exp}
&\text{Total}\\ \hline
0 & 0 & 0.807434 \\
0 & 0 & 4.53311 \\
0.0128825 & 0.522425 & 0.535308 \\
0.0557674 & 0.527644 & 0.583412 \\
0.212351 & 0.471555 & 0.683906 \\
\end{tabular}
\label{tab:ctimes_B0_T1_d2_N1000_M1000}
\end{table}

\begin{figure}
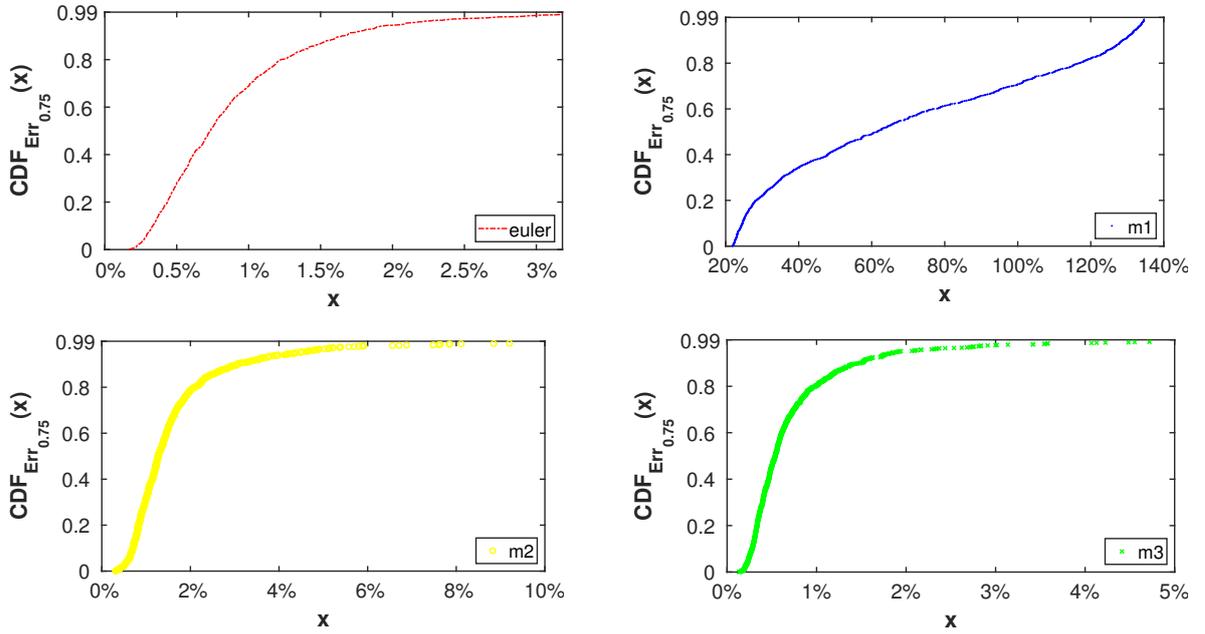
%
    \begin{center}
    \begin{minipage}[c][][c]{.47\linewidth}
        \BfixCDSeuler%
    \end{minipage}\hfill
    \begin{minipage}[c][][c]{.47\linewidth}
        \BfixCDSmone%
    \end{minipage}
    \end{center}
    \begin{center}
    \begin{minipage}[c][][c]{.47\linewidth}
        \BfixCDSmtwo%
    \end{minipage}\hfill
    \begin{minipage}[c][][c]{.47\linewidth}
        \BfixCDSmthree%
    \end{minipage}
    \end{center}
\caption{$B\equiv 0$ and $A_t$ as in \eqref{eq:AB_upper_triangular}. Empirical CDF of $\text{Err}_t$, at
$t=0.75$, for
\protect\UseVerb{euler}, \protect\UseVerb{m1}, \protect\UseVerb{m2}, \protect\UseVerb{m3}, with \protect\UseVerb{exact} as benchmark solution, obtained with $10^3$ samples.}%
\label{fig:B0_fixcds}%
\end{figure}

%The computational time for $10^3$ sampled trajectories, up to time $t=1$, is approximately $9.1$
%seconds for \verb+exact+, $48$ seconds for \verb+euler+ and $40$ seconds for either \verb+m1+,
%\verb+m2+ or \verb+m3+. The latter, however, is divided as follows: nearly $0.1$ seconds to
%compute the %\blu{
%ME %\sout{approximate logarithm}}

{
The computational time for $10^3$ sampled trajectories, up to time $t=1$, which is given in
Table \ref{tab:ctimes_B0_T1_d2_N1000_M1000}, is approximately $0.7$
seconds for \verb+exact+, $4.6$ seconds for \verb+euler+ and $0.6$ seconds for either \verb+m1+,
\verb+m2+ or \verb+m3+. The latter, however, is divided as follows: nearly $0.05$ seconds to
compute the %\blu{
ME %\sout{approximate logarithm}}
and nearly $0.55$ seconds to compute the matrix exponential with the Matlab function \verb+expm+.
{%\color{red}
Let us recall Remark \ref{rem:compTimes}} and note that the computation of the logarithm via ME is very fast thanks to the possibility
of parallelizing the computation of the integrals in $Y^{(1)}, Y^{(2)}$ and $Y^{(3)}$.
}

As it appears in the results above, the accuracy of the ME quickly deteriorates as the time
increases. This is largely due the fact that the spectral norm
\begin{equation}
\snorm{A_t} = \sqrt{\frac{1}{2} \left(t^2+\sqrt{t^4+10 t^2+9}+5\right)}
\end{equation}
is an increasing function of $t$. This behavior shall not come as a surprise, since the proof of
Theorem \ref{th:convergence} already uncovered the relation between the convergence time $\tau$
and the spectral norms of $A_t$ and $B_t$. Such relation is also consistent with the convergence
condition \eqref{eq:converg_cond_determ} that holds in the deterministic case. In order to asses
numerically the impact of the spectral norm of $A_t$ on the quality of the Magnus approximation,
we now repeat the experiments on the equation obtained by normalizing $A_t$ as in
\eqref{eq:AB_upper_triangular} with respect to $\snorm{ A_t}$. As it turns out, the accuracy of
\verb+m1+, \verb+m2+ and \verb+m3+ improves considerably with this normalization.
%$\|A_t \|$ and $\|B_t\|$.
Note that, in this case, \eqref{eq:SDE_linear_b} no longer admits a closed-form solution, while
the representation for the terms $Y^{(1)}, Y^{(2)}$ and $Y^{(3)}$ in the ME is omitted for it
becomes rather tedious to write. In Figure \ref{fig:B0_var} we plot one realization of the
trajectories of the top-right component $(X_t)_{12}$, computed with all the methods above, up to
time $t=10$. In this case we did not plot a diagonal component of the solution because the latter
are exact for \verb+m2+ and \verb+m3+, up to discretization errors of Lebesgue integrals.
%\ref{tab:B0_var}.\\
\begin{figure}[h]
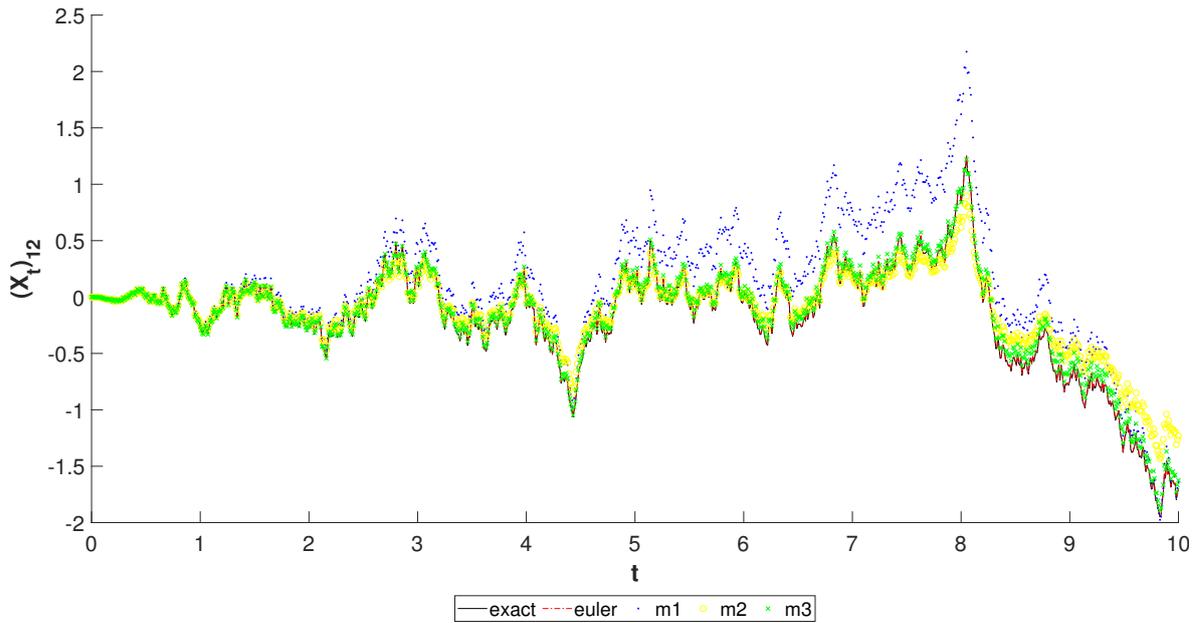
%
\BvarTrajectory
\caption{$B_{t}$ and $A_t$ as in \eqref{eq:AB_upper_triangular} normalized by its spectral norm. One realization of the trajectories of the top-right component $(X_t)_{12}$, computed with \protect\UseVerb{exact}, \protect\UseVerb{euler}, \protect\UseVerb{m1}, \protect\UseVerb{m2}, \protect\UseVerb{m3}.}%
\label{fig:B0_var}%
\end{figure}
\begin{table}
\centering
    \caption{$B_{t}$ and $A_t$ as in \eqref{eq:AB_upper_triangular} normalized by its spectral norm. Values of $E[\text{Err}_t]$ (in percentage) for
    \protect\UseVerb{euler}, \protect\UseVerb{m1}, \protect\UseVerb{m2}, \protect\UseVerb{m3}, with \protect\UseVerb{exact}, using $\Delta=10^{-4}$, as benchmark solution, obtained with $10^3$ samples.
}
\begin{tabular}{*{8}{c}}
Method  &
        $t=0.25$     & $t=0.5$      & $t=0.75$    & $t=1$       & $t=2$       &  $t=3$      & $t=10$\\
\hline \protect\UseVerb{euler} $\Delta=10^{-4}$ &
        $0.134\,\%$  & $0.191\,\%$  & $0.227\,\%$ & $0.263\,\%$ & $0.382\,\%$ & $0.489\,\%$ & $0.715\,\%$ \\
\protect\UseVerb{euler} $\Delta=10^{-3}$ &
        $0.416\,\%$  & $0.572\,\%$  & $0.691\,\%$ & $0.777\,\%$ & $0.965\,\%$ & $1.12\,\%$  & $1.35\,\%$\\
\protect\UseVerb{m1} $\Delta=10^{-2}$ &
        $4.69\,\%$   & $9.25\,\%$   & $13.6\,\%$  & $18.1\,\%$  & $34\,\%$    & $47\,\%$    & $90.9\,\%$\\
\protect\UseVerb{m2} $\Delta=10^{-2}$ &
        $0.047\,\%$  & $0.109\,\%$  & $0.263\,\%$ & $0.529\,\%$ & $2.01\,\%$  & $3.92\,\%$  & $15.2\,\%$\\
\protect\UseVerb{m3} $\Delta=10^{-2}$ &
        $0.0537\,\%$ & $0.0821\,\%$ & $0.126\,\%$ & $0.201\,\%$ & $0.642\,\%$ & $1.28\,\%$  & $5.31\,\%$\\
\end{tabular}
\label{tab:B0_var}
\end{table}
Table \ref{tab:B0_var} and Figure \ref{fig:B0_varcds} are analogous to Table \ref{tab:B0_fix} and
Figure \ref{fig:B0_fixcds} and are obtained again with $10^3$ independent samples.
%\begin{figure}%
%\Bvarcdsplot%
%\caption{CDS plot for $B\equiv 0$ and $A_t$ normalized}%
%\label{fig:B0_varcds}%
%\end{figure}
\begin{figure}
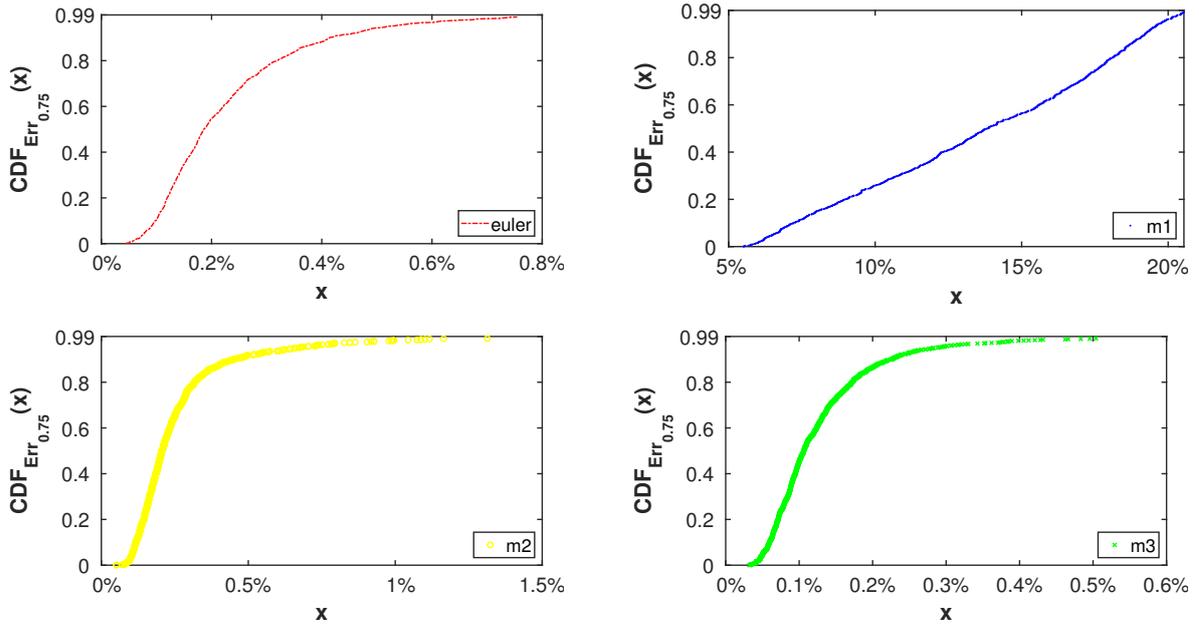
%
    \begin{center}
    \begin{minipage}[c][][c]{.47\linewidth}
        \BvarCDSeuler%
    \end{minipage}\hfill
    \begin{minipage}[c][][c]{.47\linewidth}
        \BvarCDSmone%
    \end{minipage}
    \end{center}
    \begin{center}
    \begin{minipage}[c][][c]{.47\linewidth}
        \BvarCDSmtwo%
    \end{minipage}\hfill
    \begin{minipage}[c][][c]{.47\linewidth}
        \BvarCDSmthree%
    \end{minipage}
    \end{center}
\caption{$B_{t}$ and $A_t$ as in \eqref{eq:AB_upper_triangular} normalized by its spectral norm.
Empirical CDF of $\text{Err}_t$, at $t=0.75$, for
\protect\UseVerb{euler}, \protect\UseVerb{m1}, \protect\UseVerb{m2}, \protect\UseVerb{m3}, with \protect\UseVerb{exact} as benchmark solution, obtained with $10^3$ samples.}%
\label{fig:B0_varcds}%
\end{figure}
{The computational times, reported in Table \ref{tab:ctimes_B0_T1_d2_N1000_M1000}, are comparable with those of the non normalized case. The same can said about the accuracy results reported in Table \ref{tab:B0_var}, which are comparable with those in Table \ref{tab:B0_fix}.}

\appendix

%%% The Appendices part is started with the command \appendix;
%%% appendix sections are then done as normal sections
%\appendix
%\section{Derivatives of matrix exponentials}\label{sec:lem_exp}

\section*{Declarations}
\subsection*{Funding}
This project has received funding from the European Union’s Horizon 2020 research and innovation
programme under the Marie Sklodowska-Curie grant agreement No 813261 and is part of the ABC-EU-XVA project.
\subsection*{Conflicts of interests}

The authors have no relevant financial or non-financial interests to disclose.

\subsection*{Data availability}
All data generated or analysed during this study are included in this published article.
{%\color{red}
In particular the code to produce the numerical experiments is available at\\
\url{https://github.com/kevinkamm/StochasticMagnusExpansion}.
}

{\thispagestyle{scrheadings}
%\nocite{*}
\newpage
\thispagestyle{scrheadings}\ihead{}
\begin{footnotesize}
\bibliographystyle{acm}%{chicago}
\bibliography{bib}
\end{footnotesize}
}

\end{document}